\documentclass[12pt]{article}
\usepackage{verbatim}
\usepackage{appendix,latexsym,amsfonts,amsmath,amssymb,graphicx,hyperref,amsthm,authblk,comment,epstopdf}
\usepackage{graphicx}
\usepackage{pinlabel}
\usepackage{amssymb,amsmath,comment}
\usepackage{amsthm}
\usepackage{graphicx}
\usepackage{pinlabel}
\newtheorem{corollary}{Corollary}[section]
\newtheorem{lemma}[corollary]{Lemma}
\newtheorem{proposition}[corollary]{Proposition}
\newtheorem{theorem}[corollary]{Theorem}

\newcommand{\Z}{{\mathbb Z}}

\newcommand{\R}{{\mathbb{R}}}
\newcommand{\C}{{\mathbb C}}

\newcommand{\EE}{\mathbb{E}}
\newcommand{\PP}{\mathbb{P}}

\newcommand{\HH}{\mathbb{H}}
\newcommand{\N}{\mathbb{N}}
\newcommand{\D}{\mathbb{D}}

\newcommand{\Ree}{\mbox{Re}\,}

\newcommand{\pa}{\partial}

\newcommand{\F}{{\cal F}}
\newcommand{\no}{\noindent}

\newcommand{\BGE}{\begin{equation}}
\newcommand{\BGEN}{\begin{equation*}}
\newcommand{\EDE}{\end{equation}}
\newcommand{\EDEN}{\end{equation*}}

\newcommand{\Half}{{\mathbb H}}
\newcommand{\inrad}{{\rm{inrad}}}

\def\eps{\varepsilon}

\def\til{\widetilde}
\def\ha{\widehat}
\def\sem{\setminus}
\def\lin{\overline}

\numberwithin{equation}{section}
 
 \DeclareMathOperator{\diam}{diam}
\DeclareMathOperator{\dist}{dist} 
\DeclareMathOperator{\hcap}{hcap} \DeclareMathOperator{\id}{id}

  \DeclareMathOperator{\Imm}{Im}

 \DeclareMathOperator{\doub}{doub}

\def\Z{\mathbb{Z}}

\def \Im {{\rm Im}}

\def \Half {{\mathbb H}}
\newenvironment{remark}[1][Remark]{\begin{trivlist}
		\item[\hskip \labelsep {\bfseries #1}]}{\end{trivlist}}

\begin{document}
	
\title{Green's functions for chordal SLE curves}
\author{Mohammad A. Rezaei}
\author{Dapeng Zhan\thanks{Research partially supported by NSF grant  DMS-1056840 and Simons Foundation grant \#396973.}}
\affil{Michigan State University}
\renewcommand\Authands{ and }
\maketitle
	
\begin{abstract}
For a chordal SLE$_\kappa$ ($\kappa\in(0,8)$) curve in a domain $D$, the $n$-point Green's function valued at distinct points $z_1,\dots,z_n\in D$ is defined to be $$G(z_1,\dots,z_n)=\lim_{r_1,\dots,r_n\downarrow 0} \prod_{k=1}^n r_k^{d-2} \PP[\dist(\gamma,z_k)<r_k,1\le k\le n],$$ where $d=1+\frac{\kappa}{8}$ is the Hausdorff dimension of SLE$_\kappa$, provided that the limit converges. In this paper, we will show that such Green's functions exist for any finite number of points. Along the way we provide the rate of convergence and modulus of continuity for Green's functions as well. Finally, we give up-to-constant bounds for them.
\end{abstract}

\tableofcontents

\section{Introduction}  \label{Introsec}

The Schramm-Loewner evolution (SLE) is a measure on the space of curves which was defined in the groundbreaking work of Schramm \cite{Sch}. It is the main universal object emerging as the scaling limit of many models from statistical physics. Since then the geometry of SLE curves has been studied extensively. See \cite{RS,Law1} for definition and properties of SLE.

One of the most important functions associated to SLE (in general any random process) is the Green's function. Roughly, it can be defined as the normalized probability that SLE curve hits a set of $n\geq 1$ given points in its domain. See equation \eqref{mlti-green} for precise definition. For $n=1$, the existence of Green's function for chordal SLE was given in \cite{Law4} where conformal radius was used instead of Euclidean distance. % In the case $n=1$, also we have a closed-form formula  which was first given in \cite{RS}.
For $n=2$, the existence was proved in \cite{LW} (again for conformal radius instead of Euclidean distance) following a method initiated by Beffara \cite{Bf}. Finally in \cite{LR1} the authors showed that Green's function as defined here (using Euclidean distance) exists for $n=1,2$, and obtained an explicit formula of the one-point Green's function for chordal SLE in the upper half plane (see (\ref{G(z)})). To the best of our knowledge, existence of Green's function for $n>2$ has not been proved so far. Our main goal in this paper is to show that Green's function exists for all $n\geq2$. In addition we find convergence rate and modulus of continuity of the Green's functions, and provide sharp bounds for them.

Chordal SLE$_\kappa$ ($\kappa>0$) in a simply connected domain $D$ is a probability measure on curves in $\lin D$ from one marked boundary point (or prime end) $a$ to another marked boundary point (or prime end) $b$. It is first defined in the upper half plane $\HH=\{z\in\C:\Imm z>0\}$ using chordal Loewner equation, and then extended to other domains by conformal maps. For $\kappa\ge 8$, the curve is space filling (\cite{RS}), i.e., it visits every point in the domain. In this paper we only consider SLE$_\kappa$ for $\kappa\in(0,8)$ and fix $\kappa$ throughout. It is known (\cite{Bf}) that SLE$_\kappa$ has Hausdorff dimension $d=1+\frac\kappa 8$. Let $z_1,\dots,z_n \in D$ be $n$ distinct points. The $n$-point Green's function for SLE$_\kappa$ (in $D$ from $a$ to $b$) at $z_1,\dots,z_n$ is defined by
\BGE  G_{(D;a,b)}(z_1,\dots,z_n)  = \lim_{r_1,\dots,r_n \downarrow 0} \prod_{k=1}^n r_k^{d-2}  \, \PP\Big [\bigcap_{k=1}^n\{\dist(z_k,\gamma) \leq r_k\}\Big ], \label{mlti-green}\EDE
provided the limit exists. By conformal invariance of SLE, we easily see that the Green's function satisfies conformal covariance. That is, if $G_{(\HH;0,\infty)}$ exists, then $G_{(D;a,b)}$ exists for any triple $(D;a,b)$, and if $g$ is a conformal map from $(D;a,b)$ onto $(\HH;0,\infty)$, then
$$G_{(D;a,b)}(z_1,\dots,z_n)=\prod_{k=1}^n |g'(z_j)|^{2-d} G_{(\HH;0,\infty)}(g(z_1),\dots,g(z_n)).$$
Thus, it suffices to prove the existence of $G_{(\HH;0,\infty)}$, which we write as $G$. As we mentioned above, the one-point Green's function $G(z)$ has a closed-form formula (\cite{LR1}):
\BGE G(z)=\hat c (\Imm z)^{d-2+\alpha}|z|^{-\alpha}.\label{G(z)}\EDE
where $\alpha=\frac{8}{\kappa}-1$ is the boundary exponent, and $\hat c$ is a positive constant depending on $\kappa$, which is unknown so far.

Now we can state the main result of the paper.
\begin{theorem} \label{maintheorem}
For any $n\in\N$, $G(z_1,\dots,z_n)$  exists and is locally H\"older continuous. Also there is an explicit function $F(z_1,\dots,z_n)$  (defined in \eqref{F}) such that for any distinct points $z_1,\dots,z_n\in\HH$,
$G(z_1,\dots,z_n) \asymp F(z_1,\dots,z_n)$, where
the constant depends only on $\kappa$ and $n$.
\end{theorem}

We prove stronger results than Theorem \ref{maintheorem}. Specifically we provide a rate of convergence in the limit \eqref{mlti-green}. See Theorem \ref{main}. The function $F(z_1,\dots,z_n)$ appeared implicitly in \cite{RZ} and we define it explicitly here. The upper bound for Green's function (assuming existence of $G$) was proved in \cite[Theorem 1.1]{RZ} but the lower bound is new.

Our result will shed light on the study of some random lattice paths, e.g., loop-erased random walk (LERW), which are known to converge to SLE (\cite{LSW,LERW}). More specifically, combining the convergence rate of LERW to SLE$_2$ (\cite{rate}) with our convergence rate of the rescaled visiting probability to Green's function for SLE, one may get a good estimate on the probability that a number of small discs be visited by LERW.

We may also work on the Green's function when some points lie on the boundary. In order to have a non-trivial limit, the exponent $d-2$ in the definition (\ref{mlti-green}) for these points should be replaced by $-\alpha$. For $\kappa=8/3$, the existence of boundary Green's function for any $n$ follows from the restriction property (\cite{FW}). The existence and exact formulas of boundary Green's functions when $n=1,2$ were provided in \cite{Law3}.  In \cite{JJK} the authors found closed-form formulas of boundary Green's functions of up to $4$ points assuming their existence.
Since our upper bound (Proposition \ref{RZ-Thm1.1}) and lower bound (Theorem \ref{strong-lower}) are about the probability that SLE visits discs, where the centers are allowed to lie on the boundary, we immediately have sharp bounds of the boundary or mixed type Green's functions assuming their existence, which may be proved using the main technique here.

It is also interesting to study the Green's functions for other types of SLE such as radial SLE, SLE$_\kappa(\rho)$, or stopped SLE. In \cite{radial-Green}, the authors proved the existence of the conformal radius version of one-point Green's function for radial SLE.

The rest of the paper is organized as the following. In Section \ref{Prelim} we go over basic definitions and tools that we need from complex analysis and SLE theory. Then in Section \ref{mainestimates} we describe the main estimates that we need to show convergence, continuity and lower bound. One of them is a generalization of the main result in \cite{RZ} which quantifies the probability that SLE can go back and forth between a set of points, and its proof is postponed to the Appendix. In Section \ref{maintheorems} we state our main results, and then in Section \ref{proof1} we use estimates provided in Section \ref{mainestimates} to show existence and continuity of the Green's function. We prove the theorems by induction on the number of the points following a method initiated in \cite{LW}, which is to write the $n$-point Green’s function in terms of an expectation of $(n-1)$-point Green's function with respect to two-sided radial SLE.  Finally in Section \ref{proof2} we prove  sharp lower bounds for Green's functions, which match the upper bounds obtained in \cite{RZ}. % given in Theorem \ref{strong-lower}.
\smallbreak

\noindent\textbf{Acknowledgment.} The authors acknowledge Gregory Lawler, Brent Werness and Julien Dub\'edat for helpful discussions.
Dapeng Zhan's work is partially supported by a grant from NSF (DMS-1056840) and a grant from the Simons Foundation (\#396973).

\section{Preliminaries} \label{Prelim}

\subsection{Notation and Definitions}

We fix $\kappa \in (0,8)$ and set (Hausdorff dimension and boundary exponent)
\[
d=1+\frac{\kappa}{8}, \qquad \alpha=\frac{8}{\kappa}-1.
\]
Note that $d\in(0,2)$ and $\alpha>2-d$.
Throughout, a constant (such as $d$ or $\alpha$) depends only on $\kappa$ and a variable $n\in\N$ (number of points), unless otherwise specified. We write $X\lesssim Y$ or $Y\gtrsim X$ if there is a  constant $C>0$ such that $X\le CY$. We write $X\asymp Y$ if $X\lesssim Y$ and $X\gtrsim Y$. We write $X=O(Y)$ if there are two constants $\delta,C>0$ such that if $|Y|<\delta$, then $|X|\le C|Y|$. Note that this is slightly weaker than $|X|\lesssim |Y|$.

For $y\ge 0$ define $P_y$ on $[0,\infty)$ by
$$ P_y(x)=\left\{\begin{array}{ll} y^{\alpha-(2-d)}x^{2-d},&  x\le y;\\  x^\alpha,& x\ge y.
\end{array}\right. $$
we will frequently use the following lemmas without reference.

\begin{lemma}
	For $0\le x_1<x_2$, $0\le y_1\le y_2$, $0<x$, and $0\le y$, we have
	$$\frac{P_{y_1}(x_1)}{P_{y_1}(x_2)}\le \frac{P_{y_2}(x_1)}{P_{y_2}(x_2)};$$
	$$\Big(\frac{x_1}{x_2}\Big)^\alpha\le \frac{P_y(x_1)}{P_y(x_2)}\le \Big(\frac{x_1}{x_2}\Big)^{2-d}=\frac{P_{x_2}(x_1)}{P_{x_2}(x_2)};$$
	$$ \Big(\frac{y_1}{y_2}\Big)^{\alpha-(2-d)}\le \frac{P_{y_1}(x)}{P_{y_2}(x)}\le 1  .$$\label{Py}
\end{lemma}
\begin{proof}
	For the first formula, one may first prove that it holds in the following special cases: $y_1\le y_2\in[0,x_1]$; $y_1\le y_2\in[x_1,x_2]$; and $y_1\le y_2\in [x_2,\infty]$. The formula in the general case then easily follows. The second formula follows from the first by first setting $y_1=0$ and $y_2=y$ and then $y_1=y$ and $y_2=x_2\vee y$. The third formula can be proved by considering the following cases one by one: $x\in(0,y_1]$; $x\in[y_1,y_2]$; and $x\in[y_2,\infty)$.
\end{proof}

\begin{lemma}
	Let $z_1,\dots,z_n$ be distinct points in $\lin\HH$. Let $S$ be a nonempty set in $\C$ with positive distance from $\{z_1,\dots,z_n\}$.
	Then for any permutation $\sigma$ of $\{1,\dots,n\}$,
	\BGE \prod_{k=1}^n P_{\Imm z_{\sigma(k)}}(\dist(z_{\sigma(k)},S\cup\{z_{\sigma(j)}:j<k\}))\asymp  \prod_{k=1}^n P_{\Imm z_{k}}(\dist(z_k,S\cup\{z_j:j<k\})).\label{perm}\EDE
	% where the implicit constant depends only on $\kappa$ and $n$.
	\label{perm-lem}
\end{lemma}
\begin{proof}
	It suffices to prove the lemma for $\sigma=(k_0,k_0+1)$. In this case, the factors on the LHS of (\ref{perm}) for $k\ne k_0,k_0+1$ agree with the corresponding factors on the RHS of (\ref{perm}). So we only need to focus on the factors for $k=k_0,k_0+1$. Let $w_1=z_{k_0}$, $w_2=z_{k_0+1}$, $u_j=\Imm w_j$,  $L_j=\dist(w_j,S\cup\{z_k:k<k_0\})$, $j=1,2$. Then
	it suffices to show that
	\BGE P_{u_2}(L_2)P_{u_1}(L_1\wedge |w_1-w_2|) \asymp P_{u_1}(L_1) P_{u_2}(L_2\wedge |w_2-w_1|).\label{perm2}\EDE
	Let $r=|w_1-w_2|$. Note that $|u_1-u_2|,|L_1-L_2|\le r$. We consider several cases. First, suppose $L_1\le r$. Then $L_2\le 2r$, and we get $L_1\wedge r=L_1$ and $L_2/2\le L_2\wedge r\le L_2$. From the above lemma, we immediately get (\ref{perm2}). Second, suppose $L_2\le r$. This case is similar to the first case. Third, suppose $L_1,L_2\ge r$. In this case, $L_1\wedge r=L_2\wedge r=r$, and $L_1\asymp L_2$. Now we consider subcases. First, suppose $u_1\le r$. Then $u_2\le 2r$. If $u_2\le r$, by the definition, $\frac{P_{u_2}(L_2)}{P_{u_2}(r)} =(\frac{L_2}r)^\alpha$; if $r\le u_2\le 2r$, from the previous lemma, we get $\frac{P_{u_2}(L_2)}{P_{u_2}(r)}\asymp \frac{P_{r}(L_2)}{P_{r}(r)}=(\frac{L_2}r)^\alpha$. Since $u_1\le r$, we have $\frac{P_{u_1}(L_1)}{P_{u_1}(r)}=(\frac{L_1}r)^\alpha$. Since $L_1\asymp L_2$, we get (\ref{perm2}) in the first subcase. Second, suppose $u_2\le r$. This is similar to the first subcase. Third, suppose $u_1,u_2\ge r$. Then we get $\frac{P_{u_j}(L_j)}{P_{u_j}(r)}=(\frac{L_j}r)^{2-d}$, $j=1,2$. Using $L_1\asymp L_2$, we get (\ref{perm2}) in the last subcase.
\end{proof}
For (ordered) set of distinct points $z_1,\dots,z_n\in\lin\HH\sem\{0\}$, we let $z_0=0$ and define for $1\le k\le n$,
\BGE l_k=\min_{0\le j\le k-1}\{|z_k-z_j|\},\quad d_k=\min_{0\le j\le n,j\ne k} \{|z_k-z_j|\},\quad %$$\ha d_k=\max_{1\le j\le n}  \{|z_k-z_j|\},\quad 1\le k\le n.$$
y_k=\Imm z_k,\quad R_k=d_k\wedge y_k.\label{ldyR}\EDE
Also set \BGE Q=\max_{1\le k \le n}\frac{|z_k|}{d_k}\ge 1.\label{Q}\EDE
Note that we have
\[
R_k \leq d_k \leq l_k.
\]

For $r_1,\dots,r_n>0$, define
$$F(z_1,\dots,z_n;r_1,\dots,r_n)=\prod_{k=1}^n \frac{P_{y_k}(r_k)}{P_{y_k}(l_k)};$$
\BGE F(z_1,\dots,z_n)=\lim_{r_1,\dots,r_n\to 0^+} \prod_{k=1}^n r_k^{d-2} F(z_1,\dots,z_n;r_1,\dots,r_n) =\prod_{k=1}^n \frac{y_k^{\alpha-(2-d)}}{P_{y_k}(l_k)}.\label{F}\EDE
This is the function $F$ in Theorem \ref{maintheorem}. When it is clear from the context, we write $F$ for $F(z_1,\dots,z_n)$. From Lemma \ref{Py} we see that
\BGE F(z_1,\dots,z_n;r_1,\dots,r_n)\le F(z_1,\dots,z_n)  \prod_{k=1}^n r_k^{2-d},\quad \mbox{if } r_k\le l_k,1\le k\le n.\label{F<F}\EDE
Applying Lemma \ref{perm-lem} with $S=\{0\}$, we see that for any permutation $\sigma$ of $\{1,\dots,n\}$,
\BGE F(z_1,\dots,z_n;r_1,\dots,r_n)\asymp  F(z_{\sigma(1)},\dots,z_{\sigma(n)};r_{\sigma(1)},\dots,r_{\sigma(n)}),\label{Frasymp}\EDE
and
$$ F(z_1,\dots,z_n )\asymp  F(z_{\sigma(1)},\dots,z_{\sigma(n)} ).$$ %\label{Fasymp}\EDE

Let $D$ be a simply connected domain with two distinct prime ends $w_0$ and $w_\infty$. We define
\[
F_{(D;w_0,w_\infty)}(z_1,\dots,z_n)= \prod_{j=1}^n |g'(z_j)|^{2-d}\cdot F(g(z_1),\dots,g(z_n)),
\]
where $g$ is any conformal map from $(D;w_0,w_\infty)$ onto $(\HH;0,\infty)$. Although such $g$ is not unique, the value of $F_{(D;w_0,w_\infty)}$ does not depend on the choice of $g$. % If $z$ lies in a simply connected domain $D$, by $\crad_D(z)$ we mean conformal radius of the point $z$ in $D$.

Throughout, we use $\gamma$ to denote a (random) chordal Loewner curve, use $(U_t)$ to denote its driving function, and $(g_t)$ and $(K_t)$ the chordal Loewner maps and hulls driven by $（U_t)$. This means that $\gamma$ is a continuous curve in $\lin\HH$ starting from a point on $\R$; for each $t$, $H_t:=\HH\sem K_t$ is the unbounded component of $\HH\sem\gamma[0,t]$, whose boundary contains $\gamma(t)$; and $g_t$ is a conformal map from $(H_t;\gamma(t),\infty)$ onto $(\HH;0,\infty)$ that solves the chordal Loewner equation
\BGE \pa_t g_t(z)=\frac 2{g_t(z)-U_t},\quad g_0(z)=z.\label{Loewner}\EDE
Let $Z_t=g_t-U_t$ denote the centered Loewner map, which is a conformal map from $(H_t;\gamma(t),\infty)$ onto $(\HH;0,\infty)$. See \cite{Law1} for more on Loewner curves.
% Let $(\F_t)$ be the filtration generated by $\gamma$.

When $\gamma$ is fixed, for any set $S$, $\tau_S$ is used to denote the infimum of the times that $\gamma$ visits $S$, and is set to be $\infty$ if such times do not exist. We write $\tau^{z_0}_r$ for $\tau_{\{|z-z_0|\le r\}}$, and $T_{z_0}$ for $\tau^{z_0}_0=\tau_{\{z_0\}}$. So another way to say that $\dist(\gamma,z_0)\le r$ is $\tau^{z_0}_r<\infty$.

Let $\PP$ denote the law of a chordal SLE$_\kappa$ curve in $\HH$ from $0$ to $\infty$, and $\EE$ the corresponding expectation. Then $\PP$ is a probability measure on the space of chordal Loewner curves such that the driving function $(U_t)$ has the law of $\sqrt\kappa$ times a standard Brownian motion. In fact, chordal SLE$_\kappa$ is defined by solving (\ref{Loewner}) with $U_t=\sqrt\kappa B_t$.

As we mentioned the upper bound in Theorem \ref{maintheorem} is not new. We now state \cite[Theorem 1.1]{RZ} using the notation just defined.

\begin{proposition}
	Let $z_1,\dots,z_n$ be distinct points in $\lin\HH\sem\{0\}$. Let $d_1,\dots,d_n$ be defined by (\ref{ldyR}). Let $r_j\in(0,d_j)$, $1\le j\le n$. Then we have
	$$\PP[\tau^{z_j}_{r_j}<\infty,1\le j\le n]\lesssim F(z_1,\dots,z_n;r_1,\dots,r_n).$$
	\label{RZ-Thm1.1}
\end{proposition}

\subsection{Lemmas on $\HH$-hulls}
We will need some results on $\HH$-hulls.  A relatively closed bounded subset $K$ of $\HH$ is called an $\HH$-hull if $\HH\sem K$ is simply connected. Given an $\HH$-hull $K$, we use $g_K$ to denote the unique conformal map from $\HH\sem K$ onto $\HH$ that satisfies $g_K(z)=z+O(|z|^{-1})$ as $z\to \infty$. The half-plane capacity of $K$ is $\hcap(K):=\lim_{z\to\infty} z(g_K(z)-z)$. Let $f_K=g_K^{-1}$. If $K=\emptyset$, then $g_K=f_K=\id$, and $\hcap(K)=0$. Now suppose $K\ne\emptyset$. Let $a_K=\min(\lin K\cap\R)$ and $b_K=\max(\lin K\cap\R)$. Let $K^{\doub}=K\cup[a_K,b_K]\cup\{\lin z: z\in K\}$. By Schwarz reflection principle, $g_K$ extends to a conformal map from $\C\sem K^{\doub}$ onto $\C\sem[c_K,d_K]$ for some $c_K<d_K\in\R$, and satisfies $g_K(\lin z)=\lin{g_K(z)}$. In this paper, we write $S_K$ for $[c_K,d_K]$.
\vskip 3mm
\no{\bf Examples}
\begin{itemize}
\item For $x_0\in\R$ and $r>0$, let $ \lin{D}^+_{x_0,r}$ denote semi-disc $\{z\in\HH:|z-x_0|\le r\}$, which is an $\HH$-hull. It is straightforward to check that $g_{ \lin{D}^+_{x_0,r}}(z)=z+\frac{r^2}{z-x_0}$, $\hcap( \lin{D}^+_{x_0,r})=r^2$, and $S_{ \lin{D}^+_{x_0,r}}=[x_0-2r,x_0+2r]$.
\item Each $K_t$ associated with a chordal Loewner curve $\gamma$ is an $\HH$-hull with $\hcap(K_t)=2t$. Since $\gamma(t)\in \pa K_t$ and $g_t(\gamma(t))=U_t$, we have  $U_t\in S_{K_t}$.
\end{itemize}

\begin{lemma}
	For any nonempty $\HH$-hull $K$, there is a positive measure $\mu_K$ supported by $S_K$ with total mass $|\mu_K|=\hcap(K)$ such that,
	\BGE f_K(z)-z=\int \frac{-1}{z-x}d\mu_K(x),\quad z\in\C\sem S_K.\label{f-z}\EDE
\end{lemma}\label{integral}
\begin{proof}
	This is \cite[Formula (5.1)]{LERW}.
\end{proof}

\begin{lemma}
	If a nonempty $\HH$-hull $K$ is contained in $ \lin{D}^+_{x_0,r}$ for some $x_0\in\R$ and $r>0$, then $\hcap(K)\le r^2$, $S_K\subset[x_0-2r,x_0+2r]$, and   \BGE |g_K(z)-z|\le 3r, \quad z\in\C\sem K^{\doub}.\label{f-z0}\EDE
\label{small}
\end{lemma}
\begin{proof}
	From the monotone property of $\hcap$ (\cite{Law1}), we have $\hcap(K)\le \hcap( \lin{D}^+_{x_0,r})=r^2$. From  \cite[Lemma 5.3]{LERW}, we know that $S_K\subset S_{ \lin{D}^+_{x_0,r}}=[x_0-2r,x_0+2r]$. Formula (\ref{f-z0}) follows from  \cite[Formula (3.12)]{Law1} and that $g_{K-x_0}(z-x_0)=g_K(z)-x_0$.
\end{proof}

\begin{lemma}
	Let $K$ be as in the above lemma. Then for any $z\in\C$ with $|z-x_0|\ge 5r$, we have
	\BGE |g_K(z)-z|\le 2|z-x_0|\Big(\frac{r}{|z-x_0|}\Big)^2 ;\label{1}\EDE
	\BGE \frac{|\Imm g_K(z)-\Imm z|}{|\Imm z|}\le 4\Big(\frac{r}{|z-x_0|}\Big)^2 ;\label{2}\EDE
	\BGE |g_K'(z)-1|\le 5\Big(\frac{r}{|z-x_0|}\Big)^2 .\label{3}\EDE
	%  Moreover, for any $z_1,z_2\in\C$ with $|z_j-x_0|\ge 15r$, $j=1,2$, we have
	%  \BGE |(g_K(z_1)-g_K(z_2))-(z_1-z_2)|\le 20|z_1-z_2|\Big(\frac{r}{|z_1-x_0|\wedge |z_2-x_0|}\Big)^2 \label{4}\EDE
	\label{lemma-gK}
\end{lemma}
\begin{proof}
	Since $g_{K-x_0}(z-x_0)=g_K(z)-x_0$, we may assume that $x_0=0$. From the above two lemmas, we find that $|\mu_K|\le r^2$ and
	\BGE f_K(w)-w=\int_{-2r}^{2r} \frac{-1}{z-w}d\mu_K(w),\quad w\in\C\sem[-2r,2r].\label{integralw}\EDE
	Thus, if $|w|>2r$, then $|f_K(w)-w|\le \frac{r^2}{|w|-2r}$. So $f_K$ maps the circle $\{|z|=4r\}$ onto a Jordan curve that lies within the circles $\{|z|=3.5r\}$ and $\{|z|=4.5r\}$. Thus, if $|z|>5r$, then $|g_K(z)|>4r$, and $|z-g_K(z)|=|f(g_K(z))-g_K(z)|\le  \frac{r^2}{|g_K(z)|-2r}\le r/2 $, which implies $|z|\le |g_K(z)|+r/2$, and $| g_K(z)-z|\le  \frac{r^2}{|g_K(z)|-2r}\le  \frac{r^2}{|z|-2.5r}\le \frac{r^2}{|z|/2}$.  So we get (\ref{1}).
	
	Taking the imaginary part of (\ref{integralw}), we find that, if $w\in\HH$ and $|w|>2r$, then $|\Imm f_K(w)-\Imm w|\le |\Imm w| \frac{r^2}{(|w|-2r)^2}$. Letting $w=g_K(z)$ with $z\in\HH$ and $|z|>5r$, we find that
	$$|\Imm z-\Imm g_K(z)|\le |\Imm g_K(z)|  \frac{r^2}{(|g_K(z)|-2r)^2}\le |\Imm z|\frac{r^2}{(|z|-2.5r)^2}\le |\Imm z|\frac{r^2}{(|z|/2)^2},$$
	which implies (\ref{2}). Here we used that $|\Imm g_K(z)|\le |\Imm z|$ that can be seen from (\ref{integralw}).
	
	Differentiating (\ref{integralw}) w.r.t.\ $z$, we find that, if  $|w|>2r$, then $|f_K'(w)-1|\le \frac{r^2}{(|w|-2r)^2}$.  Letting $w=g_K(z)$ with $z\in\HH$ and $|z|>5r$, we find that
	$$|1/g_K'(z)-1|\le \frac{r^2}{(|g_K(z)|-2r)^2}\le \frac{r^2}{(|z|-2.5r)^2}\le \frac{r^2}{(|z|/2)^2},$$
	which then implies (\ref{3}).
	%
	%  Let $l=|z_1|\wedge |z_2|$. From (\ref{1}) we get
	%  $$|(g_K(z_1)-z_1)|+|g_K(z_2)-z_2||\le 2(|z_1|+|z_2|)(\frac rl)^2.$$
	%  Thus, if $|z_1-z_2|\ge (|z_1|+|z_2|)/10$, then (\ref{4}) holds. Now suppose $|z_1-z_2|\le (|z_1|+|z_2|)/10$. WLOG, assume that $|z_1|\ge |z_2|$. Then $|z_1-z_2|\le \frac 15|z_1|$. Thus, for any $z\in[z_1,z_2]$, $|z|\ge \frac 45 |z_1|\ge \frac 45 l\ge 12$. From (\ref{3}) we get
	%  $$|(g_K(z_1)-g_K(z_2))-(z_1-z_2)|=\Big|\int_{[z_1,z_2]} g_K'(z)-1\Big|\le |z_1-z_2| 8\Big(\frac{r}{|z|}\Big)^2\le 8\Big(\frac{r}{\frac 45 l}\Big)^2,$$
	%  which also implies (\ref{4}).
\end{proof}

\begin{lemma}
	Let $K$ be a nonempty $\HH$-hull. Suppose $z\in\HH$ satisfies that $\dist(z,S_K)\ge 4\diam(S_K)$. Then $\dist(f_K(z),K)\ge 2 \diam(K)$. \label{SK-lemma}
\end{lemma}
\begin{proof}
	Let $r=\diam(S_K)$. Since $g_K$ maps $\C\sem K^{\doub}$ conformally onto $\C\sem S_K$, fixes $\infty$, and satisfies that $g_K'(\infty)=1$, we see that $K^{\doub}$ and $S_K$ have the same whole-plane capacity. Thus, $\diam(K)\le \diam(K^{\doub})\le \diam(S_K)$. Take any $x_0\in\lin K\cap \R$. Then $K\subset  \lin{D}^+_{x_0,r}$. So $|\mu_K|=\hcap(K)\le r^2$. Since $\dist(z,S_K)\ge 4r$, from (\ref{f-z}) we get $|f_K(z)-z|\le r/4$.
	From \cite[Lemma 5.2]{LERW}, we know $x_0\in [a_K,b_K]\subset [c_K,d_K]=S_K$. Thus, $\dist(f_K(z),K)\ge |f_K(z)-x_0|-r\ge |z-x_0|-|f_K(z)-z|-r\ge \dist(z,S_K)-2r>2r\ge 2\diam(K)$.
\end{proof}

\begin{lemma}
	Let $K$ be an $\HH$-hull, and $w_0$ be a prime end of $\HH\sem K$ that sits on $\pa K$. Let $z_0\in\HH\sem K$ and $R=\dist(z_0,K)>0$. Let $g$ be any conformal map from $\HH\sem K$ onto $\HH$ that fixes $\infty$ and sends $w_0$ to $0$. Then for $z_1\in\HH\sem K$, we have
	\BGE \frac{|g(z_1)-g(z_0)|}{|g(z_0)|}=O\Big(\frac{|z_1-z_0|}R\Big);\label{g-g/g-U}\EDE
	\BGE \frac{|\Imm g(z_1)-\Imm g(z_0)|}{\Imm g(z_0)}=O\Big( \frac{|\Imm z_1-\Imm z_0|}{\Imm z_0}\Big)+O\Big(\frac{|z_1-z_0|}R\Big)^{1/2}.\label{Im/Im}\EDE \label{gTvar}
\end{lemma}
\begin{proof} By scaling invariance, we may assume that $g=g_K-x_0$, where $x_0=g_K(w_0)\in[c_K,d_K]$. From Koebe's $1/4$ theorem, we know that $$|g(z_0)|=|g_K(z_0)-x_0|\ge \dist(g_K(z_0),[c_K,d_K])\gtrsim {|g'(z_0)|R}.$$
	Applying Koebe's distortion theorem and Cauchy's estimate, we find that, if $|z_1-z_0|<R/5$, then
	\BGE |g'(z_1)-g'(z_0)|\lesssim |g'(z_0)|\frac{|z_1-z_0|}R.\label{g'z1}\EDE
	\BGE |g'(z_1)|\asymp |g'(z_0)|,\quad |g(z_1)-g(z_0)|\lesssim |g'(z_0)||z_1-z_0|.\label{g-g}\EDE
	Combining the second formula with the lower bound of $|g(z_0)|$, we get (\ref{g-g/g-U}).
	%Koebe's distortion theorem also tells us that, if $|z_1-z_2|<R/5$, then
	%\BGE |g'(z_1)-g'(z_0)|\lesssim |g'(z_0)|.\label{g'z1}\EDE
	
	To derive (\ref{Im/Im}), we assume $\frac{|\Imm z_1-\Imm z_0|}{\Imm z_0}$ and $\frac{|z_1-z_0|}R$ are sufficiently small,  and consider several cases. First, assume that $\Imm z_0\ge \frac RC$ for some big constant $C$. From Koebe's $1/4$ theorem, we know that $\Imm g(z_0)\gtrsim {|g'(z_0)|R}$. This together with the inequalities $|\Imm g(z_1)-\Imm g(z_0)|\le |g(z_1)-g(z_0)|$ and (\ref{g-g}) implies (\ref{Im/Im}).
	
	Now assume that $\Imm z_0\le \frac RC$. Note that $z_0-\lin{z_0}=2i\Imm z_0$ and $g(z_0)-g(\lin{z_0})=2i\Imm g(z_0)$. From Koebe's distortion theorem, we see that when $C$ is big enough, \BGE |\Imm g(z_0)-g'(z_0)\Imm z_0|\lesssim |g'(z_0)|\Imm z_0\frac{\Imm z_0}R,\label{ImgT}\EDE
	which implies that
	\BGE \Imm g(z_0)\gtrsim |g'(z_0)|\Imm z_0.\label{lower-ImgT}\EDE
	Now we assume that $\Imm z_0\ge \sqrt{R|z_1-z_0|}$.
	Combining (\ref{lower-ImgT}) with (\ref{g-g}) and the inequalities $|\Imm g(z_1)-\Imm g(z_0)|\le |g(z_1)-g(z_0)|$ and $\frac{|z_1-z_0|}{\Imm z_0}\le (\frac{|z_1-z_0|}R)^{1/2}$, we get (\ref{Im/Im}).
	
	Finally, we assume that $\Imm z_0\le \sqrt{R|z_1-z_0|}$.
	Let $R_1=R-|z_1-z_0|\gtrsim R$. Then $\{|z-z_1|<R_1\}\subset \{|z-z_0|<R\}$.
	From Koebe's distortion theorem and (\ref{g'z1}), we get
	\BGE |\Imm g(z_1)-g'(z_1)\Imm z_1|\lesssim |g'(z_1)|\Imm z_1\frac{\Imm z_1}{R_1}\lesssim |g'(z_0)|\Imm z_0\frac{\Imm z_0}R .\label{ImgT1}\EDE
	Now we have
	\begin{align*}
	|\Imm g(z_1)-\Imm g(z_0)|\le& |\Imm g(z_0)-g'(z_0)\Imm z_0|+|\Imm g(z_1)-g'(z_1)\Imm z_1|\\
	+&|g'(z_1)-g'(z_0)|\Imm z_0+|g'(z_1)||\Imm z_1-\Imm z_0|.
	\end{align*}
	Combining the above inequality with the inequalities (\ref{g'z1}-\ref{ImgT1}) and $\frac {\Imm z_0}R\le (\frac{|z_1-z_0|}R)^{1/2}$, we get  (\ref{Im/Im}) in the last case.
\end{proof}

\subsection{Lemmas on extremal length}
We will need some lemmas on extremal length, which is a nonnegative quantity $\lambda(\Gamma)$ associated with a family $\Gamma$ of rectifiable curves (\cite[Definition 4-1]{Ahl}). One remarkable property of extremal length is its conformal invariance (\cite[Section 4-1]{Ahl}), i.e., if every $\gamma\in\Gamma$ is contained in a domain $\Omega$, and $f$ is a conformal map defined on $\Omega$, then $\lambda(f(\Gamma))=\lambda(\Gamma)$.
We use $d_\Omega(X,Y)$ to denote the extremal distance between $X$ and $Y$ in $\Omega$, i.e., the extremal length of the family of curves in $\Omega$ that connect $X$ with $Y$. It is known that in the special case when $\Omega$ is an annulus with radii $R_1<R_2$, and $X$ and $Y$ are the two boundary components of $\Omega$, $d_\Omega(X,Y)=\log(R_2/R_1)/(2\pi)$ (\cite[Section 4-2]{Ahl}). We will use the comparison principle (\cite[Theorem 4-1]{Ahl}): if every $\gamma\in\Gamma$ contains a $\gamma'\in\Gamma'$, then $\lambda(\Gamma)\ge \lambda(\Gamma')$. Thus, if every curve in $\Omega$ connecting $X$ with $Y$  intersects a pair of concentric circles with radii $R_2>R_1$, then $d_\Omega(X,Y)\ge \log(R_2/R_1)/(2\pi)$. We will also use the composition law (\cite[Theorem 4-2]{Ahl}): if for $j=1,2$, every $\gamma_j$ in a family $\Gamma_j$ is contained in $\Omega_j$, where $\Omega_1$ and $\Omega_2$ are disjoint open sets, and if every $\gamma$ in another family $\Gamma$ contains a $\gamma_1\in\Gamma_1$ and a $\gamma_2\in\Gamma_2$, then $\lambda(\Gamma)\ge \lambda(\Gamma_1)+\lambda(\Gamma_2)$. In addition, we need the following lemma.

\begin{comment}
\begin{lemma}
For any $r>0$, we have $\frac r{1+r}\le 27 e^{-\pi d_{\HH}([-r,0],[1,\infty))}$. \label{lem-extremal}
\end{lemma}
\begin{proof}
Let $g(z)=z+\frac 1z$ be the conformal map from $\C\sem \lin\D$ onto $\C\sem [-2,2]$, and $f=g^{-1}$. From the conformal invariance and comparison principle of extremal distance (\cite{Ahl}), we get
$$d_{\HH}([-r,0],[1,\infty))=d_{\HH}([-2,2],[4/r+2,\infty))$$
$$=d_{\HH}(\{|z|=1\} ,[f(4/r+2),\infty))\le d_{\HH}(\{|z|=1\} ,[4/r+2,\infty))$$
$$=d_{\HH}([-1,0],\{|z|=4/r+2\})=d_{\HH}([-2,2],\{|z-2|=16/r+8\})$$
$$=d_{\HH}(\{|z|=1\},f(\{|z-2|=16/r+8\})).$$
Here the ``$\le$'' follows from that $f(x)\le x$ for $x>2$; the third ``$=$'' follows from applying the conformal automorphism $z\mapsto -(4/r+2)/z$ of $\HH$; and the first and the last ``$=$'' follow from applying the conformal map $f$.

From $|g(z)-z|=1/|z|\le |z|$, we get $|g(z)|\le 2|z|$, and so
$$|f(z)-z|=|f(z)-g(f(z))|\le 1/|f(z)|\le 2/|g(f(z))|=2/|z|.$$
Assume $r\le 1$. Then $f(\{|z-2|=16/r+8\})\subset \{|z|<16/r+10+r/8\}\subset\{|z|<27/r\}$. This implies that  $f(\{|z-2|=16/r+8\})$ disconnects $\{|z|=1\}$ from $\{|z|=27/r\}$. So
$$d_{\HH}(\{|z|=1\},f(\{|z-2|=16/r+8\}))\le d_{\HH}(\{|z|=1\}, \{|z|=27/r\})=\log(27/r)/\pi,$$
which implies that $27 e^{-\pi d_{\HH}([-r,0],[1,\infty))}\ge r>\frac r{1+r}$ for $r\le 1$. If $r\ge 1$, then $27 e^{-\pi d_{\HH}([-r,0],[1,\infty))}\ge 27 e^{-\pi d_{\HH}([-1,0],[1,\infty))}\ge 1>\frac r{1+r}$. The proof is now completed.
\end{proof}
\end{comment}

\begin{lemma}
	Let $S_1$ and $S_2$ be a disjoint pair of connected bounded closed subsets of $\lin\HH$ that intersect $\R$. Then
	$$\prod_{j=1}^2 \Big(\frac{\diam(S_j)}{\dist(S_1 ,S_2 )}\wedge 1\Big)\le 144 e^{-\pi d_{\HH}(S_1,S_2)}.$$ \label{lem-extremal2}
\end{lemma}
\begin{proof} For $j=1,2$, let $S_j^{\doub}$ be the union of $S_j$ and its reflection about $\R$. By reflection principle (\cite[Exercise 4-1]{Ahl}), $d_{\HH}(S_1,S_2)=2d_{\C}(S_1^{\doub},S_2^{\doub})$.  Choose $z_j\in S_j $, $j=1,2$, such that $|z_2-z_1|=d_S:=\dist(S_1 ,S_2 )$.  Let $r_j=\max_{z\in S_j^{\doub}} |z-z_j|$, $j=1,2$. From Teichm\"uller Theorem (\cite[Theorem 4-7]{Ahl}) and conformal invariance of extremal distance (\cite{Ahl}), we find that
	$$d_{\C}(S_1^{\doub},S_2^{\doub})\le d_{\C}([-r_1,0],[d_S,d_S+r_2]) =d_{\C}([-1,0],[R,\infty))=\Lambda(R),$$
	where $R>0$ satisfies that $\frac 1{1+R}=\prod_{j=1}^2 \frac{r_j}{d_S+r_j}$, and $\Lambda(R)$ is the modulus of the Teichm\"uller domain $\C\sem([-1,0],[R,\infty))$. From \cite[Formula (4-21)]{Ahl} and the above computation, we get
	$$e^{-\pi d_{\HH}(S_1,S_2)}=e^{-2\pi \Lambda(R)}\ge \frac{1}{16(R+1)}=\frac 1{16} \prod_{j=1}^2 \frac{r_j}{d_S+r_j}.$$
	Since $\diam(S_j)\le 2r_j$ and $\frac{2r_j}{d_S}\wedge 1\le \frac{3r_j}{d_S+r_j}$, the proof is now complete.
\end{proof}

\begin{remark}
The lower bound of Lemma \ref{lem-extremal2} also holds (with a different constant), and the proof does not need Teichm\"uller Theorem. But it is not needed for our  purposes.
\end{remark}

\begin{comment}
\begin{lemma}
Let $S_1$ and $S_2$ be a disjoint pair of connected bounded closed subsets of $\lin\HH$ that intersect $\R$. Suppose that $S_2$ disconnects $S_1$ from $\infty$ in $\HH$. Then
$$\frac{\diam(S_1)}{\dist(S_1\cap\R,S_2\cap\R)}\wedge 1\le 108 e^{-\pi d_{\HH}(S_1,S_2)}.$$
\end{lemma}
\begin{proof}
Let $x_j$ and $r_j$, $j=1,2$, be as in the above proof. From the comparison principle and conformal invariance of extremal distance (\cite{Ahl}), we know that
\BGE d_{\HH}(S_1,S_2)\le  d_{\HH}([x_1-r_1,x_1],[x_2,\infty])=d_{\HH}([-r,0],[1,\infty]),\label{compareS2}\EDE
if $r>0$ satisfies $\frac{r}{1+r}= \frac{r_1}{x_2-x_1+r_1}$. Since $\diam(S_1)\le 2r_1$, we get $\frac{\diam(S_1)}{x_2-x_1}\wedge 1\le  \frac{4r_1}{x_2-x_1+r_1}$. Thus,
$$\frac{\diam(S_1)}{\dist(S_1\cap\R,S_2\cap\R)}\wedge 1\le  \frac{4r_1}{x_2-x_1+r_1}\le 108 e^{-\pi d_{\HH}([-r,0],[1,\infty))}\le 108 e^{-\pi d_{\HH}(S_1,S_2)},$$
where the second ``$\le$'' follows from Lemma \ref{lem-extremal}, and the last ``$\le$'' follows from (\ref{compareS2}).
\end{proof}
\end{comment}

\subsection{Lemmas on two-sided radial SLE}
For $z\in\HH$, and $r>0$, we use $\PP^r_z$ to denote the conditional law $\PP[\cdot|\tau^z_r<\infty]$, and use $\PP^*_z$ to denote the law of a two-sided radial SLE$_\kappa$ curve through $z$. For $z\in \R\sem\{0\}$, we use $\PP^*_z$ to denote the law of a two-sided chordal SLE$_\kappa$ curve through $z$. Let $\EE^r_z$ and $\EE^*_z$ denote the corresponding expectation. In any case, we have $\PP^*_z$-a.s., $T_z<\infty$. See \cite{LW,LZ} for definitions and more details on these measures. For a random chordal Loewner curve $\gamma$, we use $(\F_t)$ to denote the filtration generated by $\gamma$.

\begin{lemma}
	Let $z\in\HH$ and $R\in(0,|z|)$. Then $\PP_z^*$ is absolutely continuous w.r.t.\ $\PP_z^R$ on $\F_{\tau^z_R}\cap\{\tau^z_R<\infty\}$, and the Radon-Nikodym derivative is uniformly bounded. \label{RN<1}
\end{lemma}
\begin{proof}
	It is known (\cite{LW,LZ}) that $\PP_z^*$ is obtained by weighting $\PP$ using $M^z_t/G(z)$, where $M^z_t=|g_t'(z)|^{2-d}G(Z_t(z))$ and $G(z)$ is given by (\ref{G(z)}). Since $\PP_z^R$ is obtained by weighting the restriction of $\PP$ to $\{\tau^z_R<\infty\}$ using $1/\PP[\tau^z_R<\infty]$, it suffices to prove that $\frac{M^z_\tau}{G(z)} \cdot\PP[\tau<\infty]$ is uniformly bounded, where $\tau=\tau^z_R$.
	
	Let $y=\Imm z$.   From \cite[Lemma 2.6]{RZ} we have $\PP[\tau<\infty]\lesssim \frac{P_{y}(R)}{P_{y}(|z|)}$. Let $\til z=g_\tau(z)$ and $\til y=\Imm \til z$. It suffices to show that
	\BGE \frac{|\til z|^{-\alpha} \til y^{\alpha-(2-d)}}{|z|^{-\alpha}y^{\alpha-(2-d)}}\cdot |g_\tau'(z)|^{2-d}\cdot  \frac{P_{y}(R)}{P_{y}(|z|)}\lesssim 1.\label{<1}\EDE
	We consider two cases. First, suppose $y\ge R/10$. From Lemma \ref{Py}, we get $\frac{P_{y}(R)}{P_{y}(|z|)}\lesssim (\frac{y}{|z|})^\alpha (\frac Ry)^{2-d}$. Applying Koebe's $1/4$ theorem, we get $\til y\gtrsim |g_\tau'(z)| R$. Thus,
	$$\mbox{LHS of }(\ref{<1})\lesssim
	\frac{(y/|\til z|)^{\alpha} (|g_\tau'(z)| R)^{-(2-d)}}{|z|^{-\alpha}y^{\alpha-(2-d)}}\cdot |g_\tau'(z)|^{2-d}\cdot  \Big(\frac{y}{|z|}\Big)^\alpha \Big(\frac Ry\Big)^{2-d}=\Big(\frac{\til y}{|\til z|}\Big)^\alpha\le 1.$$
	So we get (\ref{<1}) in the first case. Second, assume that $y\le R/10$. Then we have $\frac{P_{y}(R)}{P_{y}(|z|)}= (\frac{R}{|z|})^\alpha$. Applying Koebe's distortion theorem, we get $\til y\asymp  |g_\tau'(z)| y$. Applying Koebe's $1/4$ theorem, we get $|\til z|\gtrsim |g_\tau'(z)| R$. Thus,
	$$\mbox{LHS of }(\ref{<1})\lesssim \frac{(|g_\tau'(z)| R)^{-\alpha} (|g_\tau'(z)| y)^{\alpha-(2-d)}}{|z|^{-\alpha} y^{\alpha-(2-d)}}\cdot |g_\tau'(z)|^{2-d}\cdot  \Big(\frac{R}{|z|}\Big)^\alpha.$$
	So we get (\ref{<1}) in the second case. The proof is now complete.
\end{proof}

\begin{lemma}
	Let $z\in \HH$ and $R\in(0,|z|)$. Then for any $w\in\HH$ such that $\frac{|w-z|}R$ is sufficiently small, $\PP_z^*$ and $\PP_w^*$ restricted to $\F_{\tau^z_R}$ are absolutely continuous w.r.t.\ each other, and
	$$ \log\Big(\frac{d\PP_w^*|_{\F_{\tau^z_R}}}{d\PP_z^*|_{\F_{\tau^z_R}}}\Big) =O\Big(\frac{|z-w|}R\Big).$$\label{continuity-two-sided}
\end{lemma}
\begin{proof}
	Let $G$ and $M^\cdot_t$ be as in the above proof.  Let $\tau=\tau^z_R$. It suffices to show that
	$$  \log\Big(\frac{M^z_\tau}{G(z)}\Big /\frac{M^w_\tau}{G(w)}\Big)=O\Big(  \frac{|z-w|}R \Big).$$ %\label{M/M}\EDE
	Since $||z|-|w||\le |z-w|$ and $|z|\ge R$, we get $\log\frac{|w|}{|z|}| =O(\frac{|z-w|}R)$.
	Let $\til z=g_\tau(z)-U_\tau$ and $\til w=g_\tau(w)-U_\tau$. From Koebe's $1/4$ theorem and distortion theorem, we get $|\til z|\gtrsim |g_\tau'(z)|R$ and $|\til z-\til w|\lesssim |g_\tau'(z)||z-w|$. So we get $\log\frac{|\til w|}{|\til z|}=O(\frac{|z-w|}R)$.
	From Koebe's distortion theorem, we get $\log\frac{|g_\tau'(w)|}{|g_\tau'(z)|}=O( \frac{|z-w|}R)$.  So it suffices to show that
	\BGE \log\Big(\frac{\Imm \til w}{\Imm w}\Big /\frac{\Imm \til z}{\Imm z}\Big)=O\Big( \frac{|z-w|}R\Big) .\label{M/M}\EDE
	
	Now we consider two cases. First, suppose that $\Imm z\ge R/8$. Since $|\Imm w-\Imm z| \le |w-z|$ we get $\log \frac{\Imm w}{\Imm z}=O(\frac{|z-w|}R)$.  Applying Koebe's $1/4$ theorem, we get $\Imm \til z\gtrsim |g_\tau'(z)|R$. Since $|\Imm \til w-\Imm \til z|\le |\til w-\til z|\lesssim |g_\tau'(z)| |z-w|$, from the above argument, we get $\log\frac{\Imm \til w}{\Imm \til z}=O(\frac{|z-w|}R)$, which implies (\ref{M/M}). Second, suppose that $\Imm z\le R/8$. Then $\Imm w<R/4$ if $|z-w|<R/8$. Applying Koebe's distortion theorem, we get $\log(\frac{\Imm \til z}{|g_\tau'(z)|\Imm z}), \log(\frac{\Imm \til w}{|g_\tau'(w)|\Imm w}) =O(\frac{|z-w|}R)$, which together with $ \log\frac{|g_\tau'(w)|}{|g_\tau'(z)|}=O(\frac{|z-w|}R)$ imply (\ref{M/M}) in the second case.
\end{proof}

\begin{remark}
The above two lemmas still hold if $z$ or $w$ lies on $\R\sem\{0\}$, and the two-sided radial measure is replaced by the two-sided chordal measure.
\end{remark}

\section{Main Estimates} \label{mainestimates}
In this section, we will provide some useful estimates for the proofs of the main theorems. As before,   $\gamma$ denotes a chordal Loewner curve; when $\gamma$ is fixed in the context, for each $t$ in the domain of $\gamma$,  $H_t$   denotes the unbounded domain of $\HH\sem \gamma[0,t]$; $\PP$ denotes the law of a chordal SLE$_\kappa$ curve in $\HH$ from $0$ to $\infty$. For $z_0\in\HH$, and $r>0$, $\tau^{z_0}_r$ denotes the first time that the relative curve hits the circle $\{|z-z_0|=r\}$; $\PP^r_{z_0}$ denotes the conditional law $\PP[\cdot|\tau^{z_0}_r<\infty]$; and $\PP^*_{z_0}$ denotes the law of a two-sided radial SLE$_\kappa$ curve in $\HH$ from $0$ to $\infty$ passing through $z_0$. A crosscut in a domain $D$ is an open simple curve in $D$, whose two ends approach to two boundary points of $D$.

\begin{theorem}
	Let $z_1,\dots,z_{n}$ be distinct points in $\lin\HH\sem\{0\}$, where $n\ge 2$. Let $r_j\in(0,d_j/8)$, $1\le j\le n$. %Let $k_0\in\{2,\dots,n\}$ and $s_{k_0}\in(r_{k_0},|z_{k_0}-z_1|\wedge |z_{k_0}|)$.
	Then we have a constant $\beta>0$ such that for any $k_0\in\{2,\dots,n\}$ and $s_{k_0}\ge 0$,
	$$\PP[\tau^{z_1}_{r_1}<\cdots<\tau^{z_n}_{r_n}<\infty;{\rm {inrad}}_{H_{\tau^{z_1}_{r_1}}}(z_{k_0})\le s_{k_0}]\lesssim  F(z_1,\dots,z_{n};r_1,\dots,r_{n})\Big(\frac{s_{k_0}}{|z_{k_0}-z_1|\wedge |z_{k_0}|}\Big)^{\beta}.$$ \label{RZ-Thm3.1}
\end{theorem}

This theorem is similar to \cite[Theorem 1.1]{RZ}, in which there do not exist the condition ${\rm {inrad}}_{H_{\tau^{z_1}_{r_1}}}(z_{k_0})\le s_{k_0}$ on the LHS or the factor $(\frac{s_{k_0}}{|z_{k_0}-z_1|\wedge |z_{k_0}|} )^{\beta}$ on the RHS. If $s_{k_0}\ge |z_{k_0}-z_1|\wedge |z_{k_0}|$, it follows from \cite[Theorem 1.1]{RZ}; otherwise we do not find a simple way to prove it using \cite[Theorem 1.1]{RZ}. The proof will follow the argument in \cite{RZ}, and take into account the additional condition ${\rm {inrad}}_{H_{\tau^{z_1}_{r_1}}}(z_{k_0})\le s_{k_0}$ during the course. Since the proof  is long and quite different from other proofs of this paper, we postpone  it to the Appendix.

\begin{lemma}
	Let $z_1 \in\HH$ and $0<r<\eta<R$. Let $Z$ be a connected subset of $\HH$. Further suppose that $r<\Imm z_1$ and $\dist(z_1,Z)>R$. Let $\ha\xi_{\tau^{z_1}_\eta}$ be the union of connected components of $H_{\tau^{z_1}_\eta}\cap\{|z-z_1|=R\}$, which disconnect $z_1$ from any point of $Z$ in $H_{\tau^{z_1}_\eta}$.  Then
	\begin{enumerate}
		\item [(i)]  $\PP_{z_1}^r[Z\subset H_{\tau^{z_1}_\eta}, \gamma[\tau^{z_1}_\eta,\tau^{z_1}_r]\cap \ha\xi_{\tau^{z_1}_\eta}\ne\emptyset ]\lesssim (\frac{\eta}{R})^{\alpha/4}$.
		\item [(ii)] $\PP_{z_1}^*[Z\subset H_{\tau^{z_1}_\eta},\gamma[\tau^{z_1}_\eta,T_{z_1}]\cap \ha\xi_{\tau^{z_1}_\eta}\ne\emptyset] \lesssim (\frac{\eta}{R} )^{\alpha/4}$.
	\end{enumerate}\label{stayin}
\end{lemma}
\begin{proof}
	(i)  From \cite[Theorem 2.3]{LR1}, we know that there are constants $C,\delta>0$ such that, if $r<\delta \Imm z_1$, then $\PP[\tau^{z_1}_r<\infty]\ge C G(z_1) r^{2-d}$. Thus, for any $r<\Imm z_1$,
	\BGE \PP[\tau^{z_1}_r<\infty]\ge C\delta^{2-d} G(z_1) r^{2-d}\gtrsim F(z_1) r^{2-d}=F(z_1;r).\label{lower-bd}\EDE
	When $Z\subset H_{\tau^{z_1}_\eta}$, by \cite[Lemma 2.1]{RZ}, there is a unique connected component of $\ha\xi_{\tau^{z_1}_\eta}$, denoted by $\xi_{\tau^{z_1}_\eta}$, which disconnects $z_1$ from $Z$ and any other connected component of $\ha\xi_{\tau^{z_1}_\eta}$ in $H_{\tau^{z_1}_\eta}$. Given that $Z\subset H_{\tau^{z_1}_\eta}$, modulo the event that $\gamma$ passes through an end point of $\xi_{\tau^{z_1}_\eta}$, which has probability zero, the event that $\gamma$ up to any time visits $\ha\xi_{\tau^{z_1}_\eta}$ coincide with the event that the same part of $\gamma$ visits $\xi_{\tau^{z_1}_\eta}$.
	We will show that
	\BGE \PP[Z\subset H_{\tau^{z_1}_\eta},\gamma[\tau^{z_1}_\eta,\tau^{z_1}_r]\cap \xi_{\tau^{z_1}_\eta}\ne\emptyset;\tau^{z_1}_r<\infty ]\lesssim F(z_1;r) \Big(\frac{\eta}{R}\Big)^{\alpha/4},\label{eta-eps-R}\EDE
	which together with (\ref{lower-bd}) implies (i).
	
	To prove (\ref{eta-eps-R}), using Lemma \ref{Py}, we may assume that $r=\eta e^{-n}$ for some $n\in\N$. Let $r_k=\eta e^{-k}$, $0\le k\le n$. Let $E$ denote the event in (\ref{eta-eps-R}). Then $E=\bigcup_{k=1}^n E_k$, where
	$$E_k=\{Z\subset H_{\tau^{z_1}_\eta},\xi_{\tau^{z_1}_\eta}\subset H_{\tau^{z_1}_{r_{k-1}}}; \gamma[\tau^{z_1}_{r_{k-1}},\tau^{z_1}_{r_k}]\cap \xi_{\tau^{z_1}_\eta}\ne\emptyset;\tau^{z_1}_{r_n}<\infty\}.$$
	Let $y_1=\Imm z_1$.
	From \cite[Lemma 2.6]{RZ} we know that
	\BGE \PP[\tau^{z_1}_{r_{k-1}}<\infty]\lesssim \frac{P_{y_1}(r_{k-1})}{P_{y_1}(|z_1|)};\quad \PP[\tau^{z_1}_{r_n}<\infty|\F_{\tau^{z_1}_{r_k}},\tau^{z_1}_{r_k}< \infty]\lesssim \frac{P_{y_1}(r_n)}{P_{y_1}(r_k)}.\label{one-point}\EDE
	Suppose $\tau^{z_1}_{r_{k-1}}<\infty$ and $\xi_{\tau^{z_1}_\eta}\subset H_{\tau^{z_1}_{r_{k-1}}}$. Then $\xi_{\tau^{z_1}_\eta}$ is a crosscut of $H_{\tau^{z_1}_{r_{k-1}}}$.
	By \cite[Lemma 2.1]{RZ}, there is a unique connected component of $\{|z-z_1|=\sqrt{r_{k-1} R}\}\cap H_{\tau^{z_1}_{r_{k-1}}}$,  denoted by $\rho$, which (i) separates $z_1$ from $\xi_{\tau^{z_1}_\eta}$ in $H_{\tau^{z_1}_{r_{k-1}}}$, and (ii) also separates $z_1$ from any other connected component of $\{|z-z_1|=\sqrt{r_{k-1} R}\}\cap H_{\tau^{z_1}_{r_{k-1}}}$ that satisfies (i). Such $\rho$ is a crosscut of $H_{\tau^{z_1}_{r_{k-1}}}$, and divides $H_{\tau^{z_1}_{r_{k-1}}}$ into a bounded domain and an unbounded domain.  Let $E_b$ (resp.\ $E_u$) denote the events that $\xi_{\tau^{z_1}_\eta}$ lies in the bounded (resp.\ unbounded) domain. See Figure \ref{Fig-Ebu}.

\begin{figure}
	\hfill
	\includegraphics[width=0.45\textwidth]{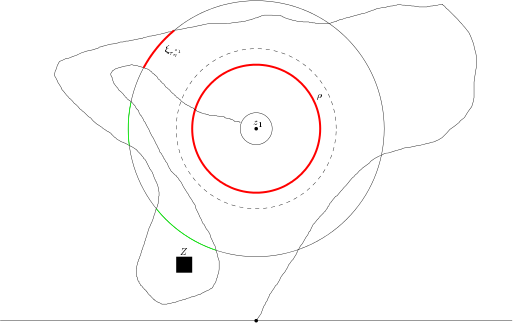}%
	\hfill
	\includegraphics[width=0.45\textwidth]{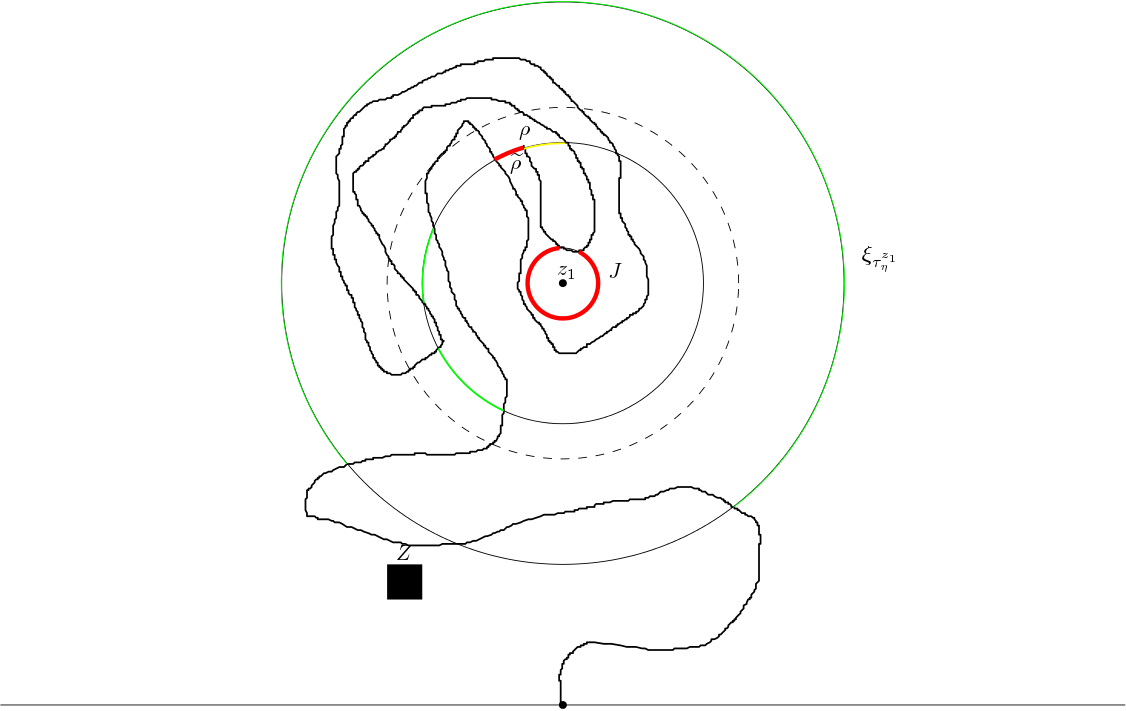}%
	\hfill
	\caption{The two pictures above illustrate the events $E_b$ (left) and $E_u$ (right). In both pictures, the circles are  all centered at $z_1$; the solid circles have radii $R>\sqrt{r_{k-1}R}>r_{k-1}$, respectively, and the dotted circle has radius $\eta$. The zigzag curves are $\gamma$ up to $\tau^{z_1}_{r_{k-1} }$ and $T_\rho$, respectively. In both pictures, the pair of arcs that contribute the factor from the boundary estimate ($\xi_{\tau^{z_1}_\eta}$ and $\rho$ on the left, $\til\rho$ and $J$ on the right) are labeled and colored red. Note that on the left, $\ha\xi_{\tau^{z_1}_\eta}$ has three components, and so is different from $\xi_{\tau^{z_1}_\eta}$; and on the right, $\ha\xi_{\tau^{z_1}_\eta}$ agrees with $\xi_{\tau^{z_1}_\eta}$. On the right, there are three connected components that satisfy the first separation property of $\rho$. The components other than $\rho$ are colored green.}%On the left, $\ha\xi_{\tau^{z_1}_\eta}$ has three connected component, among which $\xi_{\tau^{z_1}_\eta}$ is labeled and colored red, and other components are colored green; the curve $\gamma$ is stopped at $\tau^{z_1}_{r_{k-1} }$; and the set $\{|z-z_1|=\sqrt{r_{k-1} R}\}\cap H_{\tau^{z_1}_{r_{k-1}}}$ has a unique connected component $\rho$ (labeled and colored red) which  separates $z_1$ from $\xi_{\tau^{z_1}_\eta}$ in $H_{\tau^{z_1}_{r_{k-1}}}$. On the right, $\ha\xi_{\tau^{z_1}_\eta}$ is connected, and so agrees with $\xi_{\tau^{z_1}_\eta}$ (labeled and colored green); the set $\{|z-z_1|=\sqrt{r_{k-1} R}\}\cap H_{\tau^{z_1}_{r_{k-1}}}$ has three connected components which  separates $z_1$ from $\xi_{\tau^{z_1}_\eta}$ in $H_{\tau^{z_1}_{r_{k-1}}}$. One of them is $\rho$, which is labeled and colored half red and half green, and others are colored green; the curve $\gamma$ is stopped at $T_\rho$;  the red subcurve of $\rho$ is $\til\rho$; and the red arc on the circle $\{|z-z_1|=r_{k-1}\}$ is $J$.}
	\label{Fig-Ebu}
\end{figure}

For the event $E_b$, we apply \cite[Lemma 2.5]{RZ} to the crosscuts $\rho$ and $\xi_{\tau^{z_1}_\eta}$ to get
	\begin{align*}
	&\PP[Z\subset H_{\tau^{z_1}_\eta}, \gamma[\tau^{z_1}_{r_{k-1}},\tau^{z_1}_{r_k}]\cap \xi_{\tau^{z_1}_\eta}\ne\emptyset;E_b|\F_{\tau^{z_1}_{r_{k-1}}},\tau^{z_1}_{r_{k-1}}<\infty,\xi_{\tau^{z_1}_\eta}\subset H_{\tau^{z_1}_{r_{k-1}}}]\\ \lesssim& e^{-\alpha \pi d_{\C}(\rho,\xi_{\tau^{z_1}_\eta})}\lesssim \Big(\frac{r_{k-1}}{R} \Big)^{\alpha/4}.
	\end{align*}
	Combining this estimate with (\ref{one-point}) and Lemma \ref{Py}, we get
	\BGE \PP[E_k\cap E_b]\lesssim F(z_1;r)\Big(\frac{r_{k-1}}{R} \Big)^{\alpha/4}\Big(\frac{r_{k-1}}{r_k}\Big)^\alpha .\label{Ekb}\EDE
	
	If $E_u$ happens, then $\rho$ separates $z_1$ from $\infty$ in $H_{\tau^{z_1}_{r_{k-1}}}$. Let $T_\rho$ denote the first time after $\tau^{z_1}_{r_{k-1}}$ that $\gamma$ visits $\rho$, and let $\til\rho$ (resp.\ $J$) be a connected component of $\rho\cap H_{T_\rho}$ (resp.\ $\{|z-z_1|=r_{k-1}\}\cap H_{T_\rho}$ that separates $z_1$ from $\infty$ in $H_{T_\rho}$. Applying \cite[Lemma 2.5]{RZ} to $\til\rho$ and $J$, we get
	$$\PP[\tau^{z_1}_{r_k}<\infty;E_u| \F_{T_\rho},T_\rho<\infty,\tau^{z_1}_{r_{k-1}}<\infty, \xi_{\tau^{z_1}_\eta}\subset H_{\tau^{z_1}_{r_{k-1}}}]\lesssim e^{-\alpha \pi d_{\C}(\til\rho,J)}\lesssim \Big(\frac{r_{k-1}}{R} \Big)^{\alpha/4}.$$
	Combining this estimate with (\ref{one-point}) and Lemma \ref{Py}, we get
	\BGE \PP[E_k\cap E_u]\lesssim F(z_1;r)\Big(\frac{r_{k-1}}{R} \Big)^{\alpha/4} \Big(\frac{r_{k-1}}{r_k}\Big)^\alpha .\label{Eku}\EDE
	Since $E=\bigcup_{k=1}^n E_k$, using (\ref{Ekb}) and (\ref{Eku}), we get
	$$\PP[E]\lesssim F(z_1;r) \Big(\frac{r_{k-1}}{r_k}\Big)^\alpha \sum_{k=1}^n  \Big(\frac{r_{k-1}}{R} \Big)^{\alpha/4}
	\le F(z_1;r)\Big(\frac{\eta}{R} \Big)^{\alpha/4}\frac{ e^\alpha}{1-e^{-\alpha/{4}}}.$$
	From this we get (\ref{eta-eps-R}) and finish the proof of (i).
	
	(ii) From Lemma \ref{RN<1} and (i), we get $\PP_{z_1}^*[Z\subset H_{\tau^{z_1}_\eta},\gamma[\tau^{z_1}_\eta,\tau^{z_1}_r]\cap \ha\xi_{\tau^{z_1}_\eta}\ne\emptyset] \lesssim (\frac{\eta}{R} )^{\alpha/4}$ for any $r>0$ smaller than $\eta$ and $\Imm z_1$. We then complete the proof by sending $r\to 0$.
\end{proof}

\begin{corollary}
	Let $z_1,z_0 \in\HH$ and $0<r<\eta<R$. Let $Z$ be a connected subset of $\HH$. Further suppose that  $R-\eta,\eta-r>2|z_1-z_0|$, $r<\Imm z_0$ $r<\Imm z_1$, and $\dist(z_1,Z)>R$.  Let $\ha\xi_{\tau^{z_1}_\eta}$ be the union of connected components of $H_{\tau^{z_1}_\eta}\cap\{|z-z_1|=R\}$, which disconnect $z_1$ from any point of $Z$ in $H_{\tau^{z_1}_\eta}$.  Then
	\begin{enumerate}
		\item [(i)]  $\PP_{z_0}^r[Z\subset H_{\tau^{z_1}_\eta},\gamma[\tau^{z_1}_\eta,\tau^{z_0}_r]\cap \ha\xi_{\tau^{z_1}_\eta}\ne\emptyset ]\lesssim (\frac{\eta}{R})^{\alpha/4}$.
		\item [(ii)] $\PP_{z_0}^*[Z\subset H_{\tau^{z_1}_\eta},\gamma[\tau^{z_1}_\eta,T_{z_0}]\cap \ha\xi_{\tau^{z_1}_\eta}\ne\emptyset] \lesssim (\frac{\eta}{R} )^{\alpha/4}$.
	\end{enumerate}\label{stayin-cor}
\end{corollary}
\begin{proof}
	(i) Let $\eta'=\eta+|z_1-z_0|$ and $R'=R-|z_1-z_0|$. Then $\tau^{z_0}_{\eta'}\le \tau^{z_1}_{\eta}$, and $\{|z-z_0|=R'\}$ disconnects $z_1,z_0$ from $\{|z-z_1|=R\}$. Let $\ha\xi'_{\tau^{z_1}_\eta}$ be the union of connected components of $H_{\tau^{z_1}_\eta}\cap\{|z-z_1|=R'\}$, which disconnect  $z_1,z_0$ from $Z$ in $H_{\tau^{z_0}_{\eta'}}$. Then $\ha\xi'_{\tau^{z_1}_\eta}$ separates $z_1,z_0$ from $\ha\xi_{\tau^{z_1}_\eta}$ as well. If $Z\subset H_{\tau^{z_1}_\eta}$ and $\gamma[\tau^{z_1}_\eta,\tau^{z_0}_r]\cap \ha\xi_{\tau^{z_1}_\eta}\ne\emptyset$, then a.s.\ $\gamma[\tau^{z_0}_{\eta'},\tau^{z_0}_r]\cap \ha\xi'_{\tau^{z_1}_\eta}\ne\emptyset$.	Thus, by Lemma \ref{stayin},
	$$\PP_{z_0}^r[Z\subset H_{\tau^{z_1}_\eta},\gamma[\tau^{z_1}_\eta,\tau^{z_0}_r]\cap \ha\xi_{\tau^{z_1}_\eta}\ne\emptyset|\tau^{z_0}_r<\infty]  \lesssim \Big(\frac{\eta'}{R'}\Big)^{\alpha/4}\lesssim  \Big(\frac{\eta}{R}\Big)^{\alpha/4}.$$
	(ii) This follows from Lemma \ref{RN<1} and (i) by sending $r\to 0$.
\end{proof}

The next lemma will be frequently used.

\begin{lemma}
	Let $z_1,\dots,z_{n}$ be distinct points in $\HH$, where $n\ge 2$. Let $K$ be an $\HH$-hull such that $0\in\lin K$ and $\HH\sem K$ contains $z_1,\dots,z_{n}$. Let $w_0$ be a prime end of $\HH\sem K$ that sits on $\pa K$.
	% Let $S=[c_K,d_K]-g_K(w_0)$ and $g=g_K-g_K(w_0)$.
	Suppose that $\dist(z_k, K)\ge  s_k$, $2\le k\le n$, where $s_k\in (0,|z_k|\wedge |z_k-z_1|)$. Then
	\begin{align*}
	& F(z_1)F_{(\HH\sem K;w_0,\infty)}(z_2,\dots,z_{n})\\ \lesssim  & %F(z_1,\dots,z_{n}) \prod_{k=2}^{n}  \frac{P_{y_k}(l_k)}{P_{y_k}( s_k\wedge l_k)}\le
	F(z_1,\dots,z_{n}) \prod_{k=2}^{n} \Big(\frac{|z_k|\wedge |z_k-z_1|}{s_k}\Big)^\alpha \min_{2\le k\le n}\Big(\frac{\dist(g_K(z_k),S_K)}{|g_K(z_k)-g_K(w_0) |}\Big)^\alpha \\
	\lesssim  &   F(z_1,\dots,z_{n}) \prod_{k=2}^{n} \Big(\frac{|z_k|\wedge |z_k-z_1|}{s_k}\Big)^\alpha.
	\end{align*}
	%  where the implicit constants depend only on $\kappa$ and $n$.
	\label{FF-F}
\end{lemma}
\begin{proof}
	Since $w_0\in\pa K$, we get $g_K(w_0)\in S_K$. So the first inequality immediately implies the second. Let $y_k$ and $l_k$, $1\le k\le n$, be defined by (\ref{ldyR}). Let $g=g_K-g_K(w_0)$.  Let $\til z_k=g(z_k)$, $2\le k\le n$; and define $\til y_k$ and $\til l_k$ using (\ref{ldyR}) for the $n-1$ points: $\til z_k$, $2\le k\le n$. In particular, $\til l_2=|\til z_2|$. Let $S=S_K-g_K(w_0) \ni 0$. Define for $2\le k\le n$,
	$$\til l^S_k=\dist(\til z_k,S\cup \{\til z_j:2\le j<k\}),\quad l^K_k=\dist(z_k,K\cup \{z_j:2\le j<k\}).$$
	From Koebe's $1/4$ theorem, we get $|g'(z_k)| l^K_k\asymp \til l^S_k$. We claim that when $\eps$ is small,
	\BGE  \frac{P_{\til y_k}(|g'(z_k)| \eps)}{P_{\til y_k}(\til l_k^S)}\asymp \frac{P_{y_k}(\eps)}{P_{y_k}(l_k^K)},\quad \mbox{if }\eps\le \dist(z_k,K).\label{lesssim0}\EDE
	We consider two cases. If $y_k\le \dist(z_k,K)/10$, applying Koebe's distortion theorem, we get $\til y_k\asymp |g'(z_k)| y_k$. Then we have (\ref{lesssim0}) because $\frac{P_{ay}(ar)}{P_{ay}(aR)}=\frac{P_y(r)}{P_y(R)}$.  If $y_k\geq \dist(z_k,K)/10$, then $y_k\gtrsim l_k^K$. Applying Koebe's $1/4$ theorem, we get $\til y_k\gtrsim |g'(z_k)| \dist(z_k,K)\gtrsim \til l^K_k$. Thus, when $\eps\le \dist(z_k,K)$, we have (\ref{lesssim0}) because both sides of it are comparable to $(\frac{\eps}{l_k^K})^{2-d}$.

 Recall that
	$$F(z_1)=\lim_{\eps\to 0^+} \eps^{d-2} \frac{P_{y_1}(\eps)}{P_{y_1}(l_1)};\quad F(z_1,\dots,z_{n})=\lim_{\eps\to 0^+} \eps^{n(d-2)}\prod_{k=1}^{n} \frac{P_{y_k}(\eps)}{P_{y_k}(l_k)}.$$
	Since $g$ is a conformal map from $D$ onto $\HH$ that fixes $\infty$ and takes $w_0$ to $0$, we have
	$$F_{(D;w_0,\infty)}(z_2,\dots,z_{n})=\prod_{k=2}^{n} |g'(z_k)|^{2-d} \lim_{\eps\to 0^+} \eps^{(n-1)(d-2)}\prod_{k=2}^{n} \frac{P_{\til y_k}(\eps)}{P_{\til y_k}(\til l_k)}.$$
	From (\ref{lesssim0}), we get
	$$F(z_1)F_{(D;w_0,\infty)}(z_2,\dots,z_{n})\asymp  \prod_{k=2}^n \Big(\frac{P_{y_k}(l_k)}{P_{y_k}(l^K_k)}\cdot \frac{P_{\til y_k}(\til l_k^S)}{P_{\til y_k}(\til l_k)}\Big)\cdot F(z_1,\dots,z_{n}).$$
	Since $l^K_k=\dist(z_k,K)\wedge \dist(z_k:\{z_j:2\le j<k\})\ge s_k\wedge \dist(z_k:\{z_j:2\le j<k\})$, $l_k=|z_k|\wedge |z_k-z_1|\wedge \dist(z_k:\{z_j:2\le j<k\})$, and $|z_k|\wedge |z_k-z_1|\ge s_k$, we get
	$$\frac{P_{y_k}(l_k)}{P_{y_k}(l^K_k)}\le \Big(\frac{|z_k|\wedge |z_k-z_1|\wedge \dist(z_k:\{z_j:2\le j<k\})}{s_k\wedge \dist(z_k:\{z_j:2\le j<k\})}\Big)^\alpha\le \Big(\frac{|z_k|\wedge |z_k-z_1|}{s_k}\Big)^\alpha.$$
	Note that $\frac{P_{\til y_k}(\til l_k^S)}{P_{\til y_k}(\til l_k)}\le 1$, $2\le k\le n$, and  $\frac{P_{\til y_2}(\til l_2^S)}{P_{\til y_2}(\til l_2)}=\frac{P_{\til y_2}(\dist(\til z_2,S))}{P_{\til y_2}(|\til z_2|)}=(\frac{\dist(\til z_2,S)}{|\til z_2|})^\alpha$. From Lemma \ref{perm-lem}, we get $\prod_{k=2}^n \frac{P_{\til y_k}(\til l_k^S)}{P_{\til y_k}(\til l_k)}\lesssim  \min_{2\le k\le n} \Big(\frac{\dist(\til z_k,S)}{|\til z_k|}\Big)^\alpha$.
	Then the proof is completed.
\end{proof}
The next two lemmas are useful when we want to prove the lower bound.

\begin{lemma} \label{help-lower*}
	Let $z_1,\dots,z_{n}$ be distinct points in $\lin\HH\sem\{0\}$. Let $r_j\in(0,d_j)$, $1\le j\le n$, where $d_j$'s are given by (\ref{ldyR}). Let $K$ be an $\HH$-hull such that $0\in\lin K$, and let $U_0\in S_K$. Suppose that $z_k\not\in\lin K$ and
\BGE  \dist(g_K(z_j),S_K)\asymp |\til z_j|:=|g_K(z_j)-U_0|, \quad 1\le j\le n. \label{angle*}\EDE
%Let $R_1=\dist(z_1,K)$.
Suppose $I=\{1=j_1<\dots<j_{|I|}\}\subset\{1,\dots,n\}$ satisfies that $r_j\lesssim \dist(z_j,K)$. Then we have
	\begin{align*}
	&F(z_1;\dist(z_1,K))\cdot F(\til z_{j_1},\dots,\til z_{j_{|I|}};|g_K'(z_{j_1})|r_{j_1},\dots,|g_K'(z_{j_{|I|}})|r_{j_{|I|}})\\
	\gtrsim  & F(z_1,z_2,\dots,z_n;r_1,r_2,\dots,r_n).
	\end{align*}
	The implicit constant in the conclusion depends on the implicit constants in the assumption.
\end{lemma}
\begin{proof}
By reordering the points and using (\ref{Frasymp}), we may assume that $I=\{1,\dots,m\}$.
Let $y_k$ and $l_k$, $1\le k\le n$, be defined by (\ref{ldyR}). Also take $\til y_k$ and $\til l_k$ be the corresponding quantities for $\til z_k$, $1\le k\le m$. Let $S=S_K-U_0 \ni 0$.
For $1 \le k \le m$ define.
$$\til l^{S}_k=\dist(\til z_k,S\cup \{\til z_j:1\le j<k\}),\quad l^{K}_k=\dist(z_k,K\cup \{z_j:1\le j<k\}).$$
It is clear that $l^K_k \le l_k$.
By Koebe's $1/4$ theorem we have
$|g'_K(z_k)|l^{K}_k \asymp \til l^{S}_k$.
From (\ref{angle*}) we know that $\til l^{S}_k \asymp \til l_k$. Since $r_k\lesssim \dist(z_k,K)$, $1\le k\le m$, the argument of (\ref{lesssim0}) gives us
\BGE \frac{P_{\til y_k}(|g_K'(z_k)|r_k)}{P_{\til y_k}(\til l_k)}\asymp\frac{P_{y_k}(r_k)}{P_{y_k}(l_k^K)},\quad 1\le k\le m. \label{P/P1*}
\EDE
 Since $l^K_k \le l_k$,  we have
\BGE
\frac{P_{\til y_k}(|g_K'(z_k)|r_k)}{P_{\til y_k}(\til l_k)} \gtrsim  \frac{P_{y_k}(r_k)}{P_{y_k}(l_k)}, \quad 1\le k\le m. \label{P/P2*}
\EDE

Multiplying (\ref{P/P1*}) for $k=1$, (\ref{P/P2*}) for $2\le k\le m$, the equality $F(z_1;\dist(z_1,K))=\frac{P_{y_1}(l_1^K)}{P_{y_1}(l_1)}$, and the inequalities $1\ge \frac{P_{y_k}(r_k)}{P_{y_k}(l_k)}$ for $m+1\le k\le n$, we get the desired inequality.
\end{proof}

%The next lemma (which is easy) shows that if the set of points can be divided in to two well-separated sets, then $F$ is almost multiplicative on these sets.
\begin{lemma} \label{multiplicative}
Suppose we have set of distinct points $z_1,\dots,z_n$ in $\Half$. Let $l_j$, $1 \le j \le n$, be defined by \eqref{ldyR}.
Let $m\in\{1,\dots,n-1\}$. Take $w_j=z_{m+j}$, $1 \le j \le n-m$. Let $l^w_j$, $1 \le j \le n-m $, be the corresponding quantity for $w_j$'s. Suppose $l_{m+j} \asymp l^w_j$, $1\le j\le n-m$. Then
\[
F(z_1,\dots,z_m;r_1,\dots,r_m)F(z_{m+1},\dots,z_n;r_{m+1},\dots,r_n)\asymp F(z_1,\dots,z_n;r_1,\dots,r_n).
\]   	
The implicit constant in the result depends on the implicit constants in the assumption.
\end{lemma}
\begin{proof}
Just write the definition of $F$ and note that $P_{\Imm z_{m+j}}(l_{m+j}) \asymp P_{\Imm w_j}(l^w_j)$.
\end{proof}

%\begin{remark}
%Note that by definition of $F$, the LHS of the equation above is always at least as big as RHS.	
%\end{remark}

\section{Main Theorems} \label{maintheorems}
We state the main theorems of the paper in this section. It is clear that the existence and the continuity of the  (unordered) Green's function follows from the existence and the continuity of ordered Green's function, i.e., the limit
$$\lim_{r_1,\dots,r_n\downarrow 0} \prod_{j=1}^n r_j^{d-2} \PP[\tau^{z_1}_{r_1}<\dots<\tau^{z_n}_{r_n}<\infty].$$
So the statements of Theorems \ref{main} and \ref{main2} are about ordered Green's functions.
%It is clear that if the ordered Green's function exists, then the (unordered) Green's function also exists.

For that purpose we define functions $\ha G(z_1,\dots,z_n)$ by induction on $n$. For $n=1$, let $\ha G(z)=G(z)$ given by (\ref{G(z)}). Suppose $n\ge 2$ and $\ha G$ has been defined for $n-1$ points. Now we define $\ha G$ for distinct $n$ points $z_1,\dots,z_n\in\HH$. Given a chordal Loewner curve $\gamma$, for any $t\ge 0$, if $z_2,\dots,z_n\in H_t$, we define
$$\ha G_t(z_2,\dots,z_n)=\prod_{j=1}^n |g_t'(z_j)|^{2-d} \ha G(Z_t(z_2) ,\dots,Z_t(z_n));$$
otherwise define $\ha G_t(z_2,\dots,z_n)=0$. Recall that $Z_t=g_t-U_t$ is the centered Loewner map at time $t$. Now we define $\ha G(z_1,\dots,z_n)$ by
$$\ha G (z_1,\dots,z_n)=G(z_1)\EE_{z_1}^*[\ha G_{T_{z_1}}(z_2,\dots,z_n)].$$ Recall that $\EE_{z_1}^*$ is the expectation w.r.t.\ the two-sided radial SLE$_\kappa$ curve through $z_1$.

The authors of \cite{LW} proved that the two-point
(conformal radius version) Green's function exists and agrees with the $\ha G(z_1,z_2)$ defined above (up to a constant). Their proof used the closed-form formula of one-point Green's function (\ref{G(z)}). We will show their result is also true for arbitrary number of points. The difficulty is that there is no closed-form formula known for two-point Green's function. We find a way to prove the above statement without knowing the exact formula of the Green's functions. Below is our first main theorem

%The following theorem shows that $\ha G$ agrees with the ordered Green's functions.

%The following theorem is about the existence and convergence rate for the (ordered and unordered) multi-point Green functions, and the sharp estimates for the unordered multi-point Green functions.

\begin{theorem}
	There are finite constants $C_n,B_n>0$ and $\beta_n,\delta_n\in(0,1)$  such that the following holds.
	Let $z_1,\dots,z_n$ be distinct points in $\HH$. Let $R_j$, $1\le j\le n$, $Q$ and $F$ be defined by (\ref{ldyR},\ref{Q}). Then for any $r_1,\dots,r_n>0$ that satisfy
	\BGE Q^{B_n}\frac{r_j}{R_j}<\delta_n,\quad 1\le j\le n,\label{Cond-main}\EDE we have
	%\begin{align}
	%  &|\prod_{j=1}^n r_j^{d-2} \Prob[\tau^{z_1}_{r_1}<\cdots<\tau^{z_n}_{r_n}<\infty]-\ha G(z_1,\dots,z_n)|\nonumber \\
	%  \le & C_n  F   \sum_{j=1}^n  \Big(Q^{B_n} \frac{r_j}{R_j}\Big)^{\beta_n}. \label{RHS-main}
	%\end{align}
	\BGE  \Big |\prod_{j=1}^n r_j^{d-2} \PP[\tau^{z_1}_{r_1}<\cdots<\tau^{z_n}_{r_n}<\infty]-\ha G(z_1,\dots,z_n)\Big |
	\le   C_n  F   \sum_{j=1}^n  \Big(Q^{B_n} \frac{r_j}{R_j}\Big)^{\beta_n}. \label{RHS-main}
	\EDE
	As an immediate consequence, the $G(z_1,\dots,z_n)$ defined by \eqref{mlti-green} exists and is equal to $\sum_{\sigma } \ha G(z_{\sigma(1)},\dots,z_{\sigma(n)})$, where the summation is over all permutations of $\{1,\dots,n\}$. \label{main}
\end{theorem}

Proving the convergence of $n$-point Green's function requires certain modulus of continuity of $(n-1)$-point Green's functions, which is given by the following theorem.

\begin{theorem}
	There are finite constants $C_n,B_n>0$ and $\beta_n,\delta_n\in(0,1)$  such that the following holds.
	Let $z_1,\dots,z_n$ be distinct points in $\HH$. Let $d_j$, $1\le j\le n$, $Q$ and $F $ be defined by (\ref{ldyR},\ref{Q}).  If  $z_1',\dots,z_n'\in \HH$ satisfy that
	\BGE Q^{B_n}\frac{|z_j'-z_j|}{d_j}<\delta_n,\quad \frac{|\Imm  z_j' -\Imm z_j|}{\Imm z_j}<\delta_n,\quad 1\le j\le n,\label{Cond-main2} \EDE then
	%\begin{align}
	%&  |\ha G(z_1',\dots, z_n')-\ha G(z_1,\dots,z_n)| \nonumber \\
	%\le & C_n F \sum_{j=1}^n\Big(\Big( Q^{B_n }  \frac{| z_j'-z_j|}{d_j}\Big)^{\beta_n}+  \Big(\frac{|\Imm  z_j' -\Imm z_j|}{\Imm z_j}\Big)^{\beta_n} \Big).\label{RHS-main2}
	%\end{align}
	\BGE |\ha G(z_1',\dots, z_n')-\ha G(z_1,\dots,z_n)|
	\le  C_n F \sum_{j=1}^n \Big( Q^{B_n }  \frac{| z_j'-z_j|}{d_j}\Big)^{\beta_n}+  \Big(\frac{|\Imm  z_j' -\Imm z_j|}{\Imm z_j}\Big)^{\beta_n}  .\label{RHS-main2}\EDE
	Moreover, the same inequality holds true (with bigger $C_n$)  if $\ha G$ is replaced by $G$. \label{main2}
\end{theorem}
The sharp lower bound for the Green's function is provided in the theorem below. The reader may compare it with Proposition \ref{RZ-Thm1.1}.

\begin{theorem} \label{strong-lower}
	There are finite constants $C_n>0$ and $V_n>1$ such that for any distinct points $z_1,\dots,z_n\in\lin\HH\sem\{0\}$ and any $r_j\in(0,d_j)$, $1\le j\le n$, we have
	$$ \PP[\tau^{z_j}_{r_j}<\tau_{\{|z|=V_n\sum_{i=1}^{n}|z_i|\}},1\le j\le n] \geq C_n F(z_1,\dots,z_n;r_1,\dots,r_n). %\label{lowerbound}
$$
\end{theorem}

%This theorem together with Proposition \ref{RZ-Thm1.1} implies that $G(z_1,\dots,z_n)\asymp F(z_1,\dots,z_n)$.

We have a local martingale related with the Green's function.

\begin{corollary} \label{martingale}
	For fixed distinct $z_1,\dots,z_n\in\HH$, $M_t:=\ha G_t(z_1,\dots,z_n)$ is a local martingale up to the first time any $z_j$, $1\le j\le n$, is swallowed by $\gamma$.
\end{corollary}
\begin{proof}
	It suffices to prove the following. Let $K$ be any $\HH$-hull such that $0\in K$ and $z_1,\dots,z_n\in\HH\sem K$. Let $\tau=\inf\{t>0:\gamma[0,t]\not\subset K\}$. Then $M_{t\wedge \tau}$ is a martingale. To prove this, we pick a small $r>0$, and consider the martingale
	$$M^{(r)}_t:=r^{n(d-2)}\PP[\tau^{z_1}_r<\cdots<\tau^{z_n}_r<\infty|\F_{t\wedge \tau}].$$
	By the convergence theorem and Koebe's distortion theorem, we have $M^{(r)}_t\to M_{t\wedge\tau}$ as $r\to 0$. In order to have the desired result, we need uniform convergence. This can be done using the the convergence rate in Theorem \ref{main} and a compactness result from \cite{LERW}. Let $z_{j;t}=g_t(z_j)-U_t$; let $Q_t$ and $R_{j;t}$ be the $Q$ and $R_j$ for $z_{1;t},\dots,z_{n;t}$; let $F_t=\prod_{j=1}^n |g_t'(z_j)|^{2-d} F(z_{1;t},\dots,z_{n;t})$. It suffices to show that $|g_t'(z_j)|,Q_t,R_{j;t},F_t$, $1\le j\le n$, $0\le t\le\tau$, are all bounded from both above and below by a finite positive constant depending only on  $\kappa$, $K$, and $z_1,\dots,z_n$. The existence of these bounds all follow directly or indirectly from  \cite[Lemma 5.4]{LERW}. For example, to prove that $F_t$, $0\le t\le \tau$, are bounded above, we need to
	prove that $|z_{j;t}-z_{k;t}|$, $j\ne k$, and $|z_{j,t}|$, $0\le t\le \tau$, are all bounded below. It suffices to show that $|g_L(z_j)-g_L(z_k)$, $j\ne k$, and $\dist(g_L(z_j),S_L)$ for all $L$ in ${\cal H}(K)$, the set of $\HH$-hulls $L$ with $L\subset K$, are bounded below. Suppose $|g_L(z_j)-g_L(z_k)|$, $j\ne k$, $L\in{\cal H}(K)$, are not bounded below by a constant. Then there are $z_j\ne z_k$ and a sequence $(L_n)\subset {\cal H}(K)$ such that $|g_{L_n}(z_j)-g_{L_n}(z_k)|\to 0$. Since ${\cal H}(K)$ is a compact metric space (\cite[Lemma 5.4]{LERW}), by passing to a subsequence, we may assume that $L_{n}\to L_0\in{\cal H}(K)$. This then implies that $g_{L_0}(z_j)=\lim g_{L_n}(z_j)= \lim g_{L_n}(z_k)=g_{L_0}(z_k)$, which contradicts that $g_{L_0}$ is injective on $\HH\sem K$. To prove that $\dist(g_L(z_j),S_L)$ , $1\le j\le n$, $L\in{\cal H}(K)$, are bounded from below, one may choose a pair of disjoint Jordan curve $J_1,J_2$ in $\HH\sem K$, both of which disconnects $K$ from all of $z_j$'s. Then $\dist(g_L(z_j),S_L)\ge \dist(g_L(J_1),g_L(J_2))$, and the same argument as above shows that $\dist(g_L(J_1),g_L(J_2))$, $L\in{\cal H}(K)$, are bounded from below by a positive constant.
\end{proof}

\begin{remark}
  We may write $M_t=\prod_{j=1}^n|g_t'(z_j)|^{2-d} \ha G(g_t(z_1)-U_t,\dots,g_t(z_n)-U_t)$.
If we know that $\ha G$ is smooth, then using It\^o's formula and Loewner's equation (\ref{Loewner}), one can easily get a second order PDE for $\ha G$. More specifically, if we view $\ha G$ as a function on $2n$ real variables: $x_1,y_1,\dots,x_n,y_n$, then it should satisfy
$$\frac{\kappa}{2}\Big(\sum_{j=1}^n \pa_{x_j}\Big)^2 \ha G
+\sum_{j=1}^n \pa_{x_j} \ha G \cdot \frac{2x_j}{x_j^2+y_j^2}
+\sum_{j=1}^n \pa_{y_j} \ha G \cdot \frac{-2y_j}{x_j^2+y_j^2}
+(2-d)\ha G\cdot \sum_{j=1}^n \frac{-2(x_j^2-y_j^2)}{(x_j^2+y_j^2)^2}=0.$$
Since the PDE does not depend on the order of points, it is also satisfied by the unordered Green's function $G$.

We expect that the smoothness of $\ha G$ can be proved by H\"ormander's theorem because the differential operator in the above displayed formula satisfies H\"ormander's condition.
\end{remark}

\section{Proof of Theorems \ref{main} and \ref{main2}} \label{proof1}
At the beginning, we know that Theorems \ref{main} and \ref{main2} hold for $n=1$ with $\delta_1=1/2$ thanks to \cite[Theorem 2.3]{LR1} and the explicit formulas for $F(z)$ and $G(z)$. % It is also known that a weaker version of Theorem \ref{main} holds for $n=2$ by the work \cite{LW}.
We will prove Theorems \ref{main} and \ref{main2} together using induction. Let $n\ge 2$. Suppose that Theorems \ref{main} and \ref{main2} hold for $n-1$ points. We now prove that they also hold for $n$ points. We will frequently apply the Domain Markov Property (DMP) of SLE (c.f.\ \cite{Law1}) without reference, i.e., if $\gamma$ is a chordal SLE$_\kappa$ curve in $\HH$ from $0$ to $\infty$, and $\tau$ is a finite stopping time, then $Z_\tau(\gamma(\tau+\cdot))$ has the same law as $\gamma$, and is independent of $\F_\tau$.

Fix distinct points $z_1,\dots,z_{n}\in\HH$. Let $l_j$, $d_j$, $R_j$, $y_j$, $1\le j\le n$, $Q$, and $F$ be as defined in (\ref{ldyR},\ref{Q}). Throughout this section, a variable is a real number that depends on $\kappa,n$ and $z_1,\dots,z_n$.
From  the induction hypothesis, Proposition \ref{RZ-Thm1.1}, and (\ref{F}), we see that $\ha G\lesssim F$ holds for $(n-1)$ points.
We write $F_t$ for $F_{(H_t;\gamma(t),\infty)}$. Then Lemma \ref{FF-F} holds with $K=K_t$, $G(z_1)$ in place of $F(z_1)$, and $\ha G_{t}$ in place of $F_{(\HH\sem K_t;w_0,\infty)}$. We will use the following lemma.

\begin{lemma}
	%Let $k_0\in\{2,\dots,n\}$ and $s_{k_0}\in(r_{k_0},|z_{k_0}-z_1|\wedge |z_{k_0}|)$.
	There is some constant $\beta>0$ depending only on $\kappa$ and $n$ such that for any $k_0\in\{2,\dots,n\}$ and $s_{k_0}\ge 0$,
	\[
	G(z_1)E^*_{z_1}[\ha G_{T_{z_1}}(z_2,\dots,z_{n}){\bf 1}\{\inrad_{H_{T_{z_1}}}(z_{k_0})\leq s_{k_0}\}]\lesssim  F \cdot \Big(\frac{s_{k_0}}{|z_{k_0}-z_1|\wedge |z_{k_0}|}\Big)^{\beta}.
	\]\label{RZ-Thm3.1-lim}
\end{lemma}
\begin{proof} This lemma essentially follows from  the induction hypothesis, Theorem \ref{RZ-Thm3.1}, and (\ref{F}). Below are the details. Let $r_j\in(0,R_j/8)$, $1\le j\le n$. From Theorem \ref{RZ-Thm3.1}, there is a constant $\beta>0$ such that
	$$\PP[\tau^{z_1}_{r_1}<\infty]\cdot\EE[{\bf 1}\{\inrad_{H_{\tau^{z_1}_{r_1}}}(z_{k_0})\le s_{k_0}\}\PP[\tau^{z_1}_{r_1}<\cdots<\tau^{z_n}_{r_n}<\infty|\F_{\tau^{z_1}_{r_1}},\tau^{z_1}_{r_1}<\infty]]$$
	$$\lesssim  F(z_1,\dots,z_{n};r_1,\dots,r_{n})\Big(\frac{s_{k_0}}{|z_{k_0}-z_1|\wedge |z_{k_0}|}\Big)^{\beta}.$$
	By the  convergence of $(n-1)$-point Green's function, we know that
	$$\lim_{r_2,\dots,r_n\to 0} \prod_{k=2}^n r_k^{d-2} \PP[\tau^{z_1}_{r_1}<\cdots<\tau^{z_n}_{r_n}<\infty|\F_{\tau^{z_1}_{r_1}},\tau^{z_1}_{r_1}<\infty]=\ha G_{\tau^{z_1}_{r_1}}(z_2,\dots,z_n).$$
	Applying Fatou's lemma with $r_2,\dots,r_n\to 0$ and using the above displayed formulas, we get
	$$\PP[\tau^{z_1}_{r_1}<\infty]\cdot\EE[{\bf 1}\{\inrad_{H_{\tau^{z_1}_{r_1}}}(z_{k_0})\le s_{k_0}\} \ha G_{\tau^{z_1}_{r_1}}(z_2,\dots,z_n)|\tau^{z_1}_{r_1}<\infty] $$
	$$ \lesssim  \lim_{r_2,\dots,r_n\to 0} \prod_{k=2}^n r_k^{d-2} F(z_1,\dots,z_{n};r_1,\dots,r_{n})\Big(\frac{s_{k_0}}{|z_{k_0}-z_1|\wedge |z_{k_0}|}\Big)^{\beta},$$
	which together with Lemma \ref{RN<1} implies that
	$$\PP[\tau^{z_1}_{r_1}<\infty]\cdot \EE_{z_1}^*[{\bf 1}\{\inrad_{H_{\tau^{z_1}_{r_1}}}(z_{k_0})\le s_{k_0}\} \ha G_{\tau^{z_1}_{r_1}}(z_2,\dots,z_n)] $$
	$$\lesssim  \lim_{r_2,\dots,r_n\to 0} \prod_{k=2}^n r_k^{d-2} F(z_1,\dots,z_{n};r_1,\dots,r_{n})\Big(\frac{s_{k_0}}{|z_{k_0}-z_1|\wedge |z_{k_0}|}\Big)^{\beta}.$$
	By the continuity two-sided radial SLE at its end point and the continuity of $(n-1)$ point Green's function, we see that, under the law $\PP_{z_1}^*$, as $r_1\to 0$, $\inrad_{H_{\tau^{z_1}_{r_1}}}(z_{k_0})\to \inrad_{H_{T_{z_1}}}(z_{k_0})$ and $\ha G_{\tau^{z_1}_{r_1}}(z_2,\dots,z_n)\to \ha G_{T_{z_1}}(z_2,\dots,z_n)$. Since $\lim_{r_1\to 0} r_1^{d-2} \PP[\tau^{z_1}_{r_1}<\infty]=G(z_1)$, applying Fatou's lemma with $r_1\to 0$, we get the conclusion.
\end{proof}

\subsection{Convergence of Green's functions} \label{Convergence-Sec}
In this subsection, we work on the inductive step for Theorem \ref{main}. Let $0<r_j<R_j/8$, $1\le j\le n$.   % We write $\tau^j_{r }$ for $\tau^{z_j}_{r }$.
Consider the event $\{\tau^{z_1}_{r_1}<\cdots<\tau^{{z_n}}_{r_{n}}<\infty\}$.
We will transform the scaled probability $\prod_{j=1}^n r_j^{d-2} \PP[\tau^{z_1}_{r_1}<\cdots<\tau^{z_n}_{r_n}<\infty]$ in a number of steps into the ordered $n$-point Green's function $\ha G(z_1,\dots,z_n)$ defined by the expectation of ordered $(n-1)$-point Green's function w.r.t.\ the two-sided radial SLE. In each step we get an error term, and we define a (good) event such that we have a good control of the error when the event happens, and the complement of the event (bad event) has small probability.

Fix  $\vec s=( s_2,\dots, s_{n})$ with $0\le s_j\le|z_j-z_1|\wedge |z_j|$ being variables to be determined later. We define events
\BGE E_{r;\vec{ s}}=\bigcap_{j=2}^{n}\{\dist(z_j,K_{\tau^{z_1}_r})\geq  s_j\},\quad r\ge 0.\label{Ers}\EDE
Here the bad event $E_{r_1;\vec{ s}}^c$ is  the event that $\gamma$ approaches $z_{k_0}$ by distance $s_{k_0}$ for some $2\le k_0\le n$ before it approaches $z_1$ by distance $r_1$. If it also happens that $\tau^{z_1}_{r_1}<\cdots<\tau^{{z_n}}_{r_{n}}<\infty$, then  $\gamma$ goes back and forth between $z_1$ and such $z_{k_0}$. Now we decompose the main event according to $E_{r_1;\vec{ s}}$, and write
\[
\PP[\tau^{z_1}_{r_1}<\cdots<\tau^{{z_n}}_{r_{n}}<\infty]=
\PP[\tau^{z_1}_{r_1}<\cdots<\tau^{ {z_n}}_{r_{n}}<\infty;E_{r_1;\vec s}]+e_1^*.
\]
By Theorem \ref{RZ-Thm3.1} and (\ref{F}), the term $e_1^*$ satisfies that, for some constant $\beta >0$,
$$
0\le e_1^*\lesssim  \prod_{k=1}^n r_k^{2-d} F  \sum_{j= 2}^n \Big(\frac{ s_j}{|z_j|\wedge |z_j-z_1|}\Big)^{\beta } . %\label{e1}
$$

We express
\begin{align*}
&    \PP[\tau^{z_1}_{r_1}<\cdots<\tau^{ {z_n}}_{r_{n}}<\infty;E_{r_1;\vec{ s}}]\\
=&  \PP[\tau^{z_1}_{r_1}<\infty]\cdot  \EE[{\bf 1}_{E_{r_1;\vec{ s}}}\PP[\tau^{z_2}_{r_2}<\cdots<\tau^{{z_n}}_{r_{n}}<\infty| \F_{\tau^{z_1}_{r_1}};E_{r_1;\vec{ s}}]|\tau^{z_1}_{r_1}<\infty].
\end{align*}
%From \cite[Lemma 2.1]{RZ} we get
%\BGE \PP[\tau^{z_1}_{r_1}<\infty]\lesssim F(z_1,r_1) \le r_1^{2-d} F(z_1);\label{ubz1}\EDE
From Proposition \ref{RZ-Thm1.1} and Koebe's distortion theorem, we see that, if
\BGE \frac{r_k}{s_k\wedge R_k}<\frac 16,\quad 2\le k\le n,\label{Cond1}\EDE
then
\BGE \PP[\tau^{z_2}_{r_2}<\cdots<\tau^{{z_n}}_{r_{n}}<\infty| \F_{\tau^{z_1}_{r_1}};E_{r_1;\vec{ s}}]\lesssim  \prod_{k=2}^n r_k^{2-d} F_{\tau^{z_1}_{r_1}}(z_2,\dots,z_n).\label{ubzn}\EDE
Since Theorem \ref{main} holds for $n=1$, we see that, if
\BGE \frac{r_1}{R_1}<\delta_1,\label{Cond2}\EDE then
$$ |\PP[\tau^{z_1}_{r_1}<\infty]-r_1^{2-d} G(z_1)|\lesssim r_1^{2-d} F(z_1)O(r_1/R_1)^{\beta_1}.$$
Now we express
\begin{align*}
&\PP[\tau^{z_1}_{r_1}<\infty]\cdot  \EE[{\bf 1}_{E_{r_1;\vec{ s}}}\PP[\tau^{z_2}_{r_2}<\cdots<\tau^{{z_n}}_{r_{n}}<\infty| \F_{\tau^{z_1}_{r_1}};E_{r_1;\vec{ s}}]|\tau^{z_1}_{r_1}<\infty]\\
=& r_1^{2-d}  G(z_1) \EE[{\bf 1}_{E_{r_1;\vec{ s}}}\PP[\tau^{z_2}_{r_2}<\cdots<\tau^{{z_n}}_{r_{n}}<\infty| \F_{\tau^{z_1}_{r_1}};E_{r_1;\vec{ s}}]|\tau^{z_1}_{r_1}<\infty]+e_2^*.
\end{align*}
From Lemma \ref{FF-F} and (\ref{ubzn}) we see that, if (\ref{Cond1}) and (\ref{Cond2}) hold, then
$$|e_2^*|\lesssim  \prod_{k=1}^n r_k^{2-d} F\cdot \Big (\frac{r_1}{R_1}\Big)^{\beta_1} \prod_{k=2}^n \Big(\frac{|z_k|\wedge |z_k-z_1|}{ s_k }\Big)^\alpha.$$

Define the events
\BGE E_{r;\theta}=\{\dist(g_{\tau^{z_1}_{r}}(z_j),S_{K_{\tau^{z_1}_{r}}} )\ge \theta |g_{\tau^{z_1}_{r}}(z_j)-U_{\tau^{z_1}_{r}}|,2\le j\le n\},\quad r,\theta>0.\label{Erth}\EDE
We understand the bad event $E_{r;\theta}^c$ as the event that  for some $2\le j\le n$ the ``angle'' of $z_j$ is small in terms of $\theta$ viewed from the tip of $\gamma$ at the time $\tau^{z_1}_{r}$. We use the term ``angle'' because ${\dist(g_{\tau^{z_1}_{r}}(z_j),S_{K_{\tau^{z_1}_{r}}} )}\ge \Imm g_{\tau^{z_1}_{r}}(z_j)$, and
$\frac {\Imm g_{\tau^{z_1}_{r}}(z_j)}{|g_{\tau^{z_1}_{r}}(z_j)-U_{\tau^{z_1}_{r}}|}$ equals the sine of the argument of $g_{\tau^{z_1}_{r}}(z_j)-U_{\tau^{z_1}_{r}}$. If the bad event occurs, the argument must be close to $0$ or $\pi$. On the other hand, the bad event may not occur even if the argument is close to $0$ or $\pi$. In the extreme case that $g_{\tau^{z_1}_{r}}(z_j)\in\R$ and $g_{\tau^{z_1}_{r}}(z_j)>U_{\tau^{z_1}_{r}}$, the argument is $0$, and the ratio becomes $\frac {g_{\tau^{z_1}_{r}}(z_j)-\max S_{\tau^{z_1}_{r}}}{g_{\tau^{z_1}_{r}}(z_j)-U_{\tau^{z_1}_{r}}}$, which plays an important role in the proof of the convergence of boundary Green's function (\cite{Law3}). See also the third factor of the second line of the displayed formula in Lemma \ref{FF-F} and Condition (iii) in Proposition \ref{L-shape}.

Fix a variable $\theta\in(0,1)$ to be determined later. According to the occurrence of $E_{r_1;\theta}$, we express
\begin{align*}
& r_1^{2-d}  G(z_1)  \EE[{\bf 1}_{E_{r_1;\vec{ s}}}\PP[\tau^{z_2}_{r_2}<\cdots<\tau^{{z_n}}_{r_{n}}<\infty| \F_{\tau^{z_1}_{r_1}};E_{r_1;\vec{ s}}]|\tau^{z_1}_{r_1}<\infty]\\ =& r_1^{2-d}  G(z_1)\EE[{\bf 1}_{E_{r_1;\vec{ s}}\cap E_{r_1;\theta}}\PP[\tau^{z_2}_{r_2}<\cdots<\tau^{{z_n}}_{r_{n}}<\infty| \F_{\tau^{z_1}_{r_1}};E_{r_1;\vec{ s}}\cap E_{r_1;\theta}]|\tau^{z_1}_{r_1}<\infty]+e_3^*.
\end{align*}
From Lemma \ref{FF-F} and (\ref{ubzn}), we see that
$$0\le e_3^*\lesssim  \prod_{k=1}^n r_k^{2-d} F \prod_{k=2}^n \Big(\frac{|z_k|\wedge |z_k-z_1|}{ s_k }\Big)^\alpha\theta^\alpha.$$

Let $Z=Z_{\tau^{z_1}_{r_1}}$ and $\ha z_k=Z(z_k)$, $2\le k\le n$. Define $\ha d_k$, $2\le k\le n$, and $\ha Q$, for the $(n-1)$ points $\ha z_k$, $2\le k\le n$,  using (\ref{ldyR}) and (\ref{Q}), which are random quantities  measurable w.r.t.\ $\F_{\tau^{z_1}_{r_1}}$. Since Theorem \ref{main} holds for $(n-1)$ points, using Koebe's distortion theorem, we conclude that, for some constants $B_{n-1}>0$ and $\beta_{n-1}, \delta_{n-1}\in(0,1)$, if
$$ \ha Q^{B_{n-1}}\cdot\frac{r_j}{s_j\wedge R_j}<\frac{\delta_{n-1}}{8},\quad 2\le j\le n,$$
then
\begin{align*}
&\Big|\prod_{k=2}^n r_k^{d-2}\PP[\tau^{z_2}_{r_2}<\cdots<\tau^{{z_n}}_{r_{n}}<\infty| \F_{\tau^{z_1}_{r_1}};E_{r_1;\vec{ s}}]-\ha G_{\tau^{z_1}_{r_1}}(z_2,\dots,z_n)\Big|\\
\lesssim & F_{\tau^{z_1}_{r_1}}(z_2,\dots,z_n)  \sum_{j=2}^n\Big( \ha Q^{B_{n-1} }  \frac{r_j}{s_j\wedge R_j}\Big)^{\beta_{n-1}}.
\end{align*}
Suppose $E_{r_1;\theta}$ happens. Let $S=S_{K_{\tau^{z_1}_{r_1}}}$. Since $U_{\tau^{z_1}_{r_1}}\in S$, from Koebe's $1/4$ theorem, we get $ \ha d_k\gtrsim |g'(z_k)|(d_k\wedge \dist(z_k,\gamma[0,\tau^{z_1}_{r_1}])$ and
$$|\ha z_k|\le \dist(g_{\tau^{z_1}_{r_1}}(z_k),S)/\theta\asymp |g'(z_k)|\dist(z_k,\gamma[0,\tau^{z_1}_{r_1}])/\theta,$$
which together imply that
$$ \frac{|\ha z_k|}{\ha d_k}\le \frac{\dist(z_k,\gamma[0,\tau^{z_1}_{r_1}])/\theta}{d_k\wedge \dist(z_k,\gamma[0,\tau^{z_1}_{r_1}])}=\theta^{-1}\Big(\frac{\dist(z_k,\gamma[0,\tau^{z_1}_{r_1}])/\theta}{d_k}\vee 1\Big)\le \theta^{-1}\frac{|z_k|}{d_k},$$%\label{zkdk}\EDE
where the last inequality holds because $d_k,\dist(z_k,\gamma[0,\tau^{z_1}_{r_1}])\le |z_k|$.
So, on the event $E_{r_1;\theta}$, for some constant $C>1$,
\BGE \ha Q\le \frac{C}{\theta} Q.\label{zkdk}\EDE
Thus, if $E_{r_1;\theta}$ happens, and
\BGE Q^{B_{n-1}}\cdot\frac{r_j}{s_j\wedge R_j}<\frac{\theta^{B_{n-1}}\delta_{n-1}}{8C^{B_{n-1}}},\quad 2\le j\le n,\label{Cond3}\EDE then
\begin{align*}
&\Big|\prod_{k=2}^n r_k^{d-2}\PP[\tau^{z_2}_{r_2}<\cdots<\tau^{{z_n}}_{r_{n}}<\infty| \F_{\tau^{z_1}_{r_1}};E_{r_1;\vec{ s}}\cap E_{r_1;\theta}]-\ha G_{\tau^{z_1}_{r_1}}(z_2,\dots,z_n)\Big|\\
\lesssim &  F_{\tau^{z_1}_{r_1}}(z_2,\dots,z_n)\sum_{j=2}^n \Big(\theta^{-B_{n-1} } Q^{B_{n-1} } \frac{r_j}{s_j\wedge R_j}\Big)^{\beta_{n-1}} .
\end{align*}
Now we express
\begin{align*}
& r_1^{2-d}G(z_1)  \EE[{\bf 1}_{E_{r_1;\vec{ s}}\cap E_{r_1;\theta}}\PP[\tau^{z_2}_{r_2}<\cdots<\tau^{{z_n}}_{r_{n}}<\infty| \F_{\tau^{z_1}_{r_1}};E_{r_1;\vec{ s}}\cap E_{r_1;\theta}]|\tau^{z_1}_{r_1}<\infty]\\
=& r_1^{2-d}G(z_1)  \EE[{\bf 1}_{E_{r_1;\vec{ s}}\cap E_{r_1;\theta}}\prod_{k=2}^n r_k^{2-d}\ha G_{\tau^{z_1}_{r_1}}(z_2,\dots,z_n) |\tau^{z_1}_{r_1}<\infty]+e_4^*.
\end{align*}
Using Lemma \ref{FF-F}, we see that, when (\ref{Cond3}) holds,
$$|e_4^*|\lesssim   \prod_{k=1}^n r_k^{2-d} F \prod_{k=2}^n \Big(\frac{|z_k|\wedge |z_k-z_1|}{ s_k }\Big)^\alpha \sum_{j=2}^n \Big(\theta^{-B_{n-1}} Q^{B_{n-1}}\frac{r_j}{s_j\wedge R_j}\Big)^{\beta_{n-1}}.
$$

Next, we express
\begin{align*}
& r_1^{2-d}G(z_1)  \EE[{\bf 1}_{E_{r_1;\vec{ s}}\cap E_{r_1;\theta}}\prod_{k=2}^n r_k^{2-d}\ha G_{\tau^{z_1}_{r_1}}(z_2,\dots,z_n) |\tau^{z_1}_{r_1}<\infty]\\
=&\prod_{k=1}^n r_k^{2-d} G(z_1) \EE[{\bf 1}_{E_{r_1;\vec{ s}}}\ha G_{\tau^{z_1}_{r_1}}(z_2,\dots,z_n) |\tau^{z_1}_{r_1}<\infty]-e_5^*.
\end{align*}
The estimate on $e_5^*$ is the same as that on $e_3^*$ by Lemma \ref{FF-F}.

To simplify the notation, we define for $r>0$ and $\vec s\in \R_+^{n-1}$,
$$\EE_{z_1}^r=\EE[\cdot |\tau^{z_1}_r<\infty];\quad    \ha G_{r; {\vec{ s}}}={\bf 1}_{E_{r;\vec{ s}}}\ha G_{{\tau^{z_1}_{r}}}.$$
So far we have
$$ \PP[\tau^{z_1}_{r_1}<\cdots<\tau^{{z_n}}_{r_{n}}<\infty]=\prod_{k=1}^n r_k^{2-d} G(z_1) \EE_{z_1}^{r_1}[\ha G_{r_1;\vec s}(z_2,\dots,z_n)]+e_1^*+e_2^*+e_3^*+e_4^*-e_5^*.$$

For $R>r>s\ge 0$, define $E_{r,s;R}$ to be the event
\begin{align}
E_{r,s;R}=&\{\gamma[\tau^{z_1}_r,\tau^{z_1}_s]\mbox{ does not intersect any connected component of }\nonumber \\
& \{|z-z_1|=R\}\cap H_{\tau^{z_1}_r}  \mbox{ that separates }z_1\mbox{ from any }z_k, 2\le k\le n\}.\label{ErsR}
\end{align}
Here the bad event $E_{r,s;R}^c$ is the event that between the times visiting smaller circles $\{|z-z_1|=r\}$ and $\{|z-z_1|=s\}$, $\gamma$ crosses some arc on the bigger circle $\{|z-z_1|=R\}$, which is needed in order for $\gamma$ to approaches some $z_j$, $2\le j\le n$, after $\tau^{z_1}_r$.

Fix variables $\eta_1<\eta_2\in (r_1,d_1)$ to be determined later.
% Suppose that \BGE \eta_2+s_k>|z_k-z_1|,\quad 2\le k\le n.\label{Cond4}\EDE
According to whether $E_{\eta_1,r_1;\eta_2}$ occurs, we have the following decomposition:
\[
G(z_1)\EE_{z_1}^{{r_1}}[\ha G_{{r_1};\vec{ s}}(z_2,\dots,z_{n})]=G(z_1)\EE_{z_1}^{{r_1}}[{\bf 1}_{E_{\eta_1,r_1;\eta_2}} \ha G_{{r_1};\vec{ s}}(z_2,\dots,z_{n})]+e_6.
\]	
By Lemma \ref{stayin} (applied to $Z=\{z_j\}$, $2\le j\le n$) and Lemma \ref{FF-F},  we have
\[
0\le e_6\lesssim F  \prod_{j=2}^{n}\Big(\frac{|z_j|\wedge |z_j-z_1|}{s_j}\Big)^{\alpha} \Big(\frac{\eta_1}{\eta_2}\Big)^{\alpha/4} .
\]
Changing the time from $\tau^{z_1}_{r_1}$ to $\tau^{z_1}_{\eta_1}$, we get another error term $e_7$:
\[
G(z_1)\EE_{z_1}^{{r_1}}[{\bf 1}_{E_{\eta_1,{r_1};\eta_2}} \ha G_{{r_1};\vec{ s}}(z_2,\dots,z_{n})]=G(z_1)\EE_{z_1}^{{r_1}}[{\bf 1}_{E_{\eta_1,{r_1};\eta_2}} \ha G_{{\eta_1};\vec{ s}}(z_2,\dots,z_{n})]+e_7
\]

To derive an estimate for $e_7$, we use the following lemma, whose proof is postponed to the end of this subsection.

\begin{lemma}
	There exist constants $B_*>0$ and $\beta_*,\delta_*\in(0,1)$ such that the following holds. Let $0\le a<b$ be such that $z_1\in H_a$, $\dist(z_1,K_a)<|z_j-z_1|$ and $\dist(z_j,K_b)\ge s_j$, $2\le j\le n$. For $2\le j\le n$, let $\rho_j$ be the connected component of $\{|z-z_1|=|z_j-z_1|\}\cap H_a$ that contains $z_j$; and let $\xi_j$ be a crosscuts of $H_a$, which is disjoint from $\rho_j$, and disconnects $\rho_j$ from $K_b\sem K_a$  in $H_a$. Let $d_*=\min_{2\le j\le n} d_{H_a}(\rho_j,\xi_j)$. If
	\BGE Q^{B_*}\cdot  e^{-2\pi  d_*}<\delta_*, \label{Cond-Lemma}\EDE
	then
	%\begin{align*}
	% & G(z_1)|\ha G_{b}(z_2,\dots,z_{n})-\ha G_{a}(z_2,\dots,z_{n})|\\
	% \lesssim & F \prod_{k=2}^n \Big(\frac{|z_k|\wedge |z_k-z_1|}{ s_k }\Big)^\alpha (Q^{B_*}   e^{-2\pi d_*} )^{\beta_*}.
	%\end{align*}
	$$G(z_1)|\ha G_{b}(z_2,\dots,z_{n})-\ha G_{a}(z_2,\dots,z_{n})|
	\lesssim  F \prod_{k=2}^n \Big(\frac{|z_k|\wedge |z_k-z_1|}{ s_k }\Big)^\alpha (Q^{B_*}   e^{-2\pi d_*} )^{\beta_*}.$$
	\label{e4-lemma}
\end{lemma}

We now apply Lemma \ref{e4-lemma} with $a=\tau^{z_1}_{\eta_1}$, $b=\tau^{z_1}_{r_1}$,  and $\xi_k$ being a connected component of $\{|z-z_1|=\eta_2\}\cap H_{\tau^{z_1}_{\eta_1}}$ that separates $z_k$ from $z_1$. By comparison principle of extremal length, we have
$$d_{H_a}(\rho_k,\xi_k)\ge  \log(|z_k-z_1|/\eta_2)/(2\pi)\ge \log(d_1/\eta_2)/(2\pi),\quad 2\le k\le n.$$
Assume that \BGE \eta_2+s_k<|z_k-z_1|,\quad 2\le k\le n.\label{Cond4}\EDE
Then $E_{\eta_1,r_1;\eta_2}\cap E_{r_1;\vec s}=E_{\eta_1,r_1;\eta_2}\cap E_{\eta_1;\vec s}$.
Thus, for some constants $B_*>0$ and $\beta_*,\delta_*\in(0,1)$, if
\BGE Q^{B_*}\cdot  \frac{\eta_2}{d_1}<\delta_*,\label{Cond4'}\EDE
and (\ref{Cond4}) holds, then
$$|e_7|\lesssim F\prod_{k=2}^n \Big(\frac{|z_k|\wedge |z_k-z_1|}{ s_k }\Big)^\alpha \Big(Q^{B_*}   \frac{\eta_2}{d_1}\Big)^{\beta_*}.$$

Removing the restriction of the event $E_{\eta_1,r_1;\eta_2}$, we get another error term $e_8$:
\[
G(z_1)\EE_{z_1}^{{r_1}}[{\bf 1}_{E_{\eta_1,{r_1};\eta_2}} \ha G_{\eta_1;\vec{ s}}(z_2,\dots,z_{n})]=G(z_1)\EE_{z_1}^{{r_1}}[ \ha G_{{\eta_1};\vec{ s}}(z_2,\dots,z_{n})]-e_8.
\]
Here the estimate on $e_8$ is same as that on $e_6$ by Lemmas \ref{stayin} and  \ref{FF-F}.

Changing the probability measure from the conditional chordal $\EE^{r_1}_{z_1}$ to the two-sided radial $\EE_{z_1}^*$, we get another error term $e_9$:
\[
G(z_1)\EE_{z_1}^{{r_1}}[ \ha G_{{\eta_1};\vec{ s}}(z_2,\dots,z_{n})]=G(z_1)\EE^*_{z_1}[ \ha G_{{\eta_1};\vec{ s}}(z_2,\dots,z_{n})]+e_9.
\]
From \cite[Proposition 2.13]{LW} and Lemma \ref{FF-F}, we find that for some constant $\beta_0>0$,
\[
|e_9|\lesssim   F \prod_{k=2}^{n}\Big(\frac{|z_k-z_1|\wedge |z_k|}{s_k}\Big)^{\alpha} \Big(\frac{r_1}{\eta_1}\Big)^{\beta_0}.
\]

Let the event $E_{\eta_1,0;\eta_2}$ be defined by (\ref{ErsR}). We now express
\[
G(z_1) \EE^*_{z_1}[ G_{{\eta_1};\vec{ s}}(z_2,\dots,z_n)]=G(z_1)\EE^*_{z_1}[{\bf 1}_{E_{\eta_1,0;\eta_2}}  \ha G_{{\eta_1};\vec{ s}}(z_2,\dots,z_n)]+e_{10}
\]
Here the estimate on $e_{10}$ is same as that on $e_6$ by Lemmas \ref{stayin} and  \ref{FF-F}.

Changing the time from $\tau^{z_1}_{\eta_1}$ to $\tau^{z_1}_{0}=T_{z_1}$, we get another error term $e_{11}$:
\[
G(z_1)\EE^*_{z_1}[{\bf 1}_{E_{\eta_1,0;\eta_2}}  \ha G_{{\eta_1};\vec{ s}}(z_2,\dots,z_n)]
=G(z_1)\EE^*_{z_1}[{\bf 1}_{E_{\eta_1,0;\eta_2}}  \ha G_{0;\vec{ s}}(z_2,\dots,z_n)]+e_{11}.
\]
If (\ref{Cond4}) holds, then $E_{\eta_1,0;\eta_2}\cap E_{\eta_1;\vec s}=E_{\eta_1,0;\eta_2}\cap E_{0;\vec s}$.
Apply Lemma \ref{e4-lemma} with $a=\tau^{z_1}_{\eta_1}$, $b=\tau^{z_1}_{0}=T_{z_1}$,  and $\xi_k$ being a connected component of $\{|z-z_1|=\eta_2\}\cap H_{\tau^{z_1}_{\eta_1}}$ that separates $z_k$ from $z_1$, we get an estimate on $e_{11}$, which is the same as that on $e_7$, provided that (\ref{Cond4'}) holds. Note that the constants $B_*,\beta_*,\delta_*$ here may be different from those for $e_7$. But by taking the bigger $B_*$ and smaller $\beta_*$ and $\delta_*$, we may make both estimates hold for the same set of constants.

Removing the restriction of the event $E_{\eta_1,0;\eta_2}$, we get another error term $e_{12}$:
\[
G(z_1)\EE^*_{z_1}[{\bf 1}_{E_{\eta_1,0;\eta_2}}  \ha G_{0;\vec{ s}}(z_2,\dots,z_n)]=G(z_1)\EE^*_{z_1}[  \ha G_{0;\vec{ s}}(z_2,\dots,z_n)]-e_{12}.
\]
Here the estimate on $e_{12}$ is same as that $e_6$ by Lemmas \ref{stayin} and  \ref{FF-F}.

Finally, note that
$\ha G_{0;\vec s}={\bf 1}_{E_{0;\vec s}} \ha G_{T_{z_1}}$. Removing the restriction of the event $E_{0;\vec s}$, we get the last error term $e_{13}$:
\[
G(z_1)\EE^*_{z_1}[  \ha G_{0;\vec{ s}}(z_2,\dots,z_n)]=G(z_1)\EE^*_{z_1}[\ha G_{T_{z_1}}(z_2,\dots,z_n)]-e_{13}=\ha G(z_1,\dots,z_n)+e_{13}.
\]
where by Lemma \ref{RZ-Thm3.1-lim}, the estimate on $e_{13}$ is the same as that on $e_1^*/\prod_{k=1}^n r_k^{2-d}$.

At the end, we need to choose the variables $s_2,\dots,s_n$ and $\eta_1,\eta_2,\theta$, and constants $C_n,B_n>0$ and $\beta_n,\delta_n\in(0,1)$, such that if (\ref{Cond-main}) holds, then  (\ref{Cond1},\ref{Cond2},\ref{Cond3},\ref{Cond4},\ref{Cond4'}) all hold, $r_j<R_j/8$, $1\le j\le n$, and the upper bounds for $|e_s|:=|e_s^*|/\prod_{k=1}^n r_k^{2-d}$, $1\le s\le 5$, and $|e_s|$, $6\le s\le 13$, are all bounded above by the RHS of (\ref{RHS-main}).

We take $X\in(0,1)$ to be determined later, and choose $s_2,\dots,s_n$ such that
\BGE \frac{s_j}{|z_j|\wedge |z_j-z_1|}=X,\quad 2\le j\le n.\label{s_j}\EDE
Then we have
\BGE \frac{r_j}{s_j\wedge R_j}=\Big(1\vee \frac{R_j}{s_j}\Big)\cdot\frac{r_j}{R_j}\le X^{-1} \cdot\frac{r_j}{R_j},\quad 2\le j\le n. \label{rsR}\EDE
In the argument below, we assume that (\ref{Cond1},\ref{Cond2},\ref{Cond3},\ref{Cond4},\ref{Cond4'},\ref{s_j},\ref{rsR}) all hold so that we can freely use the estimates we have obtained.

From the estimate on $|e_4^*|$, we get
$$|e_4|\lesssim F Q^{B_{n-1}\beta_{n-1}} X^{-n\alpha -\beta_{n-1}} {\theta} ^{-B_{n-1}\beta_{n-1}} \max_{2\le j\le n} \Big(\frac{r_j}{R_j}\Big)^{\beta_{n-1}}.$$
From the estimates on $e_3^*$ and $e_5^*$, we get
$$|e_s|\lesssim F X^{-n\alpha}\theta^\alpha,\quad s\in\{3,5\}.$$
If we take $\theta$ such that $\theta^{\alpha}=\theta^{-B_{n-1}\beta_{n-1}}\max_{2\le j\le n}(\frac{r_j}{R_j} )^{\beta_{n-1}}$, then we get
$$|e_s|\lesssim FQ^{B_{n-1}\beta_{n-1}}X^{-n\alpha -\beta_{n-1}}\max_{2\le j\le n} \Big(\frac{r_j}{R_j}\Big)^{\frac{\alpha\beta_{n-1}}{\alpha+B_{n-1}\beta_{n-1}}},\quad 3\le s\le 5.$$
Choose $\eta_1$ and $\eta_2$ such that $\frac{r_1}{\eta_1}=\frac{\eta_1}{\eta_2}=\frac{\eta_2}{d_1}$.  Then we find that
$$|e_s|\lesssim F Q^{B_*\beta_*} X^{-n\alpha}\Big(\frac{r_1}{d_1}\Big)^{\frac 13(\frac{\alpha}{4}\wedge \beta_*\wedge \beta_0)},\quad 6\le s\le 12.$$
Since $R_1\le d_1$, combining with the estimate on $e_2^*$, we get
$$|e_s|\lesssim F Q^{B_*\beta_*} X^{-n\alpha}\Big(\frac{r_1}{R_1}\Big)^{\frac 13(\frac{\alpha}{4}\wedge \beta_*\wedge \beta_0)\wedge\beta_1},\quad s\in\{2,6,7,8,9,10,11,12\}.$$
Combining this with the estimates on $|e_s|$, $3\le s\le 5$, we get
$$|e_s|\lesssim F Q^{B_{n-1}\beta_{n-1}+B_*\beta_*}X^{-n\alpha -\beta_{n-1}}\max_{1\le j\le n} \Big(\frac{r_j}{R_j}\Big)^{\beta_{\#}},\quad 2\le s\le 12,$$
where $\beta_{\#}:=\frac 13(\frac{\alpha}{4}\wedge \beta_*\wedge \beta_0)\wedge\beta_1\wedge \frac{\alpha\beta_{n-1}}{\alpha+B_{n-1}\beta_{n-1}}$. Since $|e_1|,|e_{13}|\lesssim F X^\beta$,
if we choose $X$ such that $X^\beta=X^{-n\alpha-\beta_{n-1}}\max_{1\le j\le n}  (\frac{r_j}{R_j} )^{\beta_{\#}}$, then with $\beta_n:=\frac{\beta\beta_\#}{\beta+n\alpha+\beta_{n-1}} $, we get
\BGE |e_s|\lesssim F  Q^{B_{n-1}\beta_{n-1}+B_*\beta_*} \max_{1\le j\le n} \Big(\frac{r_j}{R_j}\Big)^{\beta_n },\quad 1\le s\le 13.\label{es}\EDE

Now we check Conditions (\ref{Cond1},\ref{Cond2},\ref{Cond3},\ref{Cond4},\ref{Cond4'}) and $r_j<R_j/8$, $1\le j\le n$. Clearly, (\ref{Cond3}) implies (\ref{Cond1}). The LHS of (\ref{Cond4'}) equals to $Q^{B_*} (\frac{r_1}{d_1})^{1/3}\le  Q^{B_*} (\frac{r_1}{R_1})^{1/3}$, and so it holds if $Q^{3B_*} \frac{r_1}{R_1}<\delta_*^3$. Thus, (\ref{Cond2}) and (\ref{Cond4'}) both hold if $Q^{3B_*} \frac{r_1}{R_1}<\delta_*^3\wedge \delta_1$. Condition (\ref{Cond4}) holds if $\eta_2<\frac{d_1}2$ and $s_k<\frac 12|z_k-z_1|\wedge |z_k|$, which are equivalent to $\frac{r_1}{d_1}<\frac 18$ and $X<\frac 12$, respectively, which further follow from
$$\max_{1\le j\le n}   \frac{r_j}{R_j}  <  \Big( \frac 12\Big)^{3+  \frac{\beta+n\alpha +\beta_{n-1}}{\beta_\#}}.$$
From (\ref{rsR}) and the choices of $X$ and $\theta$, we see that (\ref{Cond3}) follows from
$$Q^{B_{n-1}} \max_{1\le j\le n}  \frac{r_j}{R_j} <\frac{X\theta^{B_{n-1}}\delta_{n-1}}{8C^{B_{n-1}}}
=\frac{ \delta_{n-1}}{8C^{B_{n-1}}}\max_{1\le j\le n} \Big(\frac{r_j}{R_j}\Big)^{\frac{ \beta_\#}{\beta+n\alpha+\beta_{n-1}}+ \frac{B_{n-1}\beta_{n-1}}{\alpha+B_{n-1}\beta_{n-1}} }. $$
Let $\beta_\&=1-\frac{ \beta_\#}{\beta+n\alpha+\beta_{n-1}}- \frac{B_{n-1}\beta_{n-1}}{\alpha+B_{n-1}\beta_{n-1}} $.
Since $\beta_\#\le \frac{\alpha\beta_{n-1}}{\alpha+B_{n-1}\beta_{n-1}}$, we get $\beta_\&>0$. So (\ref{Cond1}) and (\ref{Cond3}) hold if $Q^{B_{n-1}/\beta_\&}  \max_{1\le j\le n}  \frac{r_j}{R_j} <(\frac{ \delta_{n-1}}{8C^{B_{n-1}}})^{1/\beta_\&}$. Thus, (\ref{Cond1},\ref{Cond2},\ref{Cond3},\ref{Cond4},\ref{Cond4'}) all hold if
$$Q^{3B_*+\frac{B_{n-1}}{\beta_\&}}  \max_{1\le j\le n}  \frac{r_j}{R_j}<\delta_n,$$
where $\delta_n:=\delta_*^3\wedge \delta_1\wedge   ( \frac 12 )^{3+  \frac{\beta+n\alpha +\beta_{n-1}}{\beta_\#}}\wedge (\frac{ \delta_{n-1}}{8C^{B_{n-1}}} )^{\frac 1 {\beta_\&}}$. Combining this with (\ref{es}), we see that, if we set $B_n=3B_*+\frac{B_{n-1}}{\beta_\&}+\frac{B_{n-1}\beta_{n-1}+B_*\beta_*}{\beta_n}$, then whenever (\ref{Cond-main}) holds, (\ref{Cond1},\ref{Cond2},\ref{Cond3},\ref{Cond4},\ref{Cond4'}) and $r_j<R_j/8$, $1\le j\le n$, all hold, and   the upper bounds for $|e_s|$, $1\le s\le 13$, are all bounded above by the RHS of (\ref{RHS-main}). It remains to prove Lemma \ref{e4-lemma} to finish this subsection.

\begin{proof}[Proof of Lemma \ref{e4-lemma}]
Since $K_{a}\subset K_{b}$ we also have $\dist(z_j,K_a)\ge s_j$, $2\le j\le n$.
Let $K=g_{a}(K_{b}\sem K_{a})$. Then $K$ is an $\HH$-hull, and $g_{b}=g_K\circ g_{a}$. Since $g_{a}(\gamma(a))=  U_{a}$, we have $U_{a}\in\lin K\cap\R$. Since $g_{b}(\gamma(b))=  U_{b}$, we have $U_{b}\in S_K$. Let $r_K=\sup\{|z-U_{a}|:z\in K\}$. From Lemma \ref{small}, we get  $ S_K\subset [U_{a}-2r_K,U_{a}+2r_K]$. Thus, $|U_{b}-U_{a}|\le 2 r_K$.

Define $z^a_j=g_{a}(z_j)$, $\rho^a_j=g_a(\rho_j)$, $\xi^a_j=g_a(\xi_j)$, $z^b_j=g_b(z_j)$, $\rho^b_j=g_b(\rho_j)$, $2\le j\le n$. Then $\rho^a_j,\rho^b_j,\xi^a_j$ are crosscuts of $\HH$, $z^a_j\in \rho^a_j$, $z^b_j\in\rho^b_j$, and $\xi^a_j$ disconnects $K$ from $\rho^a_j$.
By conformal invariance of extremal distance, we get
$$d_{\HH}( \rho^b_j,S_K)=d_{\HH}( \rho^a_j,K)=d_{H_{a}}(\rho_j,K_{b}\sem K_{a})\ge d_{H_a}(\rho_j,\xi_j)\ge d_*.$$
Applying Lemma \ref{lem-extremal2} to $\lin{\rho^a_j}$ and $\lin K$, and to $\rho^b_j$ and $S_K$, respectively, we get
\BGE \Big(\frac{\diam(\rho^a_j)}{\dist(\rho^a_j,K)}\wedge 1\Big)\cdot  \Big(\frac{\diam(K)}{\dist(\rho^a_j,K)}\wedge 1\Big)
\le 144 e^{-\pi d_*},\quad 2\le j\le n;\label{product}\EDE
\BGE \Big(\frac{\diam(\rho^b_j)}{\dist(\rho^b_j,S_K)}\wedge 1\Big)\cdot  \Big(\frac{\diam(S_K)}{\dist(\rho^b_j,S_K)}\wedge 1\Big)
\le 144 e^{-\pi d_*},\quad 2\le j\le n.\label{productb}\EDE

Fix a variable $\phi\in(0,1)$ to be determined later. Define the event $E_{a;\phi}$ using (\ref{Erth}) but with $\tau^{z_1}_r$ replaced by $a$ (instead of $\tau^{z_1}_a$). First, suppose $E_{a;\phi}$ does not occur. Since $\dist(z_j,K_a)\ge s_j$, $2\le j\le n$, from Lemma \ref{FF-F} we get
\BGE G(z_1)\ha G_{a}(z_2,\dots,z_n)\lesssim F \prod_{k=2}^n \Big(\frac{|z_k|\wedge |z_k-z_1|}{ s_k }\Big)^\alpha \phi^\alpha.\label{GG-eta}\EDE
Fix some $j\in\{2,\dots,n\}$ for a while. Applying Koebe's $1/4$ theorem, we get
$$\dist(z^b_j,S_{K_{b}})\asymp |g_{b}'(z_j)|\dist(z_j,K_{b})\le|g_{b}'(z_j)|\dist(z_j,K_{a})$$
$$= |g_K'(z^a_j)||g_{a}'(z_j)|\dist(z_j,K_{a})\asymp |g_K'(z^a_j)| \dist(z^a_j,S_{K_{a}}) $$
and
$$|z^b_j-U_{b}|\ge \dist(z^b_j,S_K)\asymp |g_K'(z^a_j)| \dist(z^a_j,K).$$
Now we consider two cases.

Case 1. $ \diam(S_K)\le \dist(z^b_j,S_K)/ 4$. In this case, since $z^a_j=f_K(z^b_j)$, applying Lemma \ref{SK-lemma}, we get $\dist(z^a_j,K)\ge 2 \diam(K)$, which implies that $\dist(z^a_j,K)\asymp |z^a_j-U_a|$ since $U_a\in \lin K$. From the above two displayed formulas, we get $\frac{\dist(z^b_j,S_{K_{b}})}{|z^b_j-U_{b}|}\lesssim \frac{\dist(z^a_j,S_{K_{a}})}{|z^a_j-U_{a}|}$.

Case 2. $ \diam(S_K)\ge \dist(z^b_j,S_K)/ 4$. From (\ref{productb}), we have
\BGE \frac{\diam(\rho^b_j)}{\dist(\rho^b_j,S_K)}\le 576e^{-\pi d_*},\label{dd}\EDE
if
\BGE 144e^{-\pi d_*}< 1/4.\label{Cond4.5}\EDE
Since $\dist(z_1,K_a)<|z_j-z_1|$, and $\rho_j\subset \{|z-z_1|=|z_j-z_1|\}$, we see that either  $\rho_j$ disconnects $K_b$ from $\infty$, or ${\rho_j}$ touches $K_b$. The former case implies that $\diam(\rho^b_j)\ge \dist(\rho^b_j,S_K)$ because $\rho^b_j$ disconnects $K$ from $\infty$, which is impossible by (\ref{dd}) if (\ref{Cond4.5}) holds. In the latter case, $\rho^b_j:=g_b(\rho_j)$ touches $S_{K_b}$, and so $\dist(z^b_j,S_{K_b})\le \diam(\rho^b_j)$. On the other hand, since $U_b\in S_K$ and $z^b_j\in \rho^b_j$, we get $|z^b_j-U_b|\ge \dist(\rho^b_j,S_K)$. Thus by (\ref{dd}), we have  $\dist(z^b_j,S_{K_b})\le 576e^{-\pi d_*} |z^b_j-U_b|$ if (\ref{Cond4.5}) holds.

Combining Case 1 with Case 2, we see that, if (\ref{Cond4.5}) holds and $E_{a;\phi}$ does not occur, then for some $2\le j\le n$,
$\dist(z^b_j,S_{K_{b}})\lesssim (\phi+e^{-\pi d_*}) |z^b_j-U_{b}|$.
This together with Lemmas \ref{FF-F} and that $\dist(z_j,K_b)\ge s_j$, $2\le j\le n$, implies that
\BGE G(z_1)\ha G_{b}(z_2,\dots,z_n)\lesssim F  \prod_{k=2}^n \Big(\frac{|z_k|\wedge |z_k-z_1|}{ s_k }\Big)^\alpha(\phi^\alpha+e^{-\alpha\pi d_*}).\label{GG-r}\EDE

Now suppose that $E_{a;\phi}$ occurs.
Since $z^a_j\in\rho^a_j$ and $U_{a}\in\lin K$, we have $|z^a_j-U_{a}|\ge \dist(\rho^a_j,K)$. We claim that $\diam(\rho^a_j)\ge \dist(z^a_j,S_{K_{a}})$. If this is not true, then the region bounded by $\rho^a_j$ in $\HH$ is disjoint from $S_{K_{a}}$, which implies that $\rho_j=g_{a}^{-1}(\rho^a_j)$ is also a crosscut of $\HH$, and the region bounded by $\rho_j$ in $\HH$ is disjoint from $K_{a}$. Since $\rho_j$ is an arc on the circle $\{|z-z_1|=|z_j-z_1|\}$, this would imply that $\dist(z_1,K_{a})\ge |z_j-z_1|$, which is a contradiction. So the claim is proved. Thus, we have
\BGE\frac{\diam(\rho^a_j)}{\dist(\rho^a_j,K)}\ge \frac{\dist(z^a_j,S_{K_{a}})}{|z^a_j-U_{a}|}\ge \phi.\label{Ejtheta}\EDE

From (\ref{product}), (\ref{Ejtheta}), $r_K\le \diam(K)$ and $z^a_j\in\rho^a_j$, we see that
\BGE \frac{r_K}{\dist(z^a_j,K)}\le \frac{144}\phi e^{-\pi d_*},\quad 2\le j\le n,\label{rK<}\EDE
as long as the RHS is less than $1$. Applying Lemma \ref{lemma-gK} with $x_0=U_{a}$, $r=r_K$, and $z=z^a_j$, from $z^b_j=g_K(z^a_j)$, we see that, if
\BGE \frac{144}\phi e^{-\pi d_*}<\frac 15,\label{Cond5}\EDE then
\BGE |z^b_j-z^a_j|\le  r_K,\quad \frac{|\Imm z^b_j-\Imm z^a_j|}{\Imm z^a_j}\le  4\Big(\frac{r_K}{\dist(z^a_j,K)}\Big)^2;\label{z-z1}\EDE
\BGE |g_K'(z^a_j)-1|\le  5\Big(\frac{r_K}{\dist(z^a_j,K)}\Big)^2.\label{g'-1}\EDE
Let $\ha z^a_j=z^a_j-U_{a}$ and $\ha z^b_j=z^b_j-U_{b}$, $2\le j\le n$. Since $|U_{b}-U_{a}|\le 2r_K$, from (\ref{z-z1}), we find that, if (\ref{Cond5}) holds, then
\BGE \frac{|\ha z^b_j-\ha z^a_j|}{|\ha z^a_j|}\le  3\frac{r_K}{\dist(z^a_j,K)},\quad \frac{|\Imm \ha z^b_j-\Imm \ha z^a_j|}{\Imm \ha z^a_j}\le 4\Big(\frac{r_K}{\dist(z^a_j,K)}\Big)^2.\label{z-z2}\EDE
By definition, we have
\begin{align*}
\ha G_{a}(z_2,\dots,z_{n})=&\prod_{j=2}^n |  g_{a}'(z_j)|^{2-d} \ha G(\ha z^a_2,\dots,\ha z^a_n);\\
\ha G_{b}(z_2,\dots,z_{n})=&\prod_{j=2}^n |  g_{b}'(z_j)|^{2-d} \ha G(\ha z^b_2,\dots,\ha z^b_n)\\
= &\prod_{j=2}^n |  g_K'(z^a_j)|^{2-d} \prod_{j=2}^n |  g_{a}'(z_j)|^{2-d} \ha G(\ha z^b_2,\dots,\ha z^b_n).
\end{align*}
Define $\ha G_{a,b}(z_2,\dots,z_{n})=\prod_{j=2}^n |  g_{a}'(z_j)|^{2-d} \ha G(\ha z^b_2,\dots,\ha z^b_n)$.
From (\ref{g'-1}) we see that there is a constant $\delta\in(0,1)$ (depending on $n$) such that, if
\BGE \max_{2\le j\le n} \frac{r_K}{\dist(z^a_j,K)} <\delta,\label{Cond6}\EDE
then
\BGE|\ha G_{b}(z_2,\dots,z_{n})-\ha G_{a,b}(z_2,\dots,z_{n})|
\lesssim \Big(\max_{2\le j\le n}\frac{r_K}{\dist(z^a_j,K)}\Big)^2 \ha G_{a,b}(z_2,\dots,z_{n}).\label{Gr-Gr'}\EDE
Define $\ha d_k$, $2\le k\le n$, and $\ha Q$ using (\ref{ldyR}) and (\ref{Q}) for the $(n-1)$ points $\ha z_2^a,\dots,\ha z_n^a$.
Since Theorem \ref{main2} holds for $(n-1)$ points, from (\ref{z-z2}) we see that, for some constants $B_{n-1}>0$ and $\beta_{n-1},\delta_{n-1}\in(0,1)$, if
$$\ha Q^{B_{n-1}}\cdot \frac{|\ha z^b_j-\ha z^a_j|}{\ha d_j}<\delta_{n-1},\quad \frac{|\Imm \ha z^b_j-\Imm \ha z^a_j|}{\Imm \ha z^a_j}<\delta_{n-1},$$
then
\begin{align*}
&|\ha G_{a,b}(z_2,\dots,z_n)-\ha G_{a}(z_2,\dots,z_n)|/F_{a}(z_2,\dots,z_n)\\
\lesssim & \sum_{j=2}^n \Big(\ha Q^{B_{n-1} }  \frac{|\ha z^b_j-\ha z^a_j|}{\ha d_j}\Big)^{\beta_{n-1}}+\Big(\frac{|\Imm \ha z^b_j-\Imm \ha z^a_j|}{\Imm \ha z^a_j}\Big)^{\beta_{n-1}}.
\end{align*}
Since $E_{a;\phi}$ occurs, (\ref{zkdk}) holds here with $\phi$ in place of $\theta$ by the same argument.
Let $B_0=B_{n-1}+1$. Then, for some constant $C>1$, if
\BGE Q^{B_0}\cdot \frac{|\ha z^b_j-\ha z^a_j|}{|\ha z^a_j|}<\frac{\phi^{B_0} \delta_{n-1}}{C^{B_0}},\quad \frac{|\Imm \ha z^b_j-\Imm \ha z^a_j|}{\Imm \ha z^a_j}<\delta_{n-1},\label{Cond7}\EDE
then
\begin{align}
&|\ha G_{a,b}(z_2,\dots,z_n)-\ha G_{a}(z_2,\dots,z_n)|/F_{a}(z_2,\dots,z_n)\nonumber\\
\lesssim & \sum_{j=2}^n   \Big(\phi^{-B_0 }Q^{B_0 } \frac{|\ha z^b_j-\ha z^a_j|}{|\ha z^a_j|}\Big)^{\beta_{n-1}}+\Big(\frac{|\Imm \ha z^b_j-\Imm \ha z^a_j|}{\Imm \ha z^a_j}\Big)^{\beta_{n-1}} .
\label{Gr'-Geta}
\end{align}
From (\ref{Cond7}) we see that the RHS of (\ref{Gr'-Geta}) is bounded above by a constant. Since $\ha G_a\lesssim F_a$ by induction hypothesis, we get $\ha G_{a,b}\lesssim F_a$ as well. From  (\ref{Gr-Gr'}) and (\ref{Gr'-Geta}),  we see that if (\ref{Cond6}) and (\ref{Cond7} ) both hold, then
\begin{align*}
&|\ha G_{b}(z_2,\dots,z_n)-\ha G_{a}(z_2,\dots,z_n)|/F_{a}(z_2,\dots,z_n)\nonumber \\
\lesssim & \Big(\max_{2\le j\le n}\frac{r_K}{\dist(z^a_j,K)}\Big)^2+ \sum_{j=2}^n \Big(\phi^{-B_0 }Q^{B_0 } \frac{|\ha z^b_j-\ha z^a_j|}{|\ha z^a_j|}\Big)^{\beta_{n-1}}+\Big(\frac{|\Imm \ha z^b_j-\Imm \ha z^a_j|}{\Imm \ha z^a_j}\Big)^{\beta_{n-1}} \\
\lesssim &\phi^{-2} e^{-2\pi d_*}+
(\phi^{-B_0 -1} Q^{B_0} e^{-\pi d_*} )^{\beta_{n-1}}+ (\phi^{-2} e^{-2\pi d_*} )^{\beta_{n-1}}\lesssim  (\phi^{-B_0 -1} Q^{B_0} e^{-\pi d_*} )^{\beta_{n-1}}  .%\label{Gr-Geta}
\end{align*}
where the second last inequality follows from (\ref{rK<}), (\ref{z-z2}), and that $|z_j-z_1|\ge d_1$, and the last inequality holds provided that
\BGE \phi^{-2}e^{-2\pi d_*}<1.\label{Cond8}\EDE
Since $\dist(z_j,K_a)\ge s_j$, $2\le j\le n$, from Lemma \ref{FF-F}, we get
%\begin{align*}
%&G(z_1)|\ha G_{b}(z_2,\dots,z_n)-\ha G_{a}(z_2,\dots,z_n)|\\
%\lesssim & F \prod_{k=2}^n \Big(\frac{|z_k|\wedge |z_k-z_1|}{ s_k }\Big)^\alpha  (\phi^{-B_0 -1} Q^{B_0} e^{-\pi d_*} )^{\beta_{n-1}}.
%\end{align*}
$$G(z_1)|\ha G_{b}(z_2,\dots,z_n)-\ha G_{a}(z_2,\dots,z_n)|
\lesssim   F \prod_{k=2}^n \Big(\frac{|z_k|\wedge |z_k-z_1|}{ s_k }\Big)^\alpha  (\phi^{-B_0 -1} Q^{B_0} e^{-\pi d_*} )^{\beta_{n-1}}.$$

Combining the above with (\ref{GG-eta},\ref{GG-r}), which holds when $E_{a;\phi}$ does not occur, we find that, as long as Conditions (\ref{Cond4.5},\ref{Cond5},\ref{Cond6},\ref{Cond7},\ref{Cond8}) all hold, no matter whether $E_{a;\phi}$ happens, we have
\begin{align*}
&G(z_1)|\ha G_{b}(z_2,\dots,z_n)-\ha G_{a}(z_2,\dots,z_n)|\\
\lesssim & F \prod_{k=2}^n \Big(\frac{|z_k|\wedge |z_k-z_1|}{ s_k }\Big)^\alpha  [e^{-\alpha\pi d_*}+\phi^\alpha+ (\phi^{-B_0 -1} Q^{B_0}e^{-\pi d_*} )^{\beta_{n-1}} ].
\end{align*}

Finally, we may find constants $b*,B_*>0$ and $\beta_*,\delta_*\in(0,1)$, such that, with $\phi=e^{-b_*\pi d_*}$, if (\ref{Cond-Lemma}) holds, then (\ref{Cond4.5},\ref{Cond5},\ref{Cond6},\ref{Cond7},\ref{Cond8}) all hold, and the quantity in the square bracket of the above displayed formula is bounded above by a constant times $(Q^{B_*}  e^{-\pi d_*} )^{\beta_*}$. This is analogous to the argument after the estimate on $e_{13}$ and before this proof.
\end{proof}

\subsection{Continuity of Green's functions}
We work on the inductive step for Theorem \ref{main2} in this subsection. Suppose $ z_1',\dots, z_n'$ are distinct points in $\HH$ such that $z_j'$ is close to $z_j$, $1\le j\le n$. The strategy of the proof is similar to that of Theorem \ref{main}. We will transform $\ha G(z_1',\dots,z_n')$ into $\ha G(z_1,\dots,z_n)$ in a number of steps. In each step we get an error term, and we define a (good) event such that we have a good control of the error when the event happens, and the complement of the event (bad event) has small probability. These events actually have already appeared in the proof of Theorem \ref{main}. In addition, we find that it suffices to prove two special cases, which are the two lemmas below.

\begin{lemma}
	With the induction hypothesis, Theorem \ref{main2} holds if $z_1'=z_1$. \label{main2A}
\end{lemma}

\begin{lemma}
	With the induction hypothesis, Theorem \ref{main2} holds if $z_k'=z_k$, $2\le k\le n$. \label{main2B}
\end{lemma}

Before proving these lemmas, we first show how they can be used to prove the inductive step for Theorem \ref{main2} from $n-1$ to $n$. We have
\begin{align*}
&|\ha G(z_1',z_2',\dots,z_n')-\ha G(z_1,z_2,\dots,z_n)|\\ \le &|\ha G(z_1',z_2',\dots,z_n')-\ha G(z_1',z_2,\dots,z_n)|+|\ha G(z_1',z_2,\dots,z_n)-\ha G(z_1,z_2,\dots,z_n)|=:I_1+I_2.
\end{align*}
By Lemma \ref{main2B}, for some constants $B_n^{(2)}>0$ and $\beta_n^{(2)},\delta_n^{(2)}\in(0,1)$, $I_2$ is bounded by the RHS of (\ref{RHS-main2}) when (\ref{Cond-main2}) holds for $j=1$. We need to use Lemma \ref{main2A} to estimate $I_1$ with the assumption that $z_1'$ is close to $z_1$ but may not equal to $z_1$. Define $d_k'$ and $l_k'$, $1\le k\le n$, $Q'$ and $F'$ using (\ref{ldyR}) and (\ref{Q}) for the $n$ points $z_1',z_1,\dots,z_n$. From Lemma \ref{main2A}, we know that, for some constants $B_n'>0$ and $\beta_n',\delta_n'\in(0,1)$, $I_1$ is bounded by the RHS of (\ref{RHS-main2}) when (\ref{Cond-main2}) holds for $2\le j\le n$, with $d_j'$, $Q'$ and $F'$ in place of $d_j$, $Q$ and $F$, respectively. Suppose
\BGE  {|z_1'-z_1|} <d_1/2,\quad \Imm z_1'\asymp \Imm z_1.\label{Cond2AB}\EDE
Then we have $|z_1'|\asymp |z_1$ and $|z_k-z_1'|\asymp|z_k-z_1|$, $2\le k\le n$, which imply that
$d_k'\asymp d_k$ and $l_k'\asymp l_k$, $1\le k\le n$, which in turn imply that $Q'\asymp Q$ and $F'\asymp F$.

Thus, there are constants $B_n^{(1)}>0$ and $\beta_n^{(1)},\delta_n^{(1)}\in(0,1)$, such that $I_1$ is bounded by the RHS of (\ref{RHS-main2}) when (\ref{Cond-main2}) holds for $2\le j\le n$. Finally, taking $B_n=B_n^{(1)}\vee B_n^{(2)}$, $\beta_n=\beta_n^{(1)}\wedge \beta_n^{(2)}$ and $\delta_n=\delta_n^{(1)}\wedge \delta_n^{(2)}\wedge 1/8$, we then finish the inductive step for Theorem \ref{main2} from $n-1$ to $n$.

\begin{proof}[Proof of Lemma \ref{main2A}.]
	Define $E_{0;\vec s}$ and $E_{0;\theta}$ using (\ref{Ers}) and (\ref{Erth})  for $z_1,z_2,\dots,z_n$; and define $E_{0;\vec s}'$ and $E_{0;\theta}'$ using (\ref{Ers}) and (\ref{Erth})  for $z_1,z_2',\dots,z_n'$. Let $T=T_{z_1}=\tau^{z_1}_0$.
	
	Fix  $\vec s=( s_2,\dots, s_{n})$ with $s_j\in(|z_j'-z_j|,|z_j-z_1|\wedge |z_j|)$ and $\theta\in(0,1)$ being variables to be determined later. From Koebe's $1/4$ theorem and distortion theorem, we see that there is a constant $\delta\in(0,1/10)$ such that,  if
	\BGE \frac{|z_j'-z_j|}{s_j}<\delta, \quad 2\le j\le n,\label{CondA}\EDE
	and $E_{0;\vec s}$ occurs, then $$4|g_T(z_j')-g_T(z_j)|<\dist(g_T(z_j),S_{K_T})\le |g_T(z_j)-U_T|,\quad 2\le j\le n,$$
	which implies that
	\BGE E_{0;\vec s}\cap E_{0;2\theta}'\subset E_{0;\vec s}\cap E_{0;\theta}\subset E_{0;\vec s}\cap E_{0;\theta/2}'.\label{E-Etheta}\EDE
	Since $\delta<1/2$,  (\ref{CondA}) clearly implies that
	\BGE E_{0;2\vec s}'\subset E_{0;\vec s}\subset E_{0;\vec s/2}'.\label{E-Es}\EDE
	
	Suppose (\ref{CondA}) holds. First, we express
	$$\ha G(z_1,z_2,\dots,z_n)=G(z_1)\EE_{z_1}^*[\ha G_{T}(z_2,\dots,z_n)]=G(z_1)\EE_{z_1}^*[{\bf 1}_{E_{0;\vec s}}\ha G_{T}(z_2,\dots,z_n)]+e_1;$$
	$$\ha G(z_1,z_2'\dots,z_n')=G(z_1)\EE_{z_1}^*[\ha G_{T}(z_2',\dots,z_n')]=G(z_1)\EE_{z_1}^*[{\bf 1}_{E_{0;\vec s}}\ha G_{T}(z_2',\dots,z_n')]+e_1'.$$
	Using Lemma \ref{RZ-Thm3.1-lim} and (\ref{E-Es}), we find that there is a constant $\beta>0$ such that
	$$0\le e_1,e_1'\lesssim F \sum_{j= 2}^n \Big(\frac{ s_j}{|z_j|\wedge |z_j-z_1|}\Big)^{\beta } .$$
	
	Second, we express
	$$G(z_1)\EE_{z_1}^*[{\bf 1}_{E_{0;\vec s}}\ha G_{T}(z_2,\dots,z_n)]=G(z_1)\EE_{z_1}^*[{\bf 1}_{E_{0;\vec s}\cap E_{0;\theta}}\ha G_{T}(z_2,\dots,z_n)]+e_2;$$
	$$G(z_1)\EE_{z_1}^*[{\bf 1}_{E_{0;\vec s}}\ha G_{T}(z_2',\dots,z_n')]=G(z_1)\EE_{z_1}^*[{\bf 1}_{E_{0;\vec s}\cap E_{0;\theta}}\ha G_{T}(z_2',\dots,z_n')]+e_2'.$$
	From Lemma \ref{FF-F},  (\ref{E-Etheta},\ref{E-Es}), and that $\ha G\lesssim F$ holds for $(n-1)$ points, we get
	$$0\le e_2,e_2'\lesssim F  \prod_{j=2}^n \Big(\frac{|z_j|\wedge |z_j-z_1|}{ s_j}\Big)^{\alpha }\theta^\alpha.$$

	Now suppose $E_{0;\vec s}$ and $E_{0;\theta}$ both occur.
	Let $Z=Z_T$, $\ha z_j=Z(z_j)$ and $\ha z_j'=Z(z_j')$, $2\le j\le n$.
	By definition, we have
	\begin{align*}
	\ha G_{T}(z_2,\dots,z_{n})=&\prod_{j=2}^n |  g_{T}'(z_j)|^{2-d} \ha G(\ha z _2,\dots,\ha z _n);\\
	\ha G_{T}(z_2',\dots,z_{n}')=&\prod_{j=2}^n |  g_{T}'(z_j')|^{2-d} \ha G(\ha z_2',\dots,\ha z_n').
	\end{align*}
	Define $\ha G_{T}'(z_2',\dots,z_{n}')=\prod_{j=2}^n |  g_{T}'(z_j)|^{2-d} \ha G(\ha z_2',\dots,\ha z_n')$. From Koebe's distortion theorem, there is a constant $\delta'\in(0,1)$ such that, if
	\BGE \frac{|z_j'-z_j|}{s_j}<\delta', \quad 2\le j\le n,\label{CondB}\EDE
	then
	\BGE|\ha G_T(z_2',\dots,z_{n}')-\ha G_T'(z_2',\dots,z_{n}')|
	\lesssim \sum_{j=2}^n \frac{|z_j'-z_j|}{s_j} \cdot \ha G_{T}'(z_2',\dots,z_{n}').\label{GT-GT'}\EDE

	Define $\ha d_k$, $2\le k\le n$, and $\ha Q$ using (\ref{ldyR}) and (\ref{Q}) for the $(n-1)$ points $\ha z_2,\dots,\ha z_n$.
	Since Theorem \ref{main2} holds for $(n-1)$ points, we see that, for some constants $B_{n-1}>0$ and $\beta_{n-1},\delta_{n-1}\in(0,1)$, if
	$$\ha Q^{B_{n-1}}\cdot \frac{|\ha z_j'-\ha z_j|}{\ha d_j}<\delta_{n-1},\quad \frac{|\Imm \ha z_j'-\Imm \ha z_j|}{\Imm \ha z_j}<\delta_{n-1},$$
	then
	\begin{align*}
	&|\ha G(\ha z_2',\dots,\ha z_n')-\ha G(\ha z_2,\dots,\ha z_n)|/F(\ha z_2,\dots,\ha z_n)\\
	\lesssim & \sum_{j=2}^n \Big(\ha Q^{B_{n-1} }  \frac{|\ha z_j'-\ha z_j|}{\ha d_j}\Big)^{\beta_{n-1}}+\Big(\frac{|\Imm \ha z_j'-\Imm \ha z_j|}{\Imm \ha z_j}\Big)^{\beta_{n-1}}.
	\end{align*}
	If $E_{0;\theta}$ occurs,(\ref{zkdk}) holds here by the same argument. Let $B_0=B_{n-1}+1$. Then, for some constant $C>1$, if
	\BGE Q^{B_0}\cdot \frac{|\ha z_j'-\ha z_j|}{|\ha z_j|}<\frac{\theta^{B_0} \delta_{n-1}}{C^{B_0}},\quad \frac{|\Imm \ha z_j'-\Imm \ha z_j|}{\Imm \ha z_j}<\delta_{n-1},\label{Cond-temp}\EDE
	then
	\begin{align}
	&|\ha G_T'(z_2',\dots,z_n')-\ha G_T(z_2,\dots,z_n)|/F_{T}(z_2,\dots,z_n)\nonumber\\
	\lesssim & \sum_{j=2}^n \Big( \Big(\theta^{-B_0 }Q^{B_0 } \frac{|\ha z_j'-\ha z_j|}{|\ha z_j|}\Big)^{\beta_{n-1}}+\Big(\frac{|\Imm \ha z_j'-\Imm \ha z_j|}{\Imm \ha z_j}\Big)^{\beta_{n-1}} \Big).
	\label{Gr'-Geta-c}
	\end{align}
	
	From (\ref{Cond-temp}) we see that the RHS of (\ref{Gr'-Geta-c}) is bounded above by a constant. Since $\ha G_T\lesssim F_{T}$, we get $\ha G_T'(z_2',\dots,z_n')\lesssim F_{T}(z_2,\dots,z_n)$. From (\ref{GT-GT'}) and (\ref{Gr'-Geta-c}), we see that, if (\ref{CondB}) and (\ref{Cond-temp}) both hold, then
	\begin{align}
	&|\ha G_T(z_2',\dots,z_n')-\ha G_T(z_2,\dots,z_n)|/F_T(z_2,\dots,z_n) \nonumber\\
	\lesssim& \sum_{j=2}^n \Big(\frac{|z_j'-z_j|}{s_j}+ \Big(\theta^{-B_0 }Q^{B_0 } \frac{|\ha z_j'-\ha z_j|}{|\ha z_j|}\Big)^{\beta_{n-1}}+\Big(\frac{|\Imm \ha z_j'-\Imm \ha z_j|}{\Imm \ha z_j}\Big)^{\beta_{n-1}} \Big).\label{G-G/F}
	\end{align}
	Applying Lemma \ref{gTvar} to $K=K_T$ and using $Z=g_T-U_T$ and $U_T\in S_{K_T}$, we find that, if (\ref{CondA}) holds, then for $2\le j\le n$,
	\BGE \frac{|\ha z_j'-\ha z_j|}{|\ha z_j|}\lesssim \frac{|z_j'-z_j|}{s_j},\quad \frac{|\Imm \ha z_j'-\Imm \ha z_j|}{\Imm \ha z_j}\lesssim \frac{|\Imm z_j'-\Imm z_j|}{\Imm z_j}+\Big(\frac{|z_j'-z_j|}{s_j}\Big)^{1/2}.\label{Imz<}\EDE
	Thus, there is a constant $C_0>0$, such that if
	\BGE Q^{B_0}\cdot \frac{| z_j'- z_j|}{s_j}<\frac{\theta^{B_0} \delta_{n-1}^2}{C_0},\quad \frac{|\Imm   z_j'-\Imm   z_j|}{\Imm   z_j}<\frac{\delta_{n-1}}{C_0},\label{CondC}\EDE
	then (\ref{Cond-temp}) holds.
	
	Now we express
	$$G(z_1)\EE_{z_1}^*[{\bf 1}_{E_{0;\vec s}\cap E_{0;\theta}}\ha G_{T}(z_2',\dots,z_n')]=G(z_1)\EE_{z_1}^*[{\bf 1}_{E_{0;\vec s}\cap E_{0;\theta}}\ha G_{T}(z_2,\dots,z_n)]+e_3.$$
	From (\ref{G-G/F},\ref{Imz<}) and Lemma \ref{FF-F}, we find that, if (\ref{CondA},\ref{CondB},\ref{CondC}) all hold, then
	%\begin{align*}
	%|e_3|\lesssim &F  \prod_{j=2}^n \Big(\frac{|z_j|\wedge |z_j-z_1|}{ s_j}\Big)^{\alpha } \cdot\\
	%& \cdot\sum_{j=2}^n \Big(\Big(\theta^{-B_0 }Q^{B_0 } \frac{|z_j'-z_j|}{s_j}\Big)^{\beta_{n-1}/2}+\Big(\frac{|\Imm z_j'-\Imm z_j|}{\Imm z_j}\Big)^{\beta_{n-1}}\Big).
	%\end{align*}
	$$|e_3|\lesssim  F  \prod_{j=2}^n \Big(\frac{|z_j|\wedge |z_j-z_1|}{ s_j}\Big)^{\alpha }
	\sum_{j=2}^n \Big(\Big(\theta^{-B_0 }Q^{B_0 } \frac{|z_j'-z_j|}{s_j}\Big)^{\beta_{n-1}/2}+\Big(\frac{|\Imm z_j'-\Imm z_j|}{\Imm z_j}\Big)^{\beta_{n-1}}\Big).$$

	At the end, we follow the argument after the estimate on $e_{13}$ in Section \ref{Convergence-Sec}. First suppose that $\frac{s_j}{|z_j|\wedge |z_j-z_1|}=X$, $2\le j\le n$, for some $X\in(0,1)$ to be determined. Then we have $\frac{|z_j'-z_j|}{s_j}\le X^{-1}\cdot \frac{|z_j'-z_j|}{d_j}$, $2\le j\le n$. Then we may set $$\theta=\max_{2\le j\le n}\Big(\frac{|z_j'-z_j|}{d_j}\Big)^a,\quad X=\max_{2\le j\le n}\Big(\frac{|z_j'-z_j|}{d_j}\Big)^b \bigvee \max_{2\le j\le n} \Big(\frac{|\Imm z_j'-\Imm z_j|}{\Imm z_j}\Big)^c$$
	for some suitable constants $a,b,c>0$. It is easy to find those $a,b,c$ and some
	constants $B_n>0$ and $\beta_n,\delta_n\in(0,1)$ such that the upper bounds for $|e_1|,|e_1'|,|e_2|,|e_2'|,|e_3|$ are all bounded by the RHS of (\ref{RHS-main2}) with $z_1'=z_1$, and if (\ref{Cond-main2}) holds, then (\ref{CondA},\ref{CondB},\ref{CondC}) all hold. The proof is now complete.
\end{proof}

\begin{proof}[Proof of Lemma \ref{main2B}.]
	Fix $s_j\in(|z_1'-z_1|,|z_j-z_1|\wedge |z_j|)$, $2\le j\le n$, and $\eta_2>\eta_1>|z_1'-z_1|$ depending on $\kappa,n,z_1,z_1',z_2,\dots,z_n$ to be determined later.  Define $E_{0;\vec s}$, $E_{\eta_1;\vec s}$, and $E_{\eta_1,0;\eta_2}$  using (\ref{Ers}), (\ref{Ers}), and (\ref{ErsR}), respectively,  for $z_1,z_2,\dots,z_n$. Define $E_{0;\vec s}'$ using (\ref{Ers}) for $z_1',z_2,\dots,z_n$, let $E_{\eta_1,\vec s}'=E_{\eta_1;\vec s}$, and define
	\begin{align*}
	E'_{\eta_1,0;\eta_2}=&\{\gamma[\tau^{z_1}_{\eta_1},T_{z_1'}]\mbox{ does not intersect any connected component of } \\
	& \{|z-z_1|=\eta_2\}\cap H_{\tau^{z_1}_{\eta_1}}  \mbox{ that separates }z_1'\mbox{ from any } z_k, 2\le k\le n\}.
	\end{align*}
	
	First, we express
	$$\ha G(z_1,z_2,\dots,z_n)=G(z_1)\EE_{z_1}^*[{\bf 1}_{E_{0;\vec s}} \ha G_{T_{z_1}}(z_2,\dots,z_n)]+e_1;$$
	$$\ha G(z_1',z_2,\dots,z_n)=G(z_1')\EE_{z_1'}^*[{\bf 1}_{E_{0;\vec s}'} \ha G_{T_{z_1'}}(z_2,\dots,z_n)]+e_1'.$$
	Now suppose (\ref{Cond2AB}) holds. Recall that we have $|z_j-z_1'|\asymp|z_j-z_1|$, $2\le j\le n$, $Q'\asymp Q$ and $F'\asymp F$.
	By Lemma \ref{RZ-Thm3.1-lim}, we see that there is a constant $\beta>0$ such that
	$$0\le e_1,e_1'\lesssim F  \sum_{j= 2}^n \Big(\frac{ s_j}{|z_j|\wedge |z_j-z_1|}\Big)^{\beta }.$$
	
	Second, we express
	$$ G(z_1)\EE_{z_1}^*[{\bf 1}_{E_{0;\vec s}} \ha G_{T_{z_1}}(z_2,\dots,z_n)]=G(z_1)\EE_{z_1}^*[{\bf 1}_{E_{0;\vec s}\cap E_{\eta_1,0;\eta_2}} \ha G_{T_{z_1}}(z_2,\dots,z_n)]+e_2;$$
	$$ G(z_1')\EE_{z_1'}^*[{\bf 1}_{E_{0;\vec s}'} \ha G_{T_{z_1'}}(z_2,\dots,z_n)]=G(z_1')\EE_{z_1'}^*[{\bf 1}_{E_{0;\vec s}'\cap E'_{\eta_1,0;\eta_2}}  \ha G_{T_{z_1'}}(z_2,\dots,z_n)]+e_2'.$$
	From Lemma \ref{stayin}, Corollary \ref{stayin-cor} (applied to $Z=\{z_j\}$, $2\le j\le n$), Lemma \ref{FF-F}, and that $|z_j-z_1'|\asymp |z_j-z_1|$ and $F'\asymp F$, we get
	$$0\le e_2,e_2'\lesssim F   \prod_{j=2}^{n}\Big(\frac{|z_j|\wedge |z_j-z_1|}{s_j}\Big)^{\alpha} \Big(\frac{\eta_1}{\eta_2}\Big)^{\alpha/4} .$$
	
	Third, we change the times in the two expressions from $T_{z_1}$ and $T_{z_1'}$, respectively, to the same time $\tau^{z_1}_{\eta_1}$, and express
	$$ G(z_1)\EE_{z_1}^*[{\bf 1}_{E_{0;\vec s}\cap E_{\eta_1,0;\eta_2}} \ha G_{T_{z_1}}(z_2,\dots,z_n)]=G(z_1)\EE_{z_1}^*[{\bf 1}_{E_{\eta_1;\vec s}\cap E_{\eta_1,0;\eta_2}} \ha G_{\tau^{z_1}_{\eta_1}}(z_2,\dots,z_n)]+e_3 ;$$
	$$  G(z_1')\EE_{z_1'}^*[{\bf 1}_{E_{0;\vec s}'\cap E'_{\eta_1,0;\eta_2}}  \ha G_{T_{z_1'}}(z_2,\dots,z_n)] =G(z_1')\EE_{z_1'}^*[{\bf 1}_{E_{\eta_1;\vec s}'\cap E'_{\eta_1,0;\eta_2}}  \ha G_{\tau^{z_1}_{\eta_1}}(z_2,\dots,z_n)]+e_3' .$$
	Now suppose (\ref{Cond4}) holds.
	%\BGE \eta_2+s_k<|z_k-z_1|,\quad 2\le k\le n.\label{Conda}\EDE
	Then $E_{\eta_1,0;\eta_2}\cap E_{\eta_1;\vec s}=E_{\eta_1,0;\eta_2}\cap E_{0;\vec s}$ and $E'_{\eta_1,0;\eta_2}\cap E'_{\eta_1;\vec s}=E'_{\eta_1,0;\eta_2}\cap E'_{0;\vec s}$. Applying Lemma \ref{e4-lemma} with $a=\tau^{z_1}_{\eta_1}$, $b=T_{z_1}$ or $b=T_{z_1'}$, and using $Q'\asymp Q$, $F'\asymp F$ and $|z_j-z_1'|\asymp |z_j-z_1|$, we find that, for some constants $B_*>0$ and $\beta_*,\delta_*\in(0,1)$, if (\ref{Cond4'}) holds, then
	$$|e_3|,|e_3'|\lesssim F \prod_{j=2}^n \Big(\frac{|z_j|\wedge |z_j-z_1|}{ s_k }\Big)^\alpha \Big(Q^{B_*}   \frac{\eta_2}{d_1}\Big)^{\beta_*}.$$
	
	Note that the proof of Lemma \ref{e4-lemma} uses Theorem \ref{main2} for $n-1$ points so we can use it here by induction hypothesis. Removing the restriction of the events $E_{\eta_1,0;\eta_2}$ and $E'_{\eta_1,0;\eta_2}$, we express
	$$ G(z_1)\EE_{z_1}^*[{\bf 1}_{E_{\eta_1;\vec s}\cap E_{\eta_1,0;\eta_2}} \ha G_{\tau^{z_1}_{\eta_1}}(z_2,\dots,z_n)]=G(z_1)\EE_{z_1}^*[{\bf 1}_{E_{\eta_1;\vec s} } \ha G_{\tau^{z_1}_{\eta_1}}(z_2,\dots,z_n)]-e_4 ;$$
	$$  G(z_1')\EE_{z_1'}^*[{\bf 1}_{E_{\eta_1;\vec s}'\cap E'_{\eta_1,0;\eta_2}}  \ha G_{\tau^{z_1}_{\eta_1}}(z_2,\dots,z_n)] =G(z_1')\EE_{z_1'}^*[{\bf 1}_{E_{\eta_1;\vec s}' }  \ha G_{\tau^{z_1}_{\eta_1}}(z_2,\dots,z_n)]-e_4' .$$
	The estimates on $e_4,e_4'$ are the same as that on $e_2,e_2'$ by Lemma \ref{stayin}, Corollary \ref{stayin-cor}, Lemma \ref{FF-F},  and that $F'\asymp F$ and $|z_j-z_1'|\asymp |z_j-z_1|$.
	
	Changing $G(z_1')$ to $G(z_1)$ on the RHS of the second displayed formula, we express
	$$G(z_1')\EE_{z_1'}^*[{\bf 1}_{E_{\eta_1;\vec s}' }  \ha G_{\tau^{z_1}_{\eta_1}}(z_2,\dots,z_n)]=G(z_1)\EE_{z_1'}^*[{\bf 1}_{E_{\eta_1;\vec s}' }  \ha G_{\tau^{z_1}_{\eta_1}}(z_2,\dots,z_n)]+e_5.$$
	From (\ref{G(z)}) and Lemma \ref{FF-F} we see that there is a constant $\delta>0$ such that, if
	\BGE \frac{|z_1'-z_1|}{|z_1|}<\delta,\quad \frac{|\Imm z_1'-\Imm z_1|}{\Imm z_1}<\delta,\label{condb}\EDE
	then
	$$|e_5|\lesssim F \prod_{k=2}^n \Big(\frac{|z_k|\wedge |z_k-z_1|}{ s_k }\Big)^\alpha\Big(\frac{|z_1'-z_1|}{|z_1|}+\frac{|\Imm z_1'-\Imm z_1|}{\Imm z_1}\Big).$$
	
	Finally, we express
	$$G(z_1)\EE_{z_1'}^*[{\bf 1}_{E_{\eta_1;\vec s}' }  \ha G_{\tau^{z_1}_{\eta_1}}(z_2,\dots,z_n)]=G(z_1)\EE_{z_1}^*[{\bf 1}_{E_{\eta_1;\vec s}}  \ha G_{\tau^{z_1}_{\eta_1}}(z_2,\dots,z_n)]+e_6.$$
	Since $E_{\eta_1,\vec s}'=E_{\eta_1;\vec s}$, the random variables in the two square brackets are the same, which is $\F_{\tau^{z_1}_{\eta_1}}$-measurable. By Lemmas \ref{continuity-two-sided} and \ref{FF-F}, we see that there is a constant $\delta$ such that, if
	\BGE \frac{|z_1'-z_1|}{\eta_1}<\delta,\label{Condc}\EDE
	then
	$$|e_6|\lesssim F \prod_{k=2}^n \Big(\frac{|z_k|\wedge |z_k-z_1|}{ s_k }\Big)^\alpha\Big(\frac{|z_1'-z_1|}{\eta_1}\Big).$$
	
	At the end, we follow the argument after the estimate on $e_{13}$ in Section \ref{Convergence-Sec}. Suppose that $\frac{s_j}{|z_j|\wedge |z_j-z_1|}=X$, $2\le j\le n$, for some $X\in(0,1)$ to be determined. Pick $\eta_1,\eta_2$ such that $|z_1'-z_1|/\eta_1=\eta_1/\eta_2=\eta_2/d_1$. It is easy to find  constants $a,B_n>0$ and $\beta_n,\delta_n\in(0,1)$  such that with $X=(\frac{|z_1'-z_1|}{d_1})^a$, if (\ref{Cond-main2}) holds for $j=1$, then Conditions (\ref{Cond2AB},\ref{Cond4},\ref{Cond4'},\ref{condb},\ref{Condc}) all hold, and the upper bounds for $|e_j|$, $1\le j\le 6$, and $|e_j'|$, $1\le j\le 4$, are all bounded by the RHS of (\ref{RHS-main2}). The proof is now complete.
\end{proof}

\section{Proof of Theorem \ref{strong-lower}} \label{proof2}

In this section we want to show the desired lower bound for the multi-point Green's function. The method of the proof is based on the generalization of the method used  in \cite{LZ} and \cite{LR2} to show the lower bound. We find the best point (almost means the nearest point but we make it precise) to go near first and we consider the event to go near that point before going near other points (as much as possible). This can be done by staying in a L-shape as defined in \cite{LZ}. It is possible that we can not go all the way to a specific given point since couple of points are very near each other. In this case we can stop in an earlier time and separate points by a conformal map. We will go through the details about this general strategy in this section.
Following Lawler and Zhou in \cite{LZ}, we define for $z\in\lin\HH$ and $\rho\in(0,1)$,
\[
L_z=[0,\Ree z] \cup [\Ree z,z] ,
\]
and
\[
L_{z,\rho}=\{z' \in \lin\Half| \dist(z',L_z) \leq \rho|z|\}.
\]

A simple geometry argument shows that, for any $z_0\in\lin\HH\sem \{0\}$ and $\rho\in(0,1)$,
\BGE L_{z_0,\rho}\cap\{z\in\lin\HH:|z|\ge |z_0|\}\subset \{|z-z_0|\le \sqrt{2\rho} |z_0|\}.\label{L-ineq}\EDE

Now we state a lemma which shows what happens to points which are not in the L-shape when we flatten the domain.

\begin{lemma} \label{farpoint}
	Suppose $0<\rho \leq \frac{1}{4}$. Then the following equations hold with implicit constants depending only on $\kappa$ and $\rho$. Suppose $z \in \Half,$  $z_1,z_2 \in \Half \setminus L_{z,2\rho},$ and $\gamma(t)$, $0\le t\le T$, is a chordal Loewner curve such that $\gamma(0)=0$, $\gamma(T)=z$, and $\gamma[0,T] \subset L_{z,\rho}$. Let $Z=Z_T$ be the centered Loewner map at time $T$. Then we have the following.
	\[
	|Z'(z_1)| \asymp 1.
	\]
	\[
	\Im(Z(z_1)) \asymp \Im(z_1).
	\]
	\[
	|Z(z_1)| \asymp |z_1|.
	\]
	\[
	|Z(z_1)-Z(z_2)| \lesssim |z_1-z_2|.
	\]
	 Finally if $z_1,z_2,\dots,z_n$ are distinct points in $ \Half \setminus L_{z,2\rho}$ and $r_1,\dots,r_n >0$ we have
	\[
	F(Z(z_1),\dots,Z(z_n);|Z'(z_1)|r_1,\dots,|Z'(z_n)|r_n) \gtrsim F(z_1,\dots,z_n;r_1,\dots,r_n).
	\]

\end{lemma}
\begin{proof}
	The proofs for first 3 equations above are in \cite[Proposition 3.2]{LZ}. For the second to last one, suppose $\eta$ is a curve in $ \Half \setminus L_{z,2\rho}$ which connects $z_1$ and $z_2$ and has length at most $c_1|z_1-z_2|$. If the closed line $l$ passing through $z_1$ and $z_2$ does not pass through $L_{z,2\rho}$ then it works otherwise we go on the $l$ until we hit $L_{z,2\rho}$ then we go up on $L_{z,2\rho}$ to modify pass such that it does not pass through $L_{z,2\rho}$. Then the length of the image of $\eta$ under $Z$ is at most $c_2|z_1-z_2|$ by derivative estimate. %This gives us one direction of it. For the reverse apply the same reasoning to map $Z^{-1}$ in the conformal image and note that $|Z'(z_1)|,|Z'(z_2)| \asymp 1$.
The last statement is a result of the definition of $F$ and the previous equations.
\end{proof}
\begin{remark}
We expect that $|Z(z_1)-Z(z_2)| \asymp |z_1-z_2|$ holds in the statement of the lemma. We do not try to prove it since it is not needed.

The same proof gives us the following modification of Lemma \ref{farpoint}. Suppose the chordal Loewner curve $\gamma$ satisfies that $\gamma[0,T] \subset \{|z| \leq R\}$. Suppose $z_1,\dots,z_n \not \in  \{|z| \leq 2R\}$. Then all the results of the Lemma \ref{farpoint} holds for $z_1,\dots,z_n$ as well. These results also follow from \cite[Lemma 5.4]{LERW}. See, e.g., the proof of Corollary \ref{martingale}.
\end{remark}

Now we   strengthen  \cite[Proposition 3.1]{LZ}. We quantify the chance that we stay in the L-shape and at the same time the tip of the curve behaves nicely. Among those estimates, (iii) means that the ``angle'' of $z_0$ (see the description after the definition (\ref{Erth}) of $E_{r;\theta}$) viewed from the tip of $\gamma$ at $\tau_0$ is not small.

\begin{proposition} \label{L-shape}
	There are uniform constants $C_0,C_1>0,N>2,b_2>1>b_1>0$ such that for every $0<\delta<1$, there is  $C_\delta>0$ such that for every $z_0 \in \lin{\Half} \setminus \{0\}$ and $0<r\le \frac{\delta |z_0|}{N}$ there exists stopping time $\tau_0=\tau_0^\delta(z_0,r)$ such that the event $E_{\tau_0}$ defined by $\tau_0<\infty$ and
	\begin{itemize}
		\item[(i)] $\dist(z_0,\gamma[0,{\tau_0}]) \in (b_1 r, b_2r)$,
		
		\item[(ii)] $\gamma[0,{\tau_0}] \subset L_{z_0,\delta}$,
		
		\item[(iii)] $\dist(g_{\tau_0}(z_0),S_{K_{\tau_0}}) \geq C_0|g_{\tau_0}(z_0)-U_{\tau_0}|=C_0|Z_{\tau_0}(z_0)|$,

\item[(iv)]		$|Z_{\tau_0}(z_0)| \leq C_1\sqrt{r|z_0|}$,
	\end{itemize}
satisfies that
\BGE \PP_{z_0}^*[E_{\tau_0}]\ge C_\delta;\label{P*ge}\EDE
\BGE \PP[E_{\tau_0}]\ge C_\delta F(z_0;r).\label{Pge}\EDE
%where $\PP$ is the chordal SLE$_\kappa$ measure, and $\PP_{z_0}^*$ is the two-sided radial (resp.\ chordal) SLE$_\kappa$ measure through $z_0$ if $z_0\in\HH$ (resp.\ $z_0\in\R\sem\{0\}$).
\end{proposition}

\begin{proof}
By scaling we may assume $\max\{|x_0|,y_0\}=1$, where $x_0=\Ree z_0$ and $y_0=\Imm z_0$. Then $|z_0|\asymp 1$. %We write $Z_t(z)=g_t(z)-U_t$.
We first prove (\ref{P*ge}), and consider two different cases to prove this. First we consider the interior case when $r$ is smaller or comparable to $y_0$, and then we consider the boundary case when $r$ is bigger or comparable to $y_0$. Also throughout the proof we consider $N$ as a fixed number (greater than $2$) which we will determine at the end.
\smallbreak
	
\no{\bf Interior Case:} Suppose for this case that $r<10y_0$. Define the stopping time $\tau$ by
	\[
	\tau=\inf\{t:  \dist(\gamma(t),z_0)=\frac{y_0}{10} \wedge r \}.
	\]
%	Call the event of staying in $L_{z,\delta}$ until $T_{z_0}$ as $E_{{z_0},\delta}$.
By \cite[Proposition 3.1]{LZ}, we know that there is $u>0$ depending only on $\kappa$ and $\frac\delta N$ such that for every $z_0 \in \Half$, $\PP^*_{z_0}[\gamma[0,T_{z_0}] \subset L_{z_0,\frac{\delta}{N}}] \geq u$.
	%Note that $\frac{\delta}{N}=(N_\delta)^{-1}$.
	By this we know that
	\[
	\PP^*_{z_0}[\gamma[0,\tau] \subset L_{z_0,\frac{\delta}{N}}] \geq u.
	\]
 Let $\til E$ denote  the event $\gamma[0,\tau] \subset L_{z_0,\frac{\delta}{N}}$. % as $E_{z_0,\frac{\delta}{N}}^\tau$.
 Now define $\tau_0$ by
	\[
	\tau_0=\inf\{t: \Upsilon_t(z_0)=\frac{y_0}{100} \wedge \frac{r}{10}\},
	\]
where $\Upsilon_t(z_0)$ is the conformal radius of $z_0$ in $H_t$.

 Now we want to show $\PP^*_{z_0}[E_{\tau_0}|\til{E}] \geq u_0$ for some constant $u_0>0$. % Here we get that $u_0$ is even independent of $\delta$ but this is not needed.
 Since $\PP^*_{z_0}$-a.s.\ $T_{z_0}<\infty$, we have $\PP_{z_0}^*[\tau_0<\infty]=1$. By Koebe's $1/4$ theorem, we immediately have Property (i).
		
	For Property (ii) let $E_{\tau_0}^1$ denote the event that after time $\tau,$ $\gamma$ stays in $L_{z_0,\delta}$ till $T_z$. From Lemma \ref{stayin} applied to $Z=\pa L_{z_0,\delta}$, we get $\PP^*_{z_0}[(E_{\tau_0}^1)^c] \lesssim N^{-c}$ for some constant $c>0$. Since $\PP[\til E]\ge u$, there is a constant $C>0$ such that
  $\PP^*_{z_0}[E_{\tau_0}^1|\til{E}] \geq 1-CN^{-c}$.
		
	For Property (iii) we use \cite[Lemma 2.2]{LZ}. By Koebe's $1/4$ theorem we know that $\log(\Upsilon_{\tau_0})-\log (\Upsilon_\tau) \le  -1$.  By \cite[Lemma 2.2]{LZ}, for any $\rho<1$ we have $\theta_0>0$ such that
	\[
	\PP^*_{z_0}[\Imm Z_{\tau_0}(z_0)/|Z_{\tau_0}(z_0)| \geq \theta_0|\mathcal{F}_\tau] \geq \rho.
	\]
	Call the event $\Imm Z_{\tau_0}(z_0)/|Z_{\tau_0}(z_0)|  \geq \theta_0$ as $E^2_{\tau_0}$. If %$S_{\tau_0}(z_0) \geq \theta_0$
	$E^2_{\tau_0}$ occurs then Property (iii) is satisfied (with the constant depending on $\theta_0$) because  $\dist(g_{\tau_0}(z_0),S_{K_{\tau_0}})\ge \Imm Z_{\tau_0}(z_0)$.

	If we choose $\rho\in(0,1)$ and $N>2$ such that $u_0=\rho-CN^{-c}>0$ then we have
	\[
	\PP^*_{z_0}[E^1_{\tau_0} \cap E^2_{\tau_0}|\til{E}]\geq \PP^*_{z_0}[E_{\tau_0}^1|\til{E}]+\PP^*_{z_0}[E_{\tau_0}^2|\til{E}]-1 \geq \rho-CN^{-c}=u_0>0.
	\]
So $\PP^*_{z_0}[E^1_{\tau_0} \cap E^2_{\tau_0}]\ge u u_0>0$.  We have seen that Properties (i)-(iii) are satisfied on the event $E^1_{\tau_0} \cap E^2_{\tau_0}$. For Property (iv), set $Z=Z_{\tau_0}$, and let $\Pi=\{z \in \Half: \Im(z)= 10\}$. Then $\Imm Z(z)\le \Imm z=10$ for $z\in\Pi$. Consider the event that Brownian motion starting at $z_0$ hits $\Pi$ before hitting $\gamma[0,{\tau_0}] \cup \R$. By Property (i) and Beurling estimate it has chance less than $c\sqrt{r}$ for some fixed constant $c$. After map $Z$, the chance that Brownian motion starting at $Z(z_0)$ hits $Z(\Pi)$ before hitting $\R$ is at least $\Im(Z(z_0))/10$ by gambler's ruin estimate which has the same order as $|Z(z_0)|$ when $E_{\tau_0}^2$ happens. So we have Property (iv) on the event $E^1_{\tau_0} \cap E^2_{\tau_0}$. Thus, $E^1_{\tau_0} \cap E^2_{\tau_0}\subset E_{\tau_0}$. This finishes the proof of (\ref{P*ge}) in the interior case.
	
\smallbreak
\no{\bf Boundary Case:} For this case assume that $1>r\geq 10y_0$. Without loss of generality we assume $x_0=1$. Then $z_0=1+iy_0$.
We follow the steps as in the interior case just we have to modify some definitions for the boundary case. First, following \cite{Law3} we consider
	\[
	x_t=\inf\{x>0: T_x>t\},\quad
D_t=H_t \cup \{\bar{z}: z \in H_t\} \cup (x_t,\infty),
	\]
	\[
	X_t=Z_t(1)=g_t(1)-U_t,\quad 	O_t=g_t(x_t)-U_t,
	\]
	\[
 J_t=\frac{X_t-O_t}{X_t},\quad  \Upsilon_t(1)=\frac{X_t-O_t}{X_t}{g_t'(1)}.
	\]
Note that $\Upsilon_t$ is $1/4$ times the conformal radius of $1$ in $D_t$. So we have
\BGE
	\frac{1}{4}\dist(1,\partial D_t) \leq \Upsilon_t(1) \leq \dist(1,\partial D_t).\label{UpDt}
\EDE
		
%	Like above we take $N_\delta=\frac{N}{\delta}$.
Take
	\[
	\tau=\inf\{t: \dist(\gamma(t),1)=100r \}.
	\]
	By \cite[Proposition 3.1]{LZ}, we know that there is $u>0$ depending on $\kappa$ and $\frac \delta N$ such that
$\PP^*_1[\gamma[0,T_1] \subset L_{1,\frac{\delta}{N}}] \geq u$.
%	In fact if the driving function is sufficiently near zero we stay in $L_{1,\frac{\delta}{N}}$.
Let $\til E$ denote the event that $\gamma[0,\tau]\subset L_{1,\frac{\delta}{N}}$. Then $\PP^*_1[\til E]\ge u$.
Now take $\tau_0$ as
	\[
	\tau_0=\inf\{ t: \Upsilon_t(1)=8r\}.
	\]

Since $\PP^*_{1}$-a.s.\ $T_{1}<\infty$, we have $\PP^*_1[\tau_0<\infty]=1$. By (\ref{UpDt}), we immediately have Property (i). Let $E^{\tau_0}_1$ denote the event that after $\tau$, the curve stays in $L_{1,\delta}$ till $T_1$. Using Lemma \ref{stayin} as in the interior case, we get  $\PP^*_1[E^1_{\tau_0}|\til E]\ge 1-CN^{-c}$ for some constants $C,c>0$. If $E^1_{\tau_0}$ happens, since $L_{1,\delta}\subset L_{z_0,\delta}$, we have Property (ii).
%The other thing is to define a suitable notion for conformal angle since 1 is on the boundary.

By Koebe's $1/4$ theorem we know that $\log(\Upsilon_{\tau_0})-\log (\Upsilon_\tau) \le  -1$. By \cite[Section 4]{Law3} we have that for any $\rho<1$ there is $\theta_0>0$ such that
	\[
	\PP^*_1[J_{\tau_0} \geq \theta_0| \mathcal{F}_\tau ] \geq \rho.
	\]
Call the event $J_{\tau_0} \geq \theta_0$ as $E^2_{\tau_0}$. Since $|z_0-1|=y_0$ and $\dist(z_0,K_{\tau_0})\ge 2r\ge 20 y_0$, by Koebe's $1/4$ theorem and distortion theorem, we get $|g_{\tau_0}(z_0)- g_{\tau_0}(1)|\le \frac 29 \dist(g_{\tau_0}(z_0),S_{K_{\tau_0}})$. Thus, by triangle inequality, $\dist(g_{\tau_0}(z_0),S_{K_{\tau_0}})\asymp \dist(g_{\tau_0}(1),S_{K_{\tau_0}})$. Since $U_{\tau_0}\in S_{K_{\tau_0}}$, we have $|g_{\tau_0}(z_0)- g_{\tau_0}(1)|\le \frac 29 |g_{\tau_0}(z_0)-U_{\tau_0}|$. So we also get $|g_{\tau_0}(z_0)-U_{\tau_0}|\asymp |g_{\tau_0}(1)-U_{\tau_0}|$. If $E^2_{\tau_0}$ happens then the Property (iii) is satisfied at the point $1$ with $C_0=\theta_0$, and so is also satisfied at the point $z_0$ with a bigger constant by the above estimates.

If we choose $\rho\in(0,1)$ and $N>2$ such that $u_0=\rho-CN^{-c}>0$ then we have $\PP_1^*[E^1_{\tau_0} \cap E^2_{\tau_0}|\til{E}]\ge u_0$.
So $\PP^*_{1}[E^1_{\tau_0} \cap E^2_{\tau_0}]\ge u u_0>0$. Since $\dist(z_0,\gamma[0,{\tau_0}]) \geq 2r$, until time $\tau_0$ the two probability measures $\PP^*_{z_0}$ and $\PP^*_{1}$ are comparable by a universal constant $c$ by \cite[Proposition 2.9]{LZ}. So we get $\PP^*_{z_0}[E^1_{\tau_0} \cap E^2_{\tau_0}]\ge u u_0/c>0$.

We have seen that Properties (i)-(iii) are satisfied on the event $E^1_{\tau_0} \cap E^2_{\tau_0}$. For Property (iv), similar to  the interior case, we use Beurling estimate. Take $D=D_{\tau_0}$. Brownian motion starting at $1$ has chance less than $c\sqrt{r}$ to hit $\Pi=\{\Imm z=10\}$ before exiting $D$. By conformal invariance of Brownian motion, this implies that distance between  $(-\infty,O_{\tau_0})$ and $Z_{\tau_0}(1)$ which is $X_{\tau_0}-O_{\tau_0}$ is not more than $c\sqrt{r}$, which then implies $g'_{\tau_0}(1)\lesssim \frac{1}{\sqrt{r}}$ because $\Upsilon_{\tau_0}\asymp r$. Since $J_{\tau_0}\ge \theta_0$, we have $|Z_{\tau_0}(1)|\lesssim \sqrt r$. By Koebe's distortion theorem we get  $|Z _{\tau_0}(z_0)-Z_{\tau_0} (1)| \lesssim g'_{\tau_0}(1)|z_0-1|\lesssim \sqrt r$. So we get $|Z_{\tau_0}(z_0)|\lesssim \sqrt r$, as desired. So we get $E^1_{\tau_0} \cap E^2_{\tau_0}\subset E_{\tau_0}$. This finishes the proof of (\ref{P*ge}) in the boundary case.

Finally, we prove (\ref{Pge}). From \cite{LW,LZ} we know that $\PP$ is absolutely continuous with respect to $\PP^*_{z_0}$ on $\F_{\tau_0}\cap\{\tau_0<\infty\}$, and the Radon-Nikodym derivative is
$$R=\left\{ \begin{array}{ll} \frac{|Z_{\tau_0}(z_0)|^\alpha \Imm(Z_{\tau_0}(z_0))^{(2-d)-\alpha}}{|g_{\tau_0}'(z_0)|^{2-d}|z_0|^\alpha y_0^{(2-d)-\alpha}}, &z_0\in\HH;\\
\\
\frac{|Z_{\tau_0}(z_0)|^\alpha}{|g_{\tau_0}'(z_0)|^\alpha |z_0|^\alpha}, &z_0\in\R\sem\{0\}.
\end{array}\right.$$
Recall that in both of the above two cases, we defined events $E^1_{\tau_0}$ and $E^2_{\tau_0}$ such that $E^1_{\tau_0} \cap E^2_{\tau_0}\subset E_{\tau_0}$ and $\PP^*_{z_0}[E^1_{\tau_0} \cap E^2_{\tau_0}]\gtrsim 1$. So it suffices to show that $R\asymp F(z_0;r)$ on $E^2_{\tau_0}$.

In the interior case, suppose $E^2_{\tau_0}$ happens. Then $\Imm Z_{\tau_0}(z_0)\asymp |Z_{\tau_0}(z_0)|$. They are also comparable to $\dist(g_{\tau_0}(z_0),S_{K_{\tau_0}})$ because $\Imm Z_{\tau_0}(z_0)\le \dist(g_{\tau_0}(z_0),S_{K_{\tau_0}})\le |Z_{\tau_0}(z_0)|$. By Koebe's $1/4$ theorem we get
$$R\asymp \frac{\dist(g_{\tau_0}(z_0),S_{K_{\tau_0}})^{2-d}}{|g_{\tau_0}'(z_0)|^{2-d}|z_0|^\alpha y_0^{(2-d)-\alpha}}
\asymp \frac{\dist(z_0,{K_{\tau_0}})^{2-d}}{|z_0|^\alpha y_0^{(2-d)-\alpha}}\asymp  \frac{r^{2-d}}{|z_0|^\alpha y_0^{(2-d)-\alpha}}=F(z_0;r). $$

In the boundary case, by Koebe's distortion theorem, we get $R\asymp \frac{|Z_{\tau_0}(z_0)|^\alpha}{|g_{\tau_0}'(z_0)|^\alpha |z_0|^\alpha}$. Suppose $E^2_{\tau_0}$ happens. Then $|Z_{\tau_0}(z_0)|\asymp \dist(g_{\tau_0}(z_0),S_{K_{\tau_0}})$. By Koebe's $1/4$ theorem we get
$$R\asymp \frac{\dist(g_{\tau_0}(z_0),S_{K_{\tau_0}})^\alpha}{|g_{\tau_0}'(z_0)|^\alpha |z_0|^\alpha}\asymp \frac{\dist( z_0,{K_{\tau_0}})^\alpha}{ |z_0|^\alpha}\asymp \frac{r^\alpha}{|z_0|^\alpha}=F(z_0;r).$$
So we get $R\asymp F(z_0;r)$ on $E^2_{\tau_0}$ in both cases. The proof is now complete.
\end{proof}

\no{\bf Remark.} Since $F(z_0;r)$  is comparable to the probability that SLE goes to distance $r$ of $z_0$, we showed that there is a good chance to  go to distance $r$ of $z_0$ in a ``good way''. Once we have this we can prove Theorem \ref{strong-lower}.

\begin{proof} [Proof of Theorem \ref{strong-lower}]
We prove the theorem by induction on $n$. For $n=1$ it is a corollary of Proposition \ref{L-shape}. Suppose that $n\ge 2$ and the theorem is true for $1,\dots,n-1$ with constants $C_j>0$ and $V_j>1$, $1\le j\le n$, and we want to prove it for $n$. We consider different cases.

We now give a summary of the cases that will be considered.
The first case: Case A happens when $\{z_1,\dots,z_n\}$ can be divided into two nonempty groups such that the first group lie inside of a smaller semidisc, and the second group lie outside of a bigger semidisc, both centered at $0$. In this case a good strategy for $\gamma$ is to visit   neighbors of all points in the first group before leaving a semidisc centered at $0$. We then reduce Case A to the induction hypothesis. The second case: Case B happens when $\{z_1,\dots,z_n\}$ can be divided into two nonempty groups such that for a point, say $z_1$, with the smallest modulus, the first group  lie inside of a thin $L$-shape w.r.t.\ $z_1$ and the second group lie outside of a thick $L$-shape w.r.t.\ $z_1$. In this case we use Proposition \ref{L-shape} to $\gamma$ to reach some suitable distance from $z_1$ before leaving an $L$-shape w.r.t.\ $z_1$ such that the ``angle'' of $z_1$ viewed from the tip of $\gamma$ is not small. By mapping the complement domain conformally onto $\HH$, we reduce this case to Case A or the induction hypothesis. The third case: Case C happens when all of $z_j$'s lie inside of a thin $L$-shape w.r.t.\ $z_1$, which has the smallest modulus. By (\ref{L-ineq}) they lie in a small disc centered at $z_1$. In this case we use Proposition \ref{L-shape} again to let $\gamma$ approach this group while staying inside an $L$-shape such that the ``angle'' of $z_1$ viewed from the tip of $\gamma$ is not small. By applying a conformal map, we then reduce this case to Case B. See Figure \ref{ABC}

\begin{figure}
	\hfill
	\includegraphics[width=0.3\textwidth]{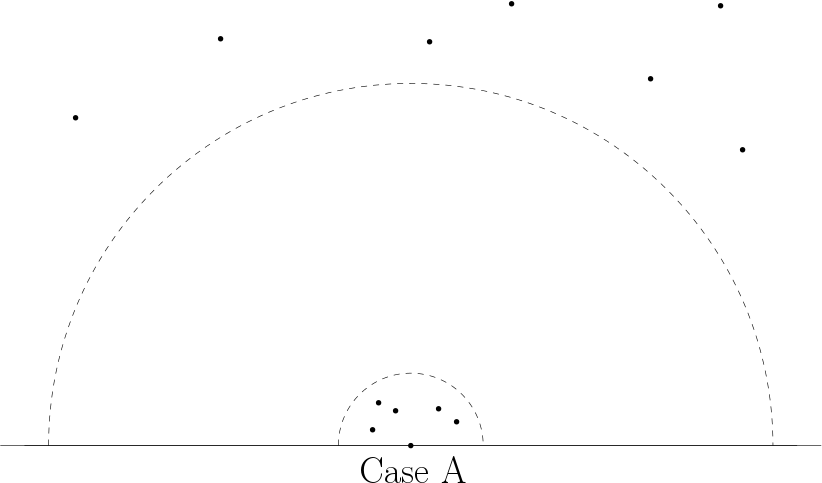}%
	\hfill
	\includegraphics[width=0.3\textwidth]{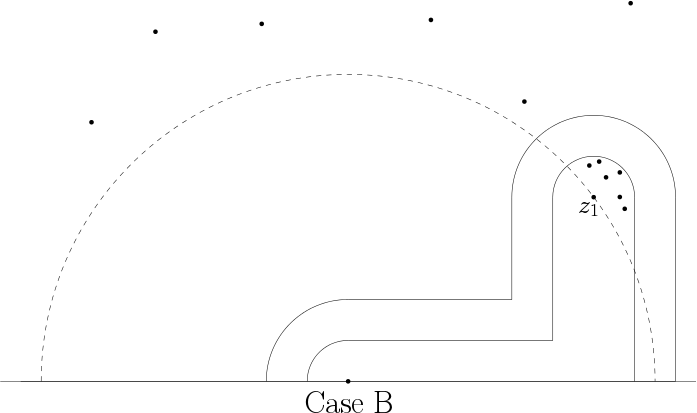}%
	\hfill
	\includegraphics[width=0.3\textwidth]{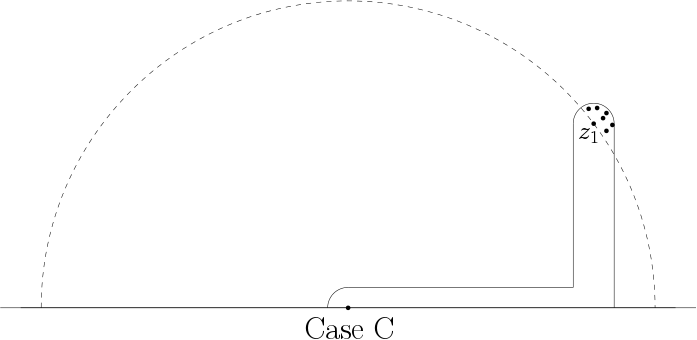}
	\hfill
	\caption{The three cases in the proof of Theorem \ref{strong-lower}.}
		\label{ABC}
\end{figure}

\smallbreak
\no{\bf Case A}: There exist $R,r>0$ and $m\in\N$ with
 $R\ge 2(\max_{1\le j\le n-1}V_j)r>0$ and $m\le n-1$ such that  $|z_j|< r$, $1\le j\le m$, and $|z_j|> R$, $m+1\le j\le n$. Let ${\tau_0}=\vee_{j=1}^m \tau^{z_j}_{r_j}$ and $r'=R/2$. From the induction hypothesis, we have $\PP[{\tau_0}<\tau_{\{|z|=r'\}}]\gtrsim F(z_1,\dots,z_m;r_1,\dots,r_m)$. Let $E_{\tau_0}$ denote the event ${\tau_0}<\tau_{\{|z|=r'\}}$. Let $\til\gamma(t)=Z_{\tau_0}(\gamma({\tau_0}+t))$, $\til z_j=Z_{\tau_0}(z_j)$, and $\til r_j=|Z_{\tau_0}'(z_{j})|r_{j}/4$, $m+1\le j\le n$. By DMP of SLE, conditionally on $\F_{\tau_0}$, $\til\gamma$ has the same law as $\gamma$. Let $\til\tau_S$ and $\til\tau^z_r$ be the stopping times that correspond to $\til\gamma$. By  induction hypothesis, we have
$$\PP[\til\tau^{\til z_j}_{\til r_j}<\til\tau_{\{|z|=V_{n-m}\sum_{j=m+1}^n|\til z_j|\}},m+1\le j\le n|\F_{\tau_0},E_{\tau_0}]\gtrsim F(\til z_{m+1},\dots,\til z_n;\til r_{m+1},\dots,\til r_n) .$$
Suppose $E_{\tau_0}$ happens. Then $K_{\tau_0}\subset\{|z|\le r'\}$. By Lemma \ref{small} and that $U_{\tau_0}\in S_{K_{\tau_0}}$ we have $|Z_{\tau_0}(z)-z|\le 5r'$ for any $z\not\in\lin{K_{\tau_0}}$. Let $\til E$ denote the event on the LHS of the above displayed formula.
By Koebe's $1/4$ theorem, we see that $E_{\tau_0}\cap\til E\subset \bigcap_{j=1}^n \{\tau^{z_j}_{r_j}<\tau_{\{|z|=r''\}}\}$, where $r''=6r'+V_{n-m}\sum_{j=m+1}^n(|z_j|+5r')$. Since $r'\le R\le |z_n|$, we can find a constant $V_n>1$ such that $r''\le V_n\sum_{j=1}^n |z_j|$. Thus,
\begin{align*}
  &\PP[\tau^{z_j}_{r_j}<\tau_{\{|z|=V_n\sum_{j=1}^n |z_j|\}}]\ge \PP[E_{\tau_0}\cap\til E]= \EE[E_{\tau_0}]\cdot\EE[\PP[\til E|\F_{\tau_0},E_{\tau_0}]]\\
  \gtrsim & F(z_1,\dots,z_m;r_1,\dots,r_m) \cdot\EE[F(\til z_{m+1},\dots,\til z_n;\til r_{m+1},\dots,\til r_n)|\F_{\tau_0},E_{\tau_0}] .\\
  \gtrsim & F(z_1,\dots,z_m;r_1,\dots,r_m)\cdot F(z_{m+1},\dots,z_n;r_{m+1},\dots,r_n)\\
  \asymp & F(z_1,\dots,z_n;r_1,\dots,r_n).
\end{align*}
where the second last estimate follows from the remark after Lemma \ref{farpoint}, and the last estimate follows from Lemma \ref{multiplicative} because $\dist(z_j,\{z_1,\dots,z_m\})\asymp |z_j|$, $m+1\le j\le n$. The proof of Case A is now complete.
\smallbreak

We will reduce other cases to Case A or the case of fewer points. By (\ref{Frasymp}) we may assume that $z_1$ has the smallest norm among  $z_j$, $1 \leq j \leq n$. %Let $$e_k=\max_{1\le j\leq k}|z_j-z_1|,\quad 2\le k\le n.$$
Fix constants $\rho_j\in(0,1/2)$, $1\le j\le n$, with $\rho_1>\cdots>\rho_n$ to be determined later.

\smallbreak
\no{\bf Case B}: $\{z_1,\dots,z_n\}\sem L_{z_1,\rho_1}\ne\emptyset$. By pigeonhole principle, Case B is a union of subcases: Case B.$k$, $1\le k\le n-1$, where Case B.$k$ denotes the case that Case B happens and $\{z_1,\dots,z_n\}\cap (L_{z_1,\rho_k}\sem L_{z_1,\rho_{k+1}})=\emptyset$.
\smallbreak
\no{\bf Case B.$k$}: In this case we have $\{z_1,\dots,z_n\}\sem L_{z_1,\rho_k}\ne\emptyset$, $\{z_1,\dots,z_n\}\cap (L_{z_1,\rho_k}\sem L_{z_1,\rho_{k+1}})=\emptyset$, and $\{z_1,\dots,z_n\}\cap  L_{z_1,\rho_{k+1}}\ne\emptyset$ because $z_1\in L_{z_1,\rho_{k+1}}$. By (\ref{Frasymp}) we may assume that $z_1,\dots,z_m\in  L_{z_1,\rho_{k+1}}$ and $z_{m+1},\dots,z_n\not\in L_{z_1,\rho_k}$, where $1\le m\le n-1$.

We will apply Proposition \ref{L-shape}. Let $N,b_1,C_1$ be the constants there. Let $\delta=\frac{2N}{b_1}\sqrt{2\rho_{k+1}}$, and $r=\frac{\delta|z_1|}N$. Let $\tau_0=\tau_0^{\delta}(z_1,r)$ and $E_{\tau_0}$ be given by Proposition \ref{L-shape}.
For $1\le j\le m$, since $z_j\in L_{z_1,\rho_{k+1}}$ and $|z_j|\ge |z_1|$, by (\ref{L-ineq}), we have $|z_j-z_1|\le  \sqrt{2\rho_{k+1}}|z_1|\le \frac{b_1 r}2$. Suppose $E_{\tau_0}$ happens. By Koebe's $1/4$ theorem, we have
$$|g_{\tau_0}'(z_1)|b_1r\le |g_{\tau_0}'(z_1)|\dist(z_1,K_{\tau_0})\le 4\dist(g_{\tau_0}(z_1),S_{K_{\tau_0}})\le 4|Z_{\tau_0}(z_1)|\le 4C_1\sqrt{r|z_1|}.$$
For $1\le j\le m$, since $\dist(z_1,K_{\tau_0})\ge b_1r\ge 2|z_j-z_1|$, by Koebe's distortion theorem, we have
$$|Z_{\tau_0}(z_j)-Z_{\tau_0}(z_1)|\le 2|g_{\tau_0}'(z_1)||z_j-z_1|\le |g_{\tau_0}'(z_1)|b_1r\le 4C_1\sqrt{r|z_1|}.$$
Since $|Z_{\tau_0}(z_1)|\le C_1\sqrt{r|z_1|}$, we get
$$|Z_{\tau_0}(z_j)|\le 5C_1\sqrt{r|z_1|},\quad 1\le j\le m.$$
Suppose that
\BGE \delta\le  {\rho_k}/2.\label{Cond-Lower1}\EDE
Since $K_{\tau_0}\subset L_{z_1,\delta}$,  and $z_j\not\in L_{z_1,\rho_k}$, $m+1\le j\le n$, by Lemma \ref{farpoint}, we see that $|g_{\tau_0}'(z_j)|\ge C_{\rho_k}$, where $C_{\rho_k}>0$ depends only on $\kappa$ and $\rho_k$. By Koebe's $1/4$ theorem, we get
$$|Z_{\tau_0}(z_j)|\ge \dist(g_{\tau_0}(z_j),S_{K_{\tau_0}})\ge |g_{\tau_0}'(z_j)|\dist(z_j,K_{\tau_0})/4\ge C_{\rho_k}\rho_k|z_1|/8,\quad m+1\le j\le n.$$
Suppose now that
\BGE C_{\rho_k}\rho_k|z_1|/8\ge 2(\max_{1\le j\le n-1}V_j) 5C_1\sqrt{r|z_1|}.\label{Cond-Lower2}\EDE
Then we see that $Z_{\tau_0}(z_1),\dots,Z_{\tau_0}(z_n)$ satisfy the condition in Case A.

We will apply Lemma \ref{help-lower*} with $K=K_{\tau_0}$ and $U_0=U_{\tau_0}$. Let $I=\{1\}\cup \{1\le j\le n: r_j\le \dist(z_j,K_{\tau_0})\}$. We check the conditions of that lemma when $E_{\tau_0}$ happens. By the definition of $I$, we have $r_j\le \dist(z_j,K_{\tau_0})$ for $j\in I\sem\{1\}$. For $j=1$, since $\dist(z_1,K_{\tau_0})\ge b_1 r\gtrsim |z_1|$ and $r_1\le d_1\le |z_1|$, we have $r_1\lesssim \dist(z_1,K_{\tau_0})$. We have to check Condition (\ref{angle*}). First, (\ref{angle*}) holds for $j=1$ by Property (iii) of $E_{\tau_0}$. Second, for $2\le j\le m$, since $|z_j-z_1|\le \frac 12\dist(z_1,K_{\tau_0})$, by Koebe's $1/4$ theorem and distortion theorem, (\ref{angle*}) also holds for these $j$. Third, for $m+1\le j\le n$, by Lemma \ref{farpoint} and Koebe's $1/4$ theorem, we have $\dist(g_{\tau_0}(z_j),S_{K_{\tau_0}})\gtrsim \dist(z_j,L_{z_1,\delta})$. On the other hand, since $K_{\tau_0}\subset L_{z_1,\delta}\subset\{|z|\le r'\}$, where $r':=2|z_1|$, we have $|Z_{\tau_0}(z)-z|\le 5r'=10|z_1|$ for any $z\in\lin\HH\sem\lin{K_{\tau_0}}$ by Lemma \ref{small}. Thus, $|Z_{\tau_0}(z_j)|\lesssim |z_j|$. Since $\rho_k\ge 2\delta$, it is clear that $|z|\lesssim \dist(z,L_{z_1,\delta})$ for any $z\in\lin\HH\sem L_{z,\rho_k}$. So we see that (\ref{angle*}) also holds for $m+1\le j\le n$.

Let $\til\gamma$, $\til z_j$, $\til r_j$, $\til \tau_S$ and $\til \tau^z_r$ be as defined in Case A. Then $\til z_j=Z_{\tau_0}(z_j)$, $1\le j\le n$, satisfy the condition in Case A. By the result of Case A (if $|I|=n$) or the induction hypothesis (if $|I|<n$), we see that $$\PP[\til\tau^{\til z_j}_{\til r_j}<\til\tau_{\{|z|=V \sum_{j\in I} |\til z_j|\}},j\in I|\F_{\tau_0},E_{\tau_0}]\gtrsim F(\til z_{j_1},\dots,\til z_{j_{|I|}};\til r_{j_1},\dots,\til r_{j_{|I|}}),$$
where $V$ is the maximum of $V_j$, $1\le j\le n-1$, and the $V_n$  as in Case A. Let $\til E$ denote the event on the LHS of the above displayed formula. Since $|\til z_j-z_j|\le 5r'$, by Koebe's $1/4$ theorem, we see that $E_{\tau_0}\cap\til E\subset \bigcap_{j=1}^n \{\tau^{z_j}_{r_j}<\tau_{\{|z|=r''\}}\}$, where $r''=6r'+V \sum_{j\in I}(|z_j|+5r') \le V_n\sum_{j=1}^n |z_j|$ for some constant $V_n>1$. Thus,
\begin{align*}
  &\PP[\tau^{z_j}_{r_j}<\tau_{\{|z|=V_n\sum_{j=1}^n |z_j|\}}]\ge \PP[E_{\tau_0}\cap\til E]= \EE[E_{\tau_0}]\cdot\EE[\PP[\til E|\F_{\tau_0},E_{\tau_0}]]\\
  \gtrsim & F(z_1;r) \cdot\EE[F(\til z_{j_1},\dots,\til z_{j_{|I|}};\til r_{j_1},\dots,\til r_{j_{|I|}})|\F_{\tau_0},E_{\tau_0}] \gtrsim F(z_1,\dots,z_n;r_1,\dots,r_n),
\end{align*}
where the last inequality follows from  Lemma \ref{help-lower*} and that $\dist(z_1,K_{\tau_0})\le b_2 r$. We remark that the implicit constant in the above estimate depends on $\rho_k$ and $\rho_{k+1}$. This does not matter because $\rho_k$ and $\rho_{k+1}$ are constants once they are determined. Now we have finished the proof  of Case B$.k$ assuming Conditions (\ref{Cond-Lower1},\ref{Cond-Lower2}).

\smallbreak
\no{\bf Case C}: $z_1,\dots,z_n\in L_{z_1,\rho_1}$. This case is the complement of Case B, and we will reduce it to Case B.  Let
\[
e_n=\max_{1\le j\leq n}|z_j-z_1|.
\]
 From (\ref{L-ineq}) we know that $e_n\le \sqrt{2\rho_1}|z_1|$.

We apply Proposition \ref{L-shape} with $z_0=z_1$, $\delta= \frac{4N}{b_1}\sqrt{\rho_1}$ and $r=\frac{2e_n}{b_1}$. Let $\tau=\tau_0^\delta(z_1,r)$ and $E_{\tau_0}$ given by that proposition. Suppose $E_{\tau_0}$ happens. By Properties (i,iii) and Koebe's $1/4$ theorem, we have
$$|Z_{\tau_0}(z_1)|\le   \dist(g_{\tau_0}(z_1),S_{K_{\tau_0}})/C_0\le 4|g_{\tau_0}'(z_1)|\dist(z_1,K_{\tau_0})/C_0\le \frac{8b_2}{b_1C_n}   |g_{\tau_0}'(z_1)| e_n.$$
By Koebe's distortion theorem, we have
$$\max_{1\le j\le n} |Z_{\tau_0}(z_j)-Z_{\tau_0}(z_1)|\ge \frac 29
|g_{\tau_0}'(z_1)|e_n.$$
Thus, if $Z_{\tau_0}(z_s)$ has the smallest norm among $Z_{\tau_0}(z_j)$, $1\le j\le n$, then $$\max_{1\le j\le n} |Z_{\tau_0}(z_j)-Z_{\tau_0}(z_s)|\ge \frac{b_1 C_n}{72 b_2}|Z_{\tau_0}(z_s)|.$$ If $\rho_1$ satisfies that
\BGE \sqrt{2\rho_1}< \frac{b_1 C_n}{72 b_2},\label{Cond-Lower3}\EDE
then from (\ref{L-ineq}) we see that not all $Z_{\tau_0}(z_j)$, $1\le j\le n$, are contained in $L_{Z_{\tau_0}(z_s),\rho_1}$. After  reordering the points, we see that $Z_{\tau_0}(z_j)$, $1\le j\le n$, satisfy the condition in Case B.

We will apply Lemma \ref{help-lower*} with $K=K_{\tau_0}$ and $U_0=U_{\tau_0}$. Let $I=\{1,\dots,n\}$. We check the conditions of that lemma when $E_{\tau_0}$ happens.  Since $r_1\le |z_1-z_1|\le e_n$ and $\dist(z_1,K_{\tau_0})\ge 2e_1$, we have $r_1<\dist(z_1,K_{\tau_0})$. For $2\le j\le n$, since $r_j\le d_j\le |z_j-z_1|\le e_n$ and $\dist(z_1,K_{\tau_0})\ge 2e_n$, we see that $r_j\le \dist(z_j,K_{\tau_0})$. So $I$ satisfies the property there. We have to check Condition (\ref{angle*}). First, (\ref{angle*}) holds for $j=1$ by Property (iii) of $E_{\tau_0}$. Second, for $2\le j\le n$, since $|z_j-z_1|\le \frac 12\dist(z_1,K_{\tau_0})$, by Koebe's $1/4$ theorem and distortion theorem, (\ref{angle*}) also holds for these $j$.

Let $\til\gamma$, $\til z_j$, $\til r_j$, $\til \tau_S$ and $\til \tau^z_r$ be as defined in Case A. By the result of Case B  we see that $$\PP[\til\tau^{\til z_j}_{\til r_j}<\til\tau_{\{|z|=V \sum_{1\le j\le n} |\til z_j|\}},1\le j\le n|\F_{\tau_0},E_{\tau_0}]\gtrsim F(\til z_{1},\dots,\til z_{n};\til r_{1},\dots,\til r_{n}),$$
where $V$ is the $V_n$ as in Case B. Let $r'=2|z_1|$. Then $K_{\tau_0}\subset\{|z|\le r'\}$. So $|Z_{\tau_0}(z)-z|\le 5r'$ for $z\in\lin\HH\sem \lin{K_{\tau_0}}$. Let $\til E$ denote the event on the LHS of the above displayed formula. By Koebe's $1/4$ theorem, we see that $E_{\tau_0}\cap\til E\subset \bigcap_{j=1}^n \{\tau^{z_j}_{r_j}<\tau_{\{|z|=r''\}}\}$, where $r''=6r'+V \sum_{j=1}^n(|z_j|+5r')\le V_n \sum_{j=1}^n |z_j|$ for some constant $V_n>1$. Thus,
\begin{align*}
  &\PP[\tau^{z_j}_{r_j}<\tau_{\{|z|=V_n\sum_{j=1}^n |z_j|\}}]\ge \PP[E_{\tau_0}\cap\til E]= \EE[E_{\tau_0}]\cdot\EE[\PP[\til E|\F_{\tau_0},E_{\tau_0}]]\\
  \gtrsim & F(z_1;r) \cdot\EE[F(\til z_{1},\dots,\til z_{n};\til r_{1},\dots,\til r_{n})|\F_{\tau_0},E_{\tau_0}] \gtrsim F(z_1,\dots,z_n;r_1,\dots,r_n),
\end{align*}
where the last inequality follows from  Lemma \ref{help-lower*} and that $\dist(z_1,K_{\tau_0})\le b_2 r$. Now we have finished the proof  of Case C assuming Condition (\ref{Cond-Lower3}).

\smallbreak
In the end, we need to find $\rho_1,\dots,\rho_n$ such that Conditions (\ref{Cond-Lower1},\ref{Cond-Lower2},\ref{Cond-Lower3}) all hold. To do this, we may first use (\ref{Cond-Lower3}) to choose $\rho_1$. Once $\rho_k$ is chosen, we may use (\ref{Cond-Lower1},\ref{Cond-Lower2}) to choose $\rho_{k+1}$ because these two inequalities are satisfied when $\rho_{k+1}$ is sufficiently small given $\rho_k$.
\end{proof}

\appendixpage
%\addappheadtotocppendix}$
\begin{appendices}
\section{Proof of Theorem \ref{RZ-Thm3.1}}
In order to prove Theorem \ref{RZ-Thm3.1}, we need some lemmas. The proof of Theorem \ref{RZ-Thm3.1} will be given after the proof of Lemma \ref{key-lem2}. We still let $\gamma$ be a chordal SLE$_\kappa$ curve in $\HH$ from $0$ to $\infty$. Throughout the appendix, we use $C$ (without subscript) to denote a positive constant depending only on $\kappa$, and use $C_x$ to denote a positive constant depending only on $\kappa$ and some variable $x$. The value of a constant may vary between occurrences.

First, let's recall the one-point estimate and the boundary estimate for chordal SLE$_\kappa$. (see Lemma 2.6 and Lemma 2.5 in \cite[Lemma 2.6, Lemma 2.5]{RZ}).

\begin{lemma}  [One-point Estimate]
	Let $T$ be a stopping time for $\gamma$. Let $z_0\in\lin\HH$, $y_0=\Imm z_0\ge0$, and $R\ge r>0$. Then
	$$\PP[\tau^{z_0}_r<\infty |\F_T,\dist(z_0,K_T)\ge R]\le C \frac{P_{y_0}(r)}{P_{y_0}(R)}.$$
\end{lemma}

\begin{lemma} [Boundary Estimate]
	Let $T$ be a stopping time. Let $\xi_1$ and $\xi_2$ be a disjoint pair of crosscuts of $H_T$ such that
	\begin{enumerate}
		\item either $\xi_1$ disconnects $\gamma(T)$ from $\xi_2$ in $H_T$, or $\gamma(T)$ is an end point of $\xi_1$;
		\item among the three bounded components of $H_T\sem (\xi_1\cup\xi_2)$, the boundary of the unbounded  component does not contain $\xi_2$.
	\end{enumerate}
	Then
	$$\PP[\tau_{\xi_2}<\infty|\F_T]\le C e^{-\alpha \pi d_{H_T}(\xi_1,\xi_2)}.$$
	% where $d_{H_T}(\xi_1,\xi_2)$ is the extremal distance between $\xi_1$ and $\xi_2$ in $H_T$. %In particular, if $\xi_1$ and $\xi_2$ are contained in a pair of concentric circles, respectively, with radii $r<s$, then $\PP[\tau_{\xi_2}<\infty|\F_T]\lesssim (\frac rs)^{\alpha/2}$.
\end{lemma}

\begin{lemma}
	Let $m\in\N$. Let $z_j\in\lin\HH$, $y_j=\Imm z_j$, and $R_j\ge r_j>0$ be such that $|z_j|>R_j$, $1\le j\le m$. Let $D_j=\{|z-z_j|<r_j\}$ and $\ha D_j=\{|z-z_j|<R_j\}$, $1\le j\le m$. Let $\ha J_0,J_0,J_0'$ be three mutually disjoint Jordan curves in $\C$, which bound Jordan domains $\ha D_0,D_0,D_0'$, respectively, such that $\ha D_0\supset D_0\supset D_0'$ and $0\not\in \lin {D_0}$. Let $A=\ha D_0\sem \lin{D_0}$ be the doubly connected domain bounded by $\ha J_0$ and $J_0$. Suppose that $A\cap \ha D_j=\emptyset$, $1\le j\le m$, and there is some $n_0\in\{1,\dots,m\}$ such that $\ha D_0\cap \ha D_{n_0}=\emptyset$. Let $\xi_j=\pa D_j\cap \HH$, $\ha\xi_j=\pa \ha D_j\cap \HH$, $0\le j\le m$, and $\xi_0'=\pa D_0'\cap\HH$.
	Let $$E=\{\tau_{\xi_0}<\tau_{\ha \xi_1}\le \tau_{\xi_1}<\tau_{\ha \xi_2}\le \tau_{\xi_2}< \cdots< \tau_{\ha\xi_m}\le \tau_{\xi_m}<\tau_{\xi_0'}<\infty\}.$$
	Then
	$$\PP[{E}|\F_{\tau_{\xi_0}}]\le C^m e^{-\alpha\pi d_{\C}(J_0,\ha J_0)/2} \prod_{j=1}^m  \frac{P_{y_j}(r_j)}{P_{y_j}(R_j)}.$$
	\label{key-lem}
\end{lemma}

\begin{remark}
	The lemma is similar to and stronger than \cite[Theorem 3.1]{RZ},   which has the same conclusion but stronger assumption: $\ha D_j$, $1\le j\le m$, are all assumed to be disjoint from $\ha D_0$. Here we only require that  $\ha D_j$, $1\le j\le m$, are disjoint from $A$, and at least one of them: $\ha D_{n_0}$ is disjoint from $\ha D_0$. The condition that $\ha D_0\cap \ha D_{n_0}=\emptyset$ can not be removed. %If every $\ha D_j$, $1\le j\le m$, is contained in $D_0$, then we are not able to get the factor $e^{-\alpha\pi d_{\C}(J_0,\ha J_0)/2}$.
	The proof is similar to that of \cite[Theorem 3.1]{RZ}.
	The symbols such as $z_j,R_j,r_j$ in the statement of this lemma and the proof below are not related with the symbols with the same names in other parts of this paper, but are related with the symbols in \cite{RZ}.
\end{remark}

\begin{proof}
	We write $\tau_0=\tau_{\xi_0}$, $\ha\tau_j=\tau_{\ha\xi_j}$ and $\tau_j=\tau_{\xi_j}$, $1\le j\le m$, and $\tau_{m+1}=\tau_{\xi_0'}$.
	
	From the one-point estimate, we have
	\BGE \PP[\tau_j<\infty|\F_{\ha\tau_{j}}]\le  C\frac{P_{y_j}(r_j)}{P_{y_j}(R_j)},\quad 1\le j\le m.\label{1-pt*}\EDE
	Thus, $\PP[E|\F_{\tau_0}]\le C^m\prod_{j=1}^m \frac{P_{y_j}(r_j)}{P_{y_j}(R_j)}$. Now we need to derive the factor $e^{-\alpha\pi d_{\C}(J_0,\ha J_0)/2}$.
	
	By mapping $A$ conformally onto an annulus, we see that there is a Jordan curve $\rho$ in $A$ that disconnects $J_0$ from $\ha J_0$, such that
	\BGE d_{\C}(\rho,J_0 )=d_{\C}(\rho,\ha J_0 )= d_{\C}(J,\ha J_0)/2.\label{extremal}\EDE
	
	Let $T=\inf\{t\ge 0:\xi_0'\not\subset H_t\}$. Let $t\in[\tau_0,T)$. Each connected component $\eta$ of $\rho\cap H_t$ is a crosscut of $H_t$, and $H_t\sem\eta$ is the disjoint union of a bounded domain and an unbounded domain. We use $H_t^*(\eta)$ to denote the bounded domain. First, consider the connected components $\eta$ of $\rho\cap H_t$ such that $\xi_0'\subset H_t^*(\eta)$. If such $\eta$ is unique, we denote it by $\rho_t$. Otherwise, applying \cite[Lemma 2.1]{RZ}, we may find the unique component $\eta_0$, such that $H_t^*(\eta_0)$ is the smallest among all of these $H_t^*(\eta)$. Again we use $\rho_t$ to denote this $\eta_0$. Let $\ha U^\rho_t=H_t^*(\rho_t)$. Then $\xi_0'\subset \ha U^\rho_t$. Next, consider the connected components $\eta$ of $\rho\cap H_t$ such that $H_t^*(\eta)\subset \ha U^\rho_t\sem \xi_0'$. Let the union of $H_t^*(\eta)$ for these $\eta$ be denoted by $U^\rho_t$. Then we have $U^\rho_t\subset\ha U^\rho_t$ and $U^\rho_t\cap \xi_0'=\emptyset$.

\begin{comment}	
\begin{figure}
\hfill
\includegraphics[width=0.45\textwidth]{1.png}%
\hfill
\includegraphics[width=0.45\textwidth]{2.png}%
\hfill
\caption{The two pictures above illustrate $\ha U^\rho_t$ and $U^\rho_t$, respectively. The red curve is $\gamma$ up to time $t$, the big circle is $\rho$, and the small circle is $\xi_0'$. The connected components of $\rho\cap H_t$ that disconnects $\xi_0'$ from $\infty$ are labeled as $\rho_t,\rho_t',\rho_t''$, where $\rho_t$ is closest to $\xi_0'$ in $H_t$. The grey region on the left picture is $\ha U^\rho_t$; and the grey region on the right picture is $U^\rho_t$.}
\end{figure}
\end{comment}
	
	Now we define a family of events. % Recall that $D_j\cap\HH=\HH^*(\xi_j)$, $1\le j\le m$.
	\begin{itemize}
		\item Let $A_{(0,1)}$ be the event that $\tau_0<\ha \tau_1\wedge T$ and $D_1\cap\HH\subset U^\rho_{\tau_0}$.
		\item For $1\le j\le n_0-1$, let $A_{(j,j)}$ be the event that $\tau_{j-1}<\tau_j<T$,  and $D_j\cap\HH\not\subset U^\rho_{\tau_{j-1}}$, but $D_j\cap\HH\subset U^\rho_{\tau_{j}}$.
		\item For $1\le j\le n_0-1$, let $A_{(j,j+1)}$ be the event that $\tau_j<\ha \tau_{j+1}\wedge T$,  and $D_j\cap\HH\not\subset U^\rho_{\tau_{j}}$, but $D_{j+1}\cap\HH\subset U^\rho_{\tau_{j}}$.
		\item For $n_0\le j\le m$, let $A_{(j,j)}$ be the event that $\tau_{j-1}<\tau_{j}<T$, and $D_j\cap\HH\not\subset \ha U^\rho_{\tau_{j-1}}$, but $D_j\cap\HH\subset \ha U^\rho_{\tau_{j}}$.
		\item For $n_0\le j\le m-1$, let $A_{(j,j+1)}$ be the event that $\tau_{j}<\ha \tau_{j+1}\wedge T$, and $D_j\cap\HH\not\subset \ha U^\rho_{\tau_{j}}$, but $D_{j+1}\cap\HH\subset \ha U^\rho_{\tau_{j}}$.
		\item Let $A_{(m,m+1)}$ be the event that $\tau_m<\tau_{m+1}\wedge T$ and $D_m\cap\HH\not\subset \ha U^\rho_{\tau_{m}}$.
	\end{itemize}
	
	Let $I=\{(j,j+1):0\le j\le m\}\cup\{(j,j):1\le j\le m\}$.  We claim that $E\subset\bigcup_{\iota\in I} A_\iota$. To see this, note that, if none of the events $A_{(j,j+1)}$, $0\le j\le n_0-1$, and $A_{(j,j)}$, $1\le j\le n_0-1$, happens, then $D_{n_0}\cap\HH\not\subset U^\rho_{\tau_{n_0}}$. Since $ D _{n_0} $ is disjoint from $\ha D_0$, we can conclude that $D _{n_0}\cap\HH\not\subset \ha U^\rho_{\tau_{n_0}}$. In fact, if $D _{n_0}\cap\HH \subset \ha U^\rho_{\tau_{n_0}}$, then from $D _{n_0}\cap\ha D_0=\emptyset$, $\rho\subset \ha D_0$, and $\rho$ surrounds $\xi_0'$, we may find a connected component $\eta$ of $\rho\cap H_{\tau_{n_0}}$ that disconnects $D _{n_0}\cap\HH$ from $\xi_0'$ in $H_{\tau_{n_0}}$. Since $D _{n_0}\cap\HH,\xi_0' \subset \ha U^\rho_{\tau_{n_0}}$, we have $\eta\subset \ha U^\rho_{\tau_{n_0}}$. From the definitions of $\rho_{n_0}$ and $\ha U^\rho_{n_0}$, we see that $\eta$ does not disconnect $\xi_0'$ from $\infty$ in $H_{\tau_{n_0}}$. Thus,
	$D _{n_0}\cap\HH\subset H_{\tau_{n_0}}^*(\eta)\subset \ha U^\rho_{\tau_{n_0}}$, and $\xi_0'\cap H_{\tau_{n_0}}^*(\eta)=\emptyset$. This shows that $D _{n_0}\cap\HH\subset U^\rho_{\tau_{n_0}}$, which is a contradiction. Since $D_{n_0}\cap\HH\not\subset \ha U^\rho_{\tau_{n_0}}$,  one of the events $A_{(j,j)}$ and $A_{(j,j+1)}$, $n_0\le j\le m$, must happen. So the claim is proved. We will finish the proof by showing that
	\BGE \PP[E\cap A_{\iota}|\F_{\tau_0}]  \le C^m e^{-\alpha\pi d_{\C}(J_0,\ha J_0)/2} \prod_{j=1}^m  \frac{P_{y_j}(r_j)}{P_{y_j}(R_j)}, \quad \iota\in I.\label{conclusion*}\EDE
	
	\no{\bf Case 1.} Suppose $A_{(0,1)}$ occurs. Then at time $\tau_0$, there is a connected component, denoted by $\til \rho_{\tau_0}$, of $\rho\cap H_{\tau_0}$, that disconnects $\ha \xi_1$ from both $\xi_0'$ and $\infty$ in $H_{\tau_0}$. Since $\xi_0'\subset D_0\cap\HH\subset H_{\tau_0}$ and $\gamma(\tau_0)\in\pa D_0$, we see that $\til\rho_{\tau_0}$ disconnects $\ha\xi_1$ also from $\gamma(\tau_0)$ in $H_{\tau_0}$. Since $\ha\xi_1$ is disjoint from $A$, it is  contained in either $D_0$ or $\C\sem \ha D_0$. If $\ha\xi_1$ is contained in $D_0$ (resp.\ $\C\sem \ha D_0$), then $J_0\cap H_{\tau_0}$ (resp.\ $\ha J_0\cap H_{\tau_0}$) contains a connected component, denoted by $\eta_{\tau_0}$, which disconnects $\ha\xi_1$ from $\til\rho_{\tau_0}$ and $\infty$ in $H_{\tau_0}$. Using the boundary estimate and (\ref{extremal}), we get
	$$ \PP[\ha\tau_1<\infty|\F_{\tau_0},A_{(0,1)}] \le C e^{-\alpha \pi d_{H_{\tau_0}}(\til\rho_{\tau_0}, \eta_{\tau_0})}\le   C e^{-\alpha\pi d_{\C}(J_0,\ha J_0)/2},$$
	which together with (\ref{1-pt*}) implies that (\ref{conclusion*}) holds for $\iota=(0,1)$.
	
	\no{\bf Case 2.} Suppose for some $1\le j\le n_0-1$, $A_{(j,j+1)}$ occurs. See Figure \ref{Case23}. Then at time $\tau_j$, there is a connected component, denoted by $\til \rho_{\tau_j}$, of $\rho\cap H_{\tau_j}$, that disconnects $\ha \xi_{j+1}$ from both $\xi_j$ and $\infty$ in $H_{\tau_j}$.  Since $\gamma(\tau_j)\in\xi_j$, we see that $\til\rho_{\tau_j}$ disconnects $\ha\xi_{j+1}$ also from $\gamma(\tau_j)$ in $H_{\tau_j}$. According to whether $\xi_{j+1}$ belongs to $D_0$ or $\C\sem\ha D_0$, we may find a connected component $\eta_{\tau_j}$ of $J_0\cap H_{\tau_0}$ or $\ha J_0\cap H_{\tau_0}$ that disconnects $\ha\xi_{j+1}$ from $\til\rho_{\tau_j}$ and $\infty$ in $H_{\tau_j}$. Using the boundary estimate and (\ref{extremal}), we get
	$$ \PP[\ha\tau_{j+1}<\infty|\F_{\tau_j},A_{(j,j+1)},\tau_j<\ha\tau_{j+1}] \le C e^{-\alpha \pi d_{H_{\tau_j}}(\til\rho_{\tau_j}, \eta_{\tau_j})}\le   C e^{-\alpha\pi d_{\C}(J_0,\ha J_0)/2},$$
	which together with (\ref{1-pt*}) implies that (\ref{conclusion*}) holds for $\iota=(j,j+1)$, $1\le j\le n_0-1$.
	
	\no{\bf Case 3.} Suppose for some $n_0\le j\le m$, $A_{(j,j+1)}$ occurs. See Figure \ref{Case23}. We write $\xi_{m+1}=\xi_0'$. Then $\rho_{\tau_j}$ disconnects $\xi_{j+1}$ from $\gamma(\tau_j)$ and $\infty$ in $H_{\tau_j}$. According to whether $\xi_{j+1}$ belongs to $D_0$ or $\C\sem\ha D_0$, we may find a connected component $\eta_{\tau_j}$ of $J_0\cap H_{\tau_0}$ or $\ha J_0\cap H_{\tau_0}$ that disconnects $\ha\xi_{j+1}$ from $\rho_{\tau_j}$ and $\infty$ in $H_{\tau_j}$. Using the boundary estimate and (\ref{extremal}), we get
	$$ \PP[\ha\tau_{j+1}<\infty|\F_{\tau_j},A_{(j,j+1)},\tau_j<\ha\tau_{j+1}] \le C e^{-\alpha \pi d_{H_{\tau_j}}(\rho_{\tau_j}, \eta_{\tau_j})}\le   C e^{-\alpha\pi d_{\C}(J_0,\ha J_0)/2},$$
	which together with (\ref{1-pt*}) implies that (\ref{conclusion*}) holds for $\iota=(j,j+1)$, $n_0\le j\le m$.
	
\begin{figure}
	\hfill
	\includegraphics[width=0.45\textwidth]{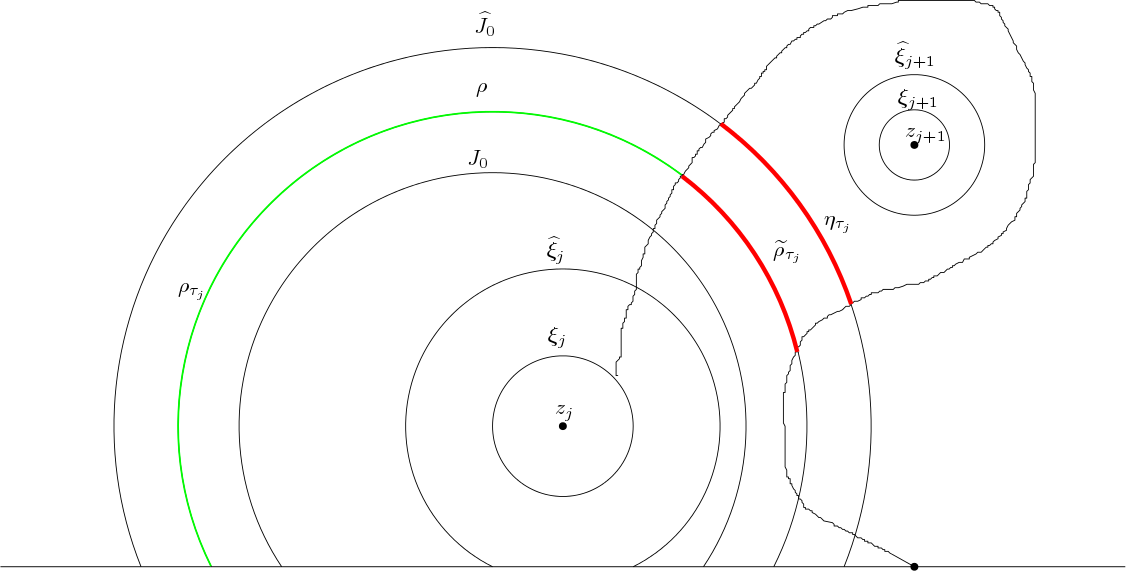}%
	\hfill
	\includegraphics[width=0.45\textwidth]{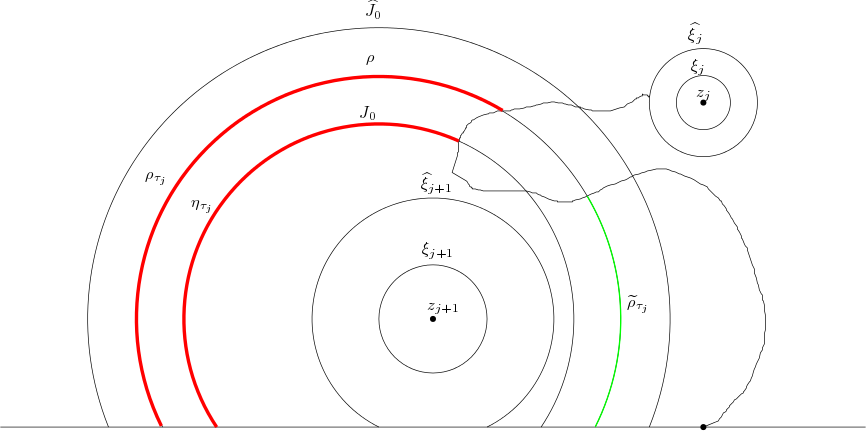}%
	\hfill
	\caption{The two pictures above illustrate Case 2 (left) and Case 3 (right), respectively. In both pictures, the zigzag curve is $\gamma$ up to $\tau_j$, and the three big arcs are $\ha J_0$, $\rho$ and $J_0$ restricted to $\HH$. But the positions of the two pairs of concentric circles $(\ha\xi_j,\xi_j)$ and $(\ha\xi_{j+1},\xi_{j+1})$ are swapped. In both pictures, the pairs of acs that contribute the factors from the boundary estimate ($\til\rho_{\tau_j}$ and $\eta_{\tau_j}$ on the left, $\rho_{\tau_j}$ and $\eta_{\tau_j}$ on the right) are labeled and colored red. We also labeled $\rho_{\tau_j}$ on the left and $\til\rho_{\tau_{j}}$ on the right, and colored both of them green. One can   see the difference between $\ha U_{\tau_j}$ and $U_{\tau_j}$ as they are bounded by $\rho_{\tau_{j}}$ and $\til \rho_{\tau_{j}}$, respectively. }
\label{Case23}
\end{figure}
	
	\no{\bf Case 4.} Suppose for some $n_0\le j\le m-1$, $A_{(j,j)}$ occurs. Define a stopping time
	$$\sigma_j=\inf\{t\ge \tau_{j-1}:D_j\cap\HH\subset \ha U^\rho_{t}\}.$$
	Then $\tau_{j-1}\le \sigma_j\le \tau_j$. From \cite[Lemma 2.2]{RZ}, we know that
	\begin{itemize}
		\item $\gamma(\sigma_j) $ is an endpoint of $\rho_{\sigma_j}$;
		\item $D_j\cap\HH\subset \ha U^\rho_{\sigma_j}$.
	\end{itemize}
	The second property implies that $\tau_{j-1}< \sigma_j< \tau_j$. Now we define two events.
	Let $F_<=\{\sigma_j<\ha\tau_j\}$ and $F_\ge=\{\ha\tau_j\le \sigma_j<\tau_j\}$. Then $A_{(j,j)}\subset F_<\cup F_\ge$.
	
	\no{\bf Case 4.1.} Suppose $F_\ge$ occurs. Let $N=\lceil \log(R_j/r_j)\rceil\in\N$. Let $\zeta_k=\{z\in\HH:|z-z_j|=(R_j^{N-k} r_j^k)^{1/N}\}$, $0\le k\le N$. Note that $\zeta_0=\ha\xi_j$ and $\zeta_N=\xi_j$. Then $F_{\ge }\subset \bigcup_{k=1}^N F_k$, where
	$$F_k:=\{\tau_{\zeta_{k-1}}\le \sigma_j<\tau_{\zeta_k}<\infty\},\quad 1\le k\le N.$$
	See Figure \ref{Case45} for an illustration of $F_k$.
	If $F_k$ occurs, then $\zeta_k\subset \ha U^\rho_{\sigma_j}$.  Since $\zeta_{k-1}\cap H_{\sigma_j}$ has a connected component $\zeta_{k-1}^{\sigma_j}$, which disconnects $\zeta_k$ from $\rho_{\sigma_j}$ in $H_{\sigma_j}$, by the boundary estimate, we get
	$$\PP[\tau_{\zeta_k}<\infty|\F_{\sigma_j},F_k]\le C e^{-\alpha\pi d_{H_{\sigma_j}}(\rho_{\sigma_j},\zeta_{k-1}^{\sigma_j})}. $$
	According to whether $\zeta_k$ belongs to $D_0$ or $\ha D_0$, we may find a connected component $\eta_{\sigma_j}$ of $J_0\cap H_{\sigma_j}$ or $\ha J_0\cap H_{\sigma_j}$ that disconnects $\zeta_{k-1}^{\sigma_j}$ from $\rho_{\sigma_j}$ and $\infty$ in $H_{\sigma_j}$. Moreover, we may find a connected component $\zeta_0^{\sigma_j}$ of $\zeta_0\cap H_{\sigma_j}$ that disconnects $\eta_{\sigma_j}$ from $\zeta_{k-1}^{\sigma_j}$.
	From the composition law of extremal length and (\ref{extremal}) we get
	$$d_{H_{\sigma_j}}(\rho_{\sigma_j},\zeta_{k-1}^{\sigma_j})\ge d_{H_{\sigma_j}}(\rho_{\sigma_j},\eta_{\sigma_j})+d_{H_{\sigma_j}}(\zeta_0^{\sigma_j},\zeta_{k-1}^{\sigma_j})
	\ge  \frac 12 d_{\C}(J_0,\ha J_0)+\frac{k-1}{2\pi N} \log\Big(\frac{R_j}{r_j}\Big)\Big .$$
	Thus, we get
	$$\PP[\tau_{\zeta_k}<\infty|\F_{\sigma_j},F_k] \le C e^{-\alpha\pi d_{\C}(J_0,\ha J_0)/2}\Big(\frac{r_j}{R_j}\Big)^{\frac\alpha2 \frac{k-1}N}. $$
	
	From the one-point estimate, we get
	$$ \PP[F_k|\F_{\tau_{j-1}}, \tau_{j-1}<\ha\tau_j]\le C \frac{P_{y_j}((R_j^{N-k+1} r_j^{k-1})^{1/N})}{P_{y_j}(R_j)};$$
	$$\PP[\tau_j<\infty |\F_{\tau_{\zeta_k}},F_k]\le C \frac{P_{y_j}(r_j)}{P_{y_j}((R_j^{N-k} r_j^{k})^{1/N})}.$$
	The above three displayed formulas together  imply that
	$$ \PP[\tau_j<\infty,F_k|\F_{\tau_{j-1}},\tau_{j-1}<\ha\tau_j]\le  C e^{-\alpha\pi d_{\C}(J_0,\ha J_0)/2} \Big(\frac{r_j}{R_j}\Big)^{\frac\alpha2 \frac{k-1}N}\Big(\frac{r_j}{R_j}\Big)^{-\alpha/N}\frac{P_{y_j}(r_j)}{P_{y_j}(R_j)} .$$ %\label{Fk*}\EDE
	Since $F_{\ge }\subset \bigcup_{k=1}^N F_k$, by summing up the above inequality over $k$, we get
	\begin{align}
		\PP[\tau_j<\infty,F_{\ge} |\F_{\tau_{j-1}},\tau_{j-1}<\ha\tau_j]&\le  C e^{-\alpha\pi d_{\C}(J_0,\ha J_0)/2} \frac{P_{y_j}(r_j)}{P_{y_j}(R_j)}\left[\Big(\frac{r_j}{R_j}\Big)^{-\alpha/N} \frac{1-(\frac{r_j}{R_j})^{\alpha/2}}{1-(\frac{r_j}{R_j})^{\alpha/(2N)}}\right]\nonumber\\
		&\le  C e^{-\alpha\pi d_{\C}(J_0,\ha J_0)/2} \frac{P_{y_j}(r_j)}{P_{y_j}(R_j)},\label{Fk*}
	\end{align}
	where the second inequality holds because the quantity inside the square bracket is bounded above by  $\frac{e^\alpha}{1-e^{-\alpha/4}}$. To see this, consider the cases $R_j/r_j\le e$ and $R_j/r_j>e$ separately.
	
	\no{\bf Case 4.2.} Suppose $F_{<}$ occurs. Then  $\ha\xi_j\subset \ha U^\rho_{\sigma_j}$. According to whether $\ha\xi_j$ belongs to $D_0$ or $\ha D_0$, we may find a connected component $\eta_{\sigma_j}$ of $J_0\cap H_{\sigma_j}$ or $\ha J_0\cap H_{\sigma_j}$ that disconnects $\ha\xi_j$ from $\rho_{\sigma_j}$ and $\infty$ in $H_{\sigma_j}$. By the boundary estimate,  we get
	$$\PP[\ha\tau_j <\infty|\F_{\sigma_j},F_{<}]\le C e^{-\alpha \pi d_{H_{\sigma_j}}(\rho_{\sigma_j}, \eta_{\sigma_j})}\le Ce^{-\alpha\pi d_{\C}(J_0,\ha J_0)/2}, $$
	which together with (\ref{1-pt*}) implies that
	\BGE \PP[\tau_j<\infty, F_{<}|\F_{\tau_{j-1}}]\le C e^{-\alpha\pi d_{\C}(J_0,\ha J_0)/2} \frac{P_{y_j}(r_j)}{P_{y_j}(R_j)}.\label{F0*}\EDE
	
	Combining (\ref{Fk*}) and (\ref{F0*}),  we get
	$$\PP[\tau_j<\infty, A_{(j,j)}|\F_{\tau_{j-1}},\tau_{j-1}<\ha\tau_j]\le C e^{-\alpha\pi d_{\C}(J_0,\ha J_0)/2} \frac{P_{y_j}(r_j)}{P_{y_j}(R_j)},$$
	which together with (\ref{1-pt*}) implies that (\ref{conclusion*}) holds for $\iota=(j,j)$, $n_0\le j\le m$.
	
	\no{\bf Case 5.} Suppose for some $1\le j\le n_0-1$, $A_{(j,j)}$ occurs.  Define a stopping time
	$$\sigma_j=\inf\{t\ge \tau_{j-1}:D_j\cap\HH\subset   U^\rho_{t}\}.$$
	To derive properties of $\sigma_j$, we claim that the following are true.
	\begin{itemize}
		\item[(i)] If $D_j\cap\HH \subset  H_{t_0}\sem U^\rho_{t_0}$, then there is  $\eps>0$ such that $D_j\cap\HH \subset H_t\sem  U^\rho_{t}$ for $t_0\le t<t_0+\eps$;
		\item[(ii)] If $D_j\cap\HH \subset   U^\rho_{t_0}$, and if $\gamma(t_0)$ is not an endpoint of a connected component of $\rho\cap H_{t_0}$ that disconnects $D_j\cap\HH $ from $\infty$ in $H_{t_0}$, then there is $\eps>0$ such that $D_j\cap\HH \subset   U^\rho_{t }$ for $t_0-\eps<t\le t_0$.
	\end{itemize}
	
	To see that (i) holds, we consider two cases. Case 1. $D_j\cap\HH \subset  H_{t_0}\sem \ha U^\rho_{t_0}$. From \cite[Lemma 2.2]{RZ}, there is $\eps>0$ such that for $t_0\le t<t_0+\eps$, $D_j\cap\HH \subset H_t\sem  \ha U^\rho_{t}$, which implies that $D_j\cap\HH \subset H_t\sem   U^\rho_{t}$. Case 2. $D_j\cap\HH \subset   \ha U^\rho_{t_0}\sem U^\rho_{t_0}$. Then there is a curve $\zeta$ in $H_{t_0}$, which connects $\xi_0'$ with $D_j$, and does not intersect $\rho$. In this case, there is $\eps>0$ such that for $t_0\le t<t_0+\eps$, $\zeta\subset H_t$ and $D_j\cap\HH\subset H_t$, which imply that  $D_j\cap\HH \subset H_t\sem  U^\rho_{t}$.
	
	Now we consider (ii). %Since $\rho_{t_0}$ is a connected component of $\rho\cap H_{t_0}$ that disconnects $\ha U^\rho_{t_0}\supset U^\rho_{t_0}\supset D_j\cap \HH$ from $\infty$, $\gamma(t_0)$ is not an endpoint of $\rho_{t_0}$. Since $D_j\cap\HH \subset   U^\rho_{t_0}\subset \ha U^\rho_{t_0}$, from \cite[Lemma 2.2]{RZ}, there is $\eps_1>0$ such that for $t_0-\eps_1< t\le t_0$, $D_j\cap\HH \subset    \ha U^\rho_{t }$.
	Since $D_j\cap\HH \subset   U^\rho_{t_0}$, there is a connected component $\zeta$ of $\rho\cap H_{t_0}$, which is contained in $\ha U^\rho_{t_0}$,  and disconnects $D_j\cap\HH$ from $\xi_0'$ and $\infty$ in $H_{t_0}$. From the assumption, $\gamma(t_0)$ is not an end point of $\zeta$. By the continuity of $\gamma$, there is $\eps_1>0$ such that $\gamma[t_0-\eps_1,t_0]\cap \lin\zeta=\emptyset$. This implies that, for $t_0-\eps_1<t\le t_0$, $\zeta$ is also a crosscut of $H_t$. Since $H_t$ is simply connected, $\zeta$ also disconnects $D_j\cap\HH$ from $\xi_0'$ and $\infty$ in $H_{t}$.
	Since $\rho_{t_0}$ is a connected component of $\rho\cap H_{t_0}$ that disconnects $\ha U^\rho_{t_0}\supset U^\rho_{t_0}\supset D_j\cap \HH$ from $\infty$, $\gamma(t_0)$ is also not an endpoint of $\rho_{t_0}$. Since $\zeta\subset \ha U^\rho_{t_0}$, from \cite[Lemma 2.2]{RZ}, there is $\eps\in(0,\eps_1)$ such that for $t_0-\eps< t\le t_0$, $\zeta \subset    \ha U^\rho_{t }$, which implies that $D_j\cap \HH\subset U^\rho_t$.
	
	From (i) and (ii) we conclude that
	\begin{itemize}
		\item $\gamma(\sigma_j) $ is an endpoint of a connected component of $\rho\cap H_{\sigma_j}$ that disconnects $D_j\cap\HH $ from $\infty$ in $H_{\sigma_j}$. Let this crosscut be denoted by $\til\rho_{\sigma_j}$.
		\item $D(z_j,r_j)\cap\HH\subset   U^\rho_{\sigma_j}$.
	\end{itemize}
	Following the proof of Case 4 with $\til\rho_{\sigma_j}$ and $U^\rho_{\sigma_j}$ in place of $\rho_{\sigma_j}$ and $\ha U^\rho_{\sigma_j}$, respectively, we conclude that (\ref{conclusion*}) holds for $\iota=(j,j)$, $1\le j\le n_0-1$. See Figure \ref{Case45} for an illustration of the subcase $F_<$ of Case 5. The proof is now complete.
\end{proof}

\begin{figure}
	\hfill
	\includegraphics[width=0.45\textwidth]{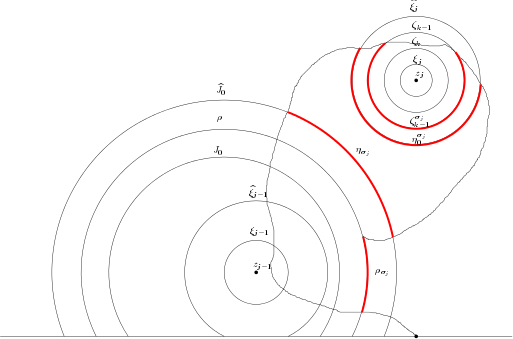}%
	\hfill
	\includegraphics[width=0.45\textwidth]{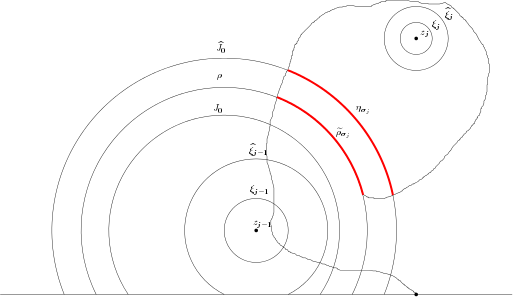}%
	\hfill
	\caption{The two pictures above illustrate the subcase $F_k\subset F_{\ge}$ of Case 4 (left) and the subcase $F_<$ of Case 5 (right), respectively. In both pictures, the zigzag curve is $\gamma$ up to $\sigma_j$, and the three big arcs are $\ha J_0$, $\rho$ and $J_0$ restricted to $\HH$. The  acs that contribute the factors from the boundary estimate ($\til\rho_{\tau_j}$, $\eta_{\tau_j}$, $\zeta^{\sigma_j}_{0}$ and $\zeta^{\sigma_j}_{k-1}$ on the left, $\rho_{\tau_j}$ and $\eta_{\tau_j}$ on the right) are labeled and colored red. }
	\label{Case45}
\end{figure}

Let $\Xi$ be a family of mutually disjoint circles with centers in $\lin\HH$, each of which does not pass through or enclose $0$. Define a partial order on $\Xi$ such that $\xi_1<\xi_2$ if $\xi_2$ is enclosed by $\xi_1$. One should keep in mind that a smaller element in $\Xi$ has bigger radius, but will be visited earlier (if it happens) by a curve started from $0$.

Suppose that $\Xi$ has a partition $\{\Xi_e\}_{e\in\cal E}$ with the following properties:
\begin{itemize}
	\item For each $e\in\cal E$, the elements in $\Xi_e$ are concentric circles with radii forming a geometric sequence with common ratio $1/4$.
	We denote the common center $z_e$, the biggest radius $R_e$, and the smallest radius $r_e$, and let $y_e=\Imm z_e$.
	\item Let $A_e=\{r_e\le |z-z_0|\le R_e\}$ be the closed annulus associated with $\Xi_e$, which is a single circle if $R_e=r_e$, i.e., $|\Xi_e|=1$. Then the annuli $A_e$, $e\in\cal E$, are mutually disjoint.
\end{itemize}
Note that every $\Xi_e$ is a totally ordered set w.r.t.\ the partial order on $\Xi$.

\begin{lemma}
	Suppose that $J_1$ and $J_2$ are disjoint Jordan curves in $\C$, which are disjoint from all $\xi\in\Xi$. Suppose that $0$ is not contained in or enclosed by $J_1$, $J_1$ is enclosed by $J_2$, and that every $\xi\in\Xi$ that lies in the doubly connected domain bounded by $J_1$ and $J_2$ disconnects $J_1$ from $J_2$. Suppose $\xi_a<\xi_b\in\Xi$ are both enclosed by $J_1$, and $\xi_c\in\Xi$  neither encloses $J_2$, or is enclosed by $J_2$.  Let $E$ denote the event that $\tau_\xi<\infty$ for all $\xi\in\Xi$, and $\tau_{\xi_a}<\tau_{\xi_c}<\tau_{\xi_b}$. Then
	$$\PP[E]\le C_{|\cal E|} e^{-\frac{\alpha}{4|\cal E|}\pi d_{\C}(J_1,J_2)} \prod_{e\in\cal E} \frac{P_{y_e}(r_e)}{P_{y_e}(R_e)},$$ \label{key-lem2}
	where $C_{|\cal E|}\in(0,\infty)$ depends only on $\kappa$ and $|\cal E|$.
\end{lemma}

%\no{\bf Discussion.} Suppose $\gamma$ visits all $\xi\in\Xi$. For $\xi_1,\xi_2\in\Xi$, if $\xi_1<\xi_2$, then $\gamma$ will visit $\xi_1$ before $\xi_2$. Other than these constraints, $\gamma$ can visit the elements in $\Xi$ in any order. The simplest case is that $\gamma$ does not jump back and forth between different groups $\{\Xi_e:e\in\cal E\}$. This means that $\gamma$ first visits all circles in $\Xi_{e_1}$ for some $e_1\in\cal E$ before all other circles in $\Xi$, then visits all circles in $\Xi_{e_2}$ for some $e_2\in{\cal E}\sem\{e_1\}$ before circles in $\Xi\sem(\Xi_{e_1}\cup\Xi_{e_2})$, and so on. In this case, we can easily use the $1$-point estimate and DMP to get the righthand side of the above formula. We use Lemma \ref{key-lem} to deal with the general cases. The key point is that $\gamma$ has to pay a price to  jump back and forth between different $\Xi_e$'s due to the factor $e^{-\alpha\pi d_{\C}(J_0,\ha J_0)/2}$ given in Lemma \ref{key-lem}.

\no{\bf Discussion.} From \cite[Theorem 3.2]{RZ}, we know that $\PP[\tau_\xi<\infty,\xi\in\Xi]\le C_{|\cal E|} \prod_{e\in\cal E} \frac{P_{y_e}(r_e)}{P_{y_e}(R_e)}$. Now we need to derive the additional factor $e^{-\frac{\alpha}{4|\cal E|}\pi d_{\C}(J_1,J_2)}$ using the condition $\tau_{\xi_a}<\tau_{\xi_c}<\tau_{\xi_b}$. % By increasing the value of $C_{|\cal E|}$, we may assume that $d_{\C}(J_1,J_2)$ is bigger than a constant depending only on $\kappa$ and $|\cal E|$.

\begin{proof} We write $\N_n$ for $\{k\in\N:k\le n\}$.
	Let $S$ denote the set of bijections $\sigma:\N_{|\Xi|}\to \Xi$ such that $\xi_1<\xi_2$ implies that $\sigma^{-1}(\xi_1)<\sigma^{-1}(\xi_2)$, and $\sigma^{-1}(\xi_a)<\sigma^{-1}(\xi_c)<\sigma^{-1}(\xi_b)$. Let
	$$E^\sigma=\{\tau_{\sigma(1)}<\tau_{\sigma(2)}<\cdots<\tau_{\sigma(|\Xi|)}<\infty\},\quad \sigma\in S.$$
	Then we have
	\BGE E=\bigcup_{\sigma\in S} E^\sigma.\label{E-sigma}\EDE
	We will derive an upper bound of $\PP[E^\sigma]$ in (\ref{E-sigma-est}).
	
	Fix $\sigma\in S$. For $e\in\cal E$, if there is no $\xi\in\Xi$ such that $\xi>\max\Xi_e$, then we say that $e$ is a maximal element in $E$. In this case, we define $\ha \Xi_e=\Xi_e$ and $\xi_e^*=\max\Xi_e$. If $e$ is not a maximal element in $E$, let
	$\xi_e^*$ denote the first $\xi>\max \Xi_e$ that is visited by $\gamma$ on the event $E^\sigma$, and define $\ha\Xi_e=\Xi_e\cup\xi_e^*$. This definition certainly depends on $\sigma$. Label the elements of $\ha\Xi_e$ by $\xi^e_0<\cdots<\xi^e_{N_e}=\xi_e^*$, where $N_e=|\ha\Xi_e|-1$.
	
	For $e\in E$, define
	$$J_e=\{1\le n\le N_e:\sigma^{-1}(\xi^e_n)>  \sigma^{-1}(\xi^e_{n-1})+1\}.$$
	Roughly speaking, $n\in J_e$ means that between $\tau_{\xi^e_{n-1}}$ and $\tau_{\xi^e_n}$, $\gamma$ visits other element in $\Xi$ that it has not visited before on the event $E_\sigma$.
	
	Order the elements of $J_e\cup\{0\}$ by $0=s_e(0)<\cdots<s_e(M_e)$, where $M_e=|J_e|$. Set $s_e(M_e+1)=N_e+1$. Every $\ha\Xi_e$ can be partitioned into $M_e+1$  subsets:
	$$\ha\Xi_{(e,j)}=\{\xi^e_n:s_e(j)\le n\le s_e(j+1)-1\},\quad 0\le j\le M_e.$$
	The meaning of the partition is that, after $\gamma$ visits the first element in $\ha\Xi_{(e,j)}$, which must be $\xi^e_{s_e(j)}$, it then visits all elements in $\ha\Xi_{(e,j)}$ without visiting any other circles in $\Xi$ that it has not visited before.
	Let $I=\{(e,j):e\in{\cal E}, 0\le j\le M_e\}$. Then $\{\ha\Xi_{\iota}:\iota\in I\}$ is a cover of $\Xi$. Note that every $\sigma^{-1}(\ha \Xi_\iota)$, $\iota\in I$, is a connected subset of $\Z$.
	
	For $\iota\in I$, let $e_\iota$ denote the first coordinate of $\iota$, $z_\iota=z_{e_\iota}$ and $y_\iota=\Imm z_\iota$.
	Define $P_\iota$ for each $\iota\in I$. If $\max \ha\Xi_\iota \in \Xi_{e_\iota}$, define $P_\iota= \frac{P_{y_{\iota}}(R_{\max \ha\Xi_\iota})}{P_{y_{\iota}}(R_{\min\ha \Xi_\iota})}$, where we use $R_\xi$ to denote the radius of $\xi$. If $\max \ha\Xi_\iota \not\in \Xi_{e_\iota}$, which means $\max \ha\Xi_\iota=\xi^*_{e_\iota}> \max \Xi_{e_\iota}$, then we consider two subcases. If $\ha\Xi_\iota$ contains only one element (i.e., $\xi^*_{e_\iota}$) or two elements (i.e., $\xi^*_{e_\iota}$ and $ \max \Xi_{e_\iota}$), then let $P_\iota=1$; otherwise let $P_\iota= \frac{P_{y_{\iota}}(R_{\max \Xi_{e_\iota}})}{P_{y_{\iota}}(R_{\min \ha\Xi_\iota})}$.  From the one-point estimate, we get
	\BGE \PP[\tau_{\max\ha \Xi_\iota}<\infty|\F_{\min\ha\Xi_\iota}]\le C P_\iota,\quad \iota\in I.\label{P-iota}\EDE
	Let $P_e= \frac{P_{y_e}(r_e)}{P_{y_e}(R_e)}$, $e\in\cal E$. From Lemma \ref{Py} we get
	\BGE \prod_{j=0}^{M_e} P_{(e,j)}\le 4^{\alpha M_e} P_e,\quad e\in\cal E.\label{P-iota<}\EDE
	
	We have $|I|=\sum_{e\in\cal E}(M_e+1)$. Considering the order that $\gamma$ visits $\ha\Xi_\iota$, $\iota\in I$, we get a bijection map $\sigma_I:\N_{|I|}\to  I$ such that $n_1<n_2$ implies that  $\max\sigma^{-1}(\ha\Xi_{\sigma_I(n_1)})\le\min\sigma^{-1}(\ha\Xi_{\sigma_I(n_2)})$, and $n_1=n_2-1$ implies that $\min\sigma^{-1}(\ha\Xi_{\sigma_I(n_2)})-\max\sigma^{-1}(\ha\Xi_{\sigma_I(n_1)})\in\{0,1\}$. The difference may take value $0$ if $\max\ha\Xi_{\sigma_I(n_1)}=\xi_e^*\not\in \Xi_e$ for  $e=e_{\sigma_I(n_1)}$. We may express $E^\sigma$ as
	$$E^\sigma=\{\tau_{\min\ha\Xi_{\sigma_I(1)}}\le\tau_{\max\ha\Xi_{\sigma_I(1)}}\le\tau_{\min\ha \Xi_{\sigma_I(2)}} \le \cdots\le \tau_{\min\ha \Xi_{\sigma_I(|I|)}}<\tau_{\max\ha \Xi_{\sigma_I(|I|)}}<\infty\}.$$
	
	Fix $e_0\in\cal E$. Let $n_j=\sigma_I^{-1}((e_0,j))$, $0\le j\le M_{e_0}$. Then $n_{j+1}\ge n_{j}+2$, $0\le j\le M_{e_0}-1$. Fix $0\le j\le M_{e_0}-1$. Let $m=n_{j+1}-n_{j}-1$. If $\max\ha \Xi_{\sigma_I(n_j+k)}$ and $\min\ha \Xi_{\sigma_I(n_j+k)}$ are concentric for $1\le k\le m$, applying Lemma \ref{key-lem} with $\ha J_0=\min\Xi_{e_0}$, $J_0=\max\ha\Xi_{({e_0},j)}=\max\ha\Xi_{\sigma_I(n_{j})}$, $J'_0= \min\ha\Xi_{({e_0},j+1)}=\min\ha\Xi_{\sigma_I(n_{j+1})}$, $\{|z-z_k|=R_k\}=\min\ha\Xi_{\sigma_I(n_{j}+k)}$ and $\{|z-z_k|=r_k\}=\max\ha\Xi_{\sigma_I(n_j+k)}$, $1\le k\le m$, we get
	\BGE \PP[E^\sigma_{[\max\ha\Xi_{\sigma_I(n_{j})},\min\ha\Xi_{\sigma_I(n_{j+1})}]}|\F_{\tau_{\max\ha\Xi_{\sigma_I(n_{j})}}}]\le C^m 4^{-\alpha/4(s_{e_0}(j+1)-1)} \prod_{n=n_{j}+1}^{n_{j+1}-1}P_{\sigma_I(n)},\label{E[]}\EDE
	where $E^\sigma_{[\max\ha\Xi_{\sigma_I(n_{j})},\min\ha\Xi_{\sigma_I(n_{j+1})}]}$ is the %$\F_{\tau_{\min\ha\Xi_{\sigma_I(n_{j+1})}}}$-measurable
	event
	$$\{\tau_{\max\ha\Xi_{\sigma_I(n_{j})}}\le \tau_{\min\ha\Xi_{\sigma_I(n_{j}+1)}}\le \tau_{\max\ha\Xi_{\sigma_I(n_{j}+1)}}\le \cdots \le  \tau_{\max\ha\Xi_{\sigma_I(n_{j}+m)}}\le \tau_{\min\ha\Xi_{\sigma_I(n_{j+1})}}<\infty\}.$$
	Because of the definition of $P_\iota$, $\iota\in I$, the above estimate still holds in the general case, i.e., there may be some $1\le k\le n$ such that $\max\ha \Xi_{\sigma_I(n_j+k)}=\xi_e^*\not\in \Xi_e$, where $e=e_{\sigma_I(n_j+k)}$.
	
	We say that $\gamma$ makes a $(J_1,J_2)$ jump during $[\max\ha\Xi_{\sigma_I(n_{j})},\min\ha\Xi_{\sigma_I(n_{j+1})}]$ if $\min\Xi_{e_0}$ is enclosed by $J_1$, and there is at least one $k_0\in\N_m$ such that $\min\ha\Xi_{\sigma_I(n_{j}+k_0)}$ is not enclosed by $J_2$. In this case, applying Lemma \ref{key-lem} with $J_0=J_1$ and $\ha J_0=J_2$, we get
	$$ \PP[E^\sigma_{[\max\ha\Xi_{\sigma_I(n_{j})},\min\ha\Xi_{\sigma_I(n_{j+1})}]}|\F_{\tau_{\max\ha\Xi_{\sigma_I(n_{j})}}}]\le C^m   e^{-\alpha\pi d_{\C}(J_1,\ha J_2)/2} \prod_{n=n_{j}+1}^{n_{j+1}-1}P_{\sigma_I(n)}.$$
	Combining this with (\ref{E[]}), we get
	\BGE \PP[E^\sigma_{[\max\ha\Xi_{\sigma_I(n_{j})},\min\ha\Xi_{\sigma_I(n_{j+1})}]}|\F_{\tau_{\max\ha\Xi_{\sigma_I(n_{j})}}}]\le C^m   e^{-\frac\alpha 4\pi d_{\C}(J_1,\ha J_2)} 4^{-\frac \alpha 8(s_{e_0}(j+1)-1)} \prod_{n=n_{j}+1}^{n_{j+1}-1}P_{\sigma_I(n)}.\label{jump}\EDE
	
	Letting $j$ vary between $0$ and $M_{e_0}-1$ and using (\ref{P-iota}) and (\ref{E[]}), we get
	$$\PP[E^\sigma]\le C^{|I|} 4^{-\alpha/4 \sum_{j=1}^{M_{e_0}} (s_{e_0}(j)-1)} \prod_{\iota\in I} P_\iota.$$
	Using (\ref{P-iota<}) and $|I|=\sum_e (M_e+1)$, we find that
	$$\PP[E^\sigma]\le C^{|{\cal E}|} C^{\sum_{e\in\cal E}M_e}4^{-\frac\alpha 4 \sum_{j=1}^{M_{e_0}} s_{e_0}(j)}\prod_{e\in\cal E} P_e.$$
	
	Since $\sigma^{-1}(\xi_a)<\sigma^{-1}(\xi_c)<\sigma^{-1}(\xi_b)$, $\xi_a<\xi_b $ are  enclosed by $J_1$, and $\xi_c$ is not enclosed by $J_2$, there must exist some $e_0\in\cal E$ and some $j\in[0,M_{e_0}-1]$ such that $\gamma$ makes a $(J_1,J_2)$ jump during $[\max\ha\Xi_{\sigma_I(n_{j})},\min\ha\Xi_{\sigma_I(n_{j+1})}]$. In that case,  using (\ref{P-iota}), (\ref{E[]}), and (\ref{jump}), we get
	$$\PP[E^\sigma]\le C^{|{\cal E}|} C^{\sum_{e\in\cal E}M_e}e^{-\frac \alpha4\pi d_{\C}(J_1,\ha J_2)} 4^{-\frac\alpha 8 \sum_{j=1}^{M_{e_0}} s_{e_0}(j)}\prod_{e\in\cal E} P_e.$$
	Taking a geometric average of the above upper bounds for $\PP[E^\sigma]$ over $e_0\in\cal E$, we get
	\BGE \PP[E^\sigma]\le   C^{|{\cal E}|} C^{\sum_{e\in\cal E}M_e}e^{-\frac{\alpha}{4|\cal E|}\pi d_{\C}(J_1,\ha J_2)} 4^{-\frac\alpha{8|{\cal E}|}\sum_{e\in\cal E} \sum_{j=1}^{M_{e}} s_{e}(j)}\prod_{e\in\cal E} P_e.\label{E-sigma-est}\EDE
	
	So far we have omitted the $\sigma$ on $I$, $M_e$, $s_e(j)$ and etc; we will put $\sigma$ on the superscript if we want to emphasize the dependence on $\sigma$. From (\ref{E-sigma}) and (\ref{E-sigma-est}), we get
	\BGE \PP[E]\le C^{|{\cal E}|} \sum_{(M_e;(s_e(j))_{j=0}^{M_e})_{e\in\cal E}}|S_{(M_e,(s_e(j)))}|  C^{\sum_{e\in\cal E}M_e}e^{-\frac{\alpha}{4|\cal E|}\pi d_{\C}(J_1,\ha J_2)}  4^{-\frac\alpha{8|{\cal E}|}\sum_{e\in\cal E} \sum_{j=1}^{M_{e}} s_{e}(j)}\prod_{e\in\cal E} P_e,\label{PE-S}\EDE
	where
	$$S_{(M_e,(s_e(j)))}:=\{\sigma\in S:M^\sigma_e=M_e, s^\sigma_e(j)=s_e(j), 0\le j\le M_e,e\in\cal M\},$$
	and the first summation in (\ref{PE-S}) is over all possible $(M_e;(s_e(j))_{j=0}^{M_e})_{e\in\cal E}$, namely, $M_e\ge 0$ and  $0=s_e(0)<s_e(1)<\cdots s_e(M_e)\le N_e$ for every $e\in\cal E$. It now suffices to show that
	\BGE \sum_{(M_e;(s_e(j))_{j=1}^{M_e})_{e\in\cal E}}|S_{(M_e,(s_e(j)))}|  C^{\sum_{e\in\cal E}M_e}4^{-\frac\alpha{8|{\cal E}|}\sum_{e\in\cal E} \sum_{j=1}^{M_{e}} s_{e}(j)} \le C_{|\cal E|},\label{suffice}\EDE
	for some $C_{|\cal E|}<\infty$ depending only on $|\cal E|$ and $\kappa$.
	
	We now bound the size of $S_{(M_e,(s_e(j)))}$. Note that when $M^\sigma_e$ and $s^\sigma_e(j)$, $0\le j\le M^\sigma_e$, $e\in\cal E$, are given, $\sigma$ is then determined by $\sigma_I:\N_{|I^\sigma|}\to  I^\sigma$, which is in turn determined by $e_{\sigma_I(n)}$, $1\le n\le |I^\sigma|=\sum_{e\in\cal E}(M^\sigma_e+1)$.
	%, because if $e_{\sigma_I(n)}=e_0$, then $\sigma_I(n)=(e_0,j_0)$, where $j_0=\min\{0\le j\le M_{e_0}: (e_0,j)\not\in \sigma_I(m), m<n\}$.
	Since each $e_{\sigma_I(n)}$ has at most $|\cal E|$ possibilities, we have $|S_{(M_e,(s_e(j)))}|\le |{\cal E}|^{\sum_{e\in\cal E}(M_e+1)}$.
	Thus, the left-hand side of (\ref{suffice}) is bounded by
	\begin{align*}
		&|{\cal E}|^{|{\cal E}|} \sum_{(M_e;(s_e(j))_{j=0}^{M_e})_{e\in\cal E}} \prod_{e\in\cal E}   (C|{\cal E}|)^{ M_e}4^{-\frac\alpha{8|{\cal E}|}  \sum_{j=1}^{M_{e}} s_{e}(j)}\\
		=&|{\cal E}|^{|{\cal E}|} \prod_{e\in\cal E} \sum_{M_e=0}^{N_e}  (C|{\cal E}|)^{ M_e} \sum_{0=s_e(0)<\cdots<s_e(M_e)\le N_e}4^{-\frac\alpha{8|{\cal E}|}  \sum_{j=1}^{M_{e}} s_{e}(j)}\\
		\le & |{\cal E}|^{|{\cal E}|} \prod_{e\in\cal E} \sum_{M=0}^{\infty}   (C|{\cal E}|)^{ M} \sum_{s(1)=1}^\infty\cdots  \sum_{s(M)=M}^\infty
		4^{-\frac\alpha{8|{\cal E}|}  \sum_{j=1}^{M} s(j)}\\
		\le &  |{\cal E}|^{|{\cal E}|} \prod_{e\in\cal E} \sum_{M=0}^{\infty}   (C|{\cal E}|)^{ M}\prod_{j=1}^M \sum_{s(j)=j}^\infty
		4^{-\frac\alpha{8|{\cal E}|}  s(j)}\\
		=  &\left[|{\cal E}| \sum_{M=0}^{\infty}   \left(\frac{C|{\cal E}|}{1-4^{-\frac\alpha{8|{\cal E}|} }}\right)^{ M} 4^{-\frac\alpha{16|{\cal E}|}  M(M+1)} \right]^{|{\cal E}|}.
	\end{align*}
	The conclusion now follows since the summation inside the square bracket equals to a finite number depending only on $\kappa$ and $|\cal E|$.
\end{proof}

\begin{proof}[Proof of Theorem \ref{RZ-Thm3.1}]
	By (\ref{Frasymp}), we may change the order of the points $z_1,\dots,z_n$. Thus, it suffices to show that
	\BGE \PP[\tau^{z_j}_{r_j}<\infty,1\le j\le n; \tau^{z_{1}}_{s_{1}}<\tau^{z_{2}}_{r_{2}}<\tau^{z_1}_{r_1}] \le C_n   \prod_{j=1}^n \frac{P_{y_j}(r_j)}{P_{y_j}(l_j)}\cdot\Big(\frac{s_{1}}{|z_{1}-z_{2}|\wedge |z_{1}|}\Big)^{\frac{\alpha}{32n^2}}, \label{perm-Thm}\EDE
	for any distinct points $z_1,\dots,z_n\in\lin\HH\sem\{0\}$, $r_j\in(0,d_j)$, $1\le j\le n$,  and $s_1\ge 0$,
	%$s_1\in (r_1,|z_{1}-z_{2}|\wedge |z_{1}|)$,
	where $y_j,l_j,d_j$ are defined by (\ref{ldyR}). If $s_1\le r_1$, the event on the LHS is empty, and the formula trivially holds; if $s\ge |z_{1}-z_{2}|\wedge |z_{1}|$, the formula follows from \cite[Theorem 1.1]{RZ}. For the rest of the proof, we assume that $s_1\in (r_1,|z_{1}-z_{2}|\wedge |z_{1}|)$.
	
	We want to deduce the theorem from Lemma \ref{key-lem2}, so we want to construct a family  $\Xi$ of  mutually disjoint circles and Jordan curves $J_1,J_2$.
	
	Suppose $4^{h_j} r_j\le l_j\le 4^{h_j+1} r_j$ for some $h_j\in\N$, $1\le j\le n$.  By increasing the value of $s_1$, we may assume that $s_1=4^{\til h_1} r_1$, where $\til h_1\in\N$ and $\til h_1>h_1$. Define
	\[
	\xi_j^s=\{|z-z_j|=4^{h_j-s} r_j\}, \quad 1\le j\le n,\quad 1\le s\le h_j.
	\]
	The family $\{\xi_j^s:1\le j\le n,\quad 1\le s\le h_j\}$ may not be mutually disjoint. So we can not define $\Xi$ to be this family. To solve this issue, we will remove some circles as follows. For $1\le j<k\le n$, let $D_k=\{|z-z_k|\le l_k/4\}$, which contains every $\xi_k^r$, $1\le r\le h_k$, and
	\BGE I_{j,k}=\{\xi_j^s: 1\le s\le h_j, \xi_j^s\cap D_k\ne\emptyset\}.\label{I}\EDE
	Then $\Xi:=\{\xi_j^s:1\le j\le n, 1\le s\le h_j\}\sem \bigcup_{1\le j<k\le n} I_{j,k}$ is mutually disjoint.
	If $\dist(\gamma,z_j)\le r_j$, then $\gamma$ intersects every $\xi_j^s$, $1\le s\le h_j$. So we get
	\BGE \PP[\dist(\gamma,z_j)\le r_j,1\le j\le n]\le \PP\Big[\bigcap_{j=1}^n \bigcap_{s=1}^{h_j}\{\gamma\cap \xi_j^s\ne\emptyset\}\Big]
	\le \PP\Big[\bigcap_{\xi\in\Xi} \{\gamma\cap\xi\ne\emptyset\}\Big].\label{dist-xi}\EDE
	
	Next, we construct a partition $\{\Xi_e:e\in\cal E\}$ of $\Xi$.
	We introduce some notation: if $e$ is a family of  circles centered at $z_0\in\lin\HH$ with biggest radius $R$ and smallest radius $r$, then we define $A_e=\{r\le |z-z_0|\le R\}$ and $P_e=\frac{P_{\Imm z_0}(r)}{P_{\Imm z_0}(R)}$.
	
	First, $\Xi$ has a natural partition $\Xi_j$, $1\le j\le n$, such that $\Xi_j$ is composed of circles centered at $z_j$. For each $j$, we construct a graph $G_j$, whose vertex set is $\Xi_j$, and $\xi_1\ne\xi_2\in \Xi_j$ are connected by an edge iff the bigger radius is $4$ times the smaller one, and the open annulus between them does not contain any other circle in $\Xi$. Let ${\cal E}_j$ denote the set of connected components of $G_j$. Then we partition $\Xi_j$ into $\Xi_e$, $e\in {\cal E}_j$, such that every $\Xi_e$ is the vertex set of $e\in{\cal E}_j$. Then the circles in every $\Xi_e$ are concentric circles with radii forming a geometric sequence with common ratio $1/4$, and the closed annuli $A_e$ associated with $\Xi_e$, $e\in{\cal E}_j$, are mutually disjoint. From the construction we also see that for any $j<k$, and $e\in{\cal E}_j$, $A_e$ does not intersect $D_k$, which contains every $A_e$ with $e\in{\cal E}_k$. Let ${\cal E}=\bigcup_{j=1}^n {\cal E}_j$. Then $A_e$, $e\in\cal E$, are mutually disjoint. Thus, $\{\Xi_e:e\in\cal E\}$ is a partition of $\Xi$ that satisfies the properties before Lemma \ref{key-lem2}.
	
	We observe that for $j<k$, $\bigcup_{\xi\in\Xi_k}\xi \subset D_k$ can be covered by an annulus centered at $z_j$ with ratio less than $4$ because
	$$\frac{\max_{z\in D_k}\{|z-z_j|\}}{\min_{z\in D_k}\{|z-z_j|\}}\le \frac{|z_j-z_k|+l_k/4}{|z_j-z_k|-l_k/4}\le \frac{l_k+l_k/4}{l_k-l_k/4}<4.$$
	Thus, every $I_{j,k}$ defined in (\ref{I}) contains at most one element. We also see that, for $j<k$, $\bigcup_{\xi\in\Xi_k}\xi \subset D_k$ intersects at most $2$ annuli from $\{4^{h_j-s} r_j\le |z-z_j|\le 4^{h_j-s+1} r_j\}$, $2\le s\le h_j$. If $j>k$, by construction, $\bigcup_{\xi\in\Xi_k}\xi$ is disjoint from the annuli $\{4^{h_j-s} r_j\le |z-z_j|\le 4^{h_j-s+1} r_j\}$, $2\le s\le h_j$, which are contained in $D_j$.
	
	From \cite[Theorem 1.1]{RZ}, we have $\PP[\tau^{z_j}_{r_j}<\infty,1\le j\le n]\le C_n \prod_{j=1}^n \frac{P_{y_j}(r_j)}{P_{y_j}(l_j)}$. So we may assume that $|z_2-z_1|\wedge |z_1|>4^{4n+1} s_1$. Since for $k\ge 2$, $\bigcup_{\xi\in\Xi_k}\xi \subset D_k$ can be covered by an annulus centered at $z_1$ with ratio less than $4$, by pigeon hole principle, we can find a closed annulus centered at $z_1$ with two radii $r<R$ satisfying $s_1\le r<R\le |z_2-z_1|\wedge |z_1|$ and $R/r\le (\frac{|z_2-z_1|\wedge |z_1|}{s_1})^{1/2n}$ that is disjoint from all $\bigcup_{\xi\in\Xi_k}\xi \subset D_k$, $k\ge 2$. Moreover, we may choose $R$ and $r$ such that the boundary circles are disjoint from every $\xi\in\Xi$. Applying Lemma \ref{key-lem2} with
	$J_1=\{|z-z_1|=r\}$, $J_2=\{|z-z_1|=R\}$, $\xi_a=\{|z-z_1|=s_1\}$, $\xi_b=\{|z-z_1|=r_1\}$, $\xi_c=\{|z-z_2|=r_2\}$, and $\{\Xi_e:e\in\cal E\}$, we find that \BGE \PP[\tau^{z_j}_{r_j}<\infty,1\le j\le n; \tau^{z_{1}}_{s_{1}}<\tau^{z_{2}}_{r_{2}}<\tau^{z_1}_{r_1}]
	\le C_{|\cal E|}\Big(\frac{s_1}{|z_1-z_2|\wedge |z_1|}\Big)^{\frac{\alpha}{16n|\cal E|}} \prod_{j=1}^n \prod_{e\in {\cal E}_j} P_e.\label{dist-xi2}\EDE
	Here we set $\prod_{e\in{\cal E}_j}P_e=1$ if ${\cal E}_j=\emptyset$. We will finish the proof by proving that $|{\cal E}|\le 2n$ and  $\prod_{e\in {\cal E}} P_e\le C_n\frac{P_{y_j}(r_j)}{P_{y_j}(l_j)}$.
	
	We now bound $|{\cal E}|=\sum_{j=1}^n |{\cal E}_j|$. For $1\le m\le n$, we use ${\cal E}^{(m)}_j$, $1\le j\le m$, to denote the set of connected components of the graph $G^{(m)}_j$ obtained by  removing the circles in $I_{j,k}$, $j<k\le m$, from $\Xi_j$. Let ${\cal E}^{(m)}=\bigcup_{j=1}^m {\cal E}^{(m)}_j$. Then ${\cal E}={\cal E}^{(n)}$. For $2\le m\le n$, and $1\le j\le m-1$, we may define a map $f_{m}:\bigcup_{j=1}^{m-1}{\cal E}^{(m)}_j \to {\cal E}^{(m-1)}$ such that for every $e\in{\cal E}^{(m)}_j$, $1\le j\le m-1$, $f_{m}(e)$ is the unique element in ${\cal E}^{(m-1)}_j$ that contains $e$. Then each $e\in{\cal E}^{(m-1)}$ has at most $2$ preimages, and $e\in{\cal E}^{(m-1)}$ has exactly $2$ preimages iff $D_m$ is contained in the interior of $A_e$. Since the annuli $A_e$, $e\in {\cal E}^{(m-1)} $, are mutually disjoint, at most one of them has two preimages. Since  ${\cal E}^{(m)}_m$ contains only one element, we find that $|{\cal E}^{(m)}|\le |{\cal E}^{(m-1)}|+2$. From $|{\cal E}^{(1)}|=1$ and ${\cal E}={\cal E}^{(n)}$, we get $|{\cal E}|\le 2n-1$.
	
	To estimate $\prod_{e\in {\cal E}} P_e$, we introduce $S_j$ to be the family of pairs of circles $\{\{|z-z_j|=4^s r_j\},\{|z-z_j|=4^{s-1} r_j\}\}$, $s\in\N$. Let $S^{(m)}_j$ denote the set of $e'\in S_j$ such that $A_{e'}\subset \bigcup_{e\in {\cal E}^{(m)}_j} A_e$. Then $\prod_{e\in {\cal E}^{(m)}_j} P_e=\prod_{e'\in S^{(m)}_j} P_{e'}$. Note that, for $m>j$, $A_{e'}$, $e'\in S^{(m)}_j$ can be obtained from $A_{e'}$, $e'\in S^{(m-1)}_j$, by removing the annuli in the latter group that intersects $D_m$. Since $D_m$ can be covered by an annulus centered at $z_j$ with ratio less than $4$, it can intersect at most two of $A_{e'}$, $e'\in S_j$. Using Lemma \ref{Py}, we find that $\prod_{e\in {\cal E}^{(m)}_j} P_e\le 4^{2\alpha} \prod_{e\in {\cal E}^{(m-1)}_j} P_e$. Since $l_j\le 4^{h_j+1} r_j$, we get
	$\prod_{e\in {\cal E}^{(j)}_j} P_e=\frac{P_{y_j}(r_j)}{P_{y_j}(4^{h_j}r_j)}\le 4^\alpha \frac{P_{y_j}(r_j)}{P_{y_j}(l_j)}$. Thus, $\prod_{e\in {\cal E}^{(n)}_j} P_e\le 4^{\alpha(2n-2j+1)} \frac{P_{y_j}(r_j)}{P_{y_j}(l_j)}$, which implies that
	$$\prod_{e\in {\cal E}^{(n)}} P_e=\prod_{j=1}^n  \prod_{e\in {\cal E}^{(n)}_j} P_e \le \prod_{j=1}^n 4^{\alpha(2n-2j+1)}\frac{P_{y_j}(r_j)}{P_{y_j}(l_j)}=4^{\alpha n^2} \prod_{j=1}^n \frac{P_{y_j}(r_j)}{P_{y_j}(l_j)}.$$
	The proof is now complete.
\end{proof}
\end{appendices}


\begin{thebibliography}{00}
\bibitem{Ahl} L.\ V.\ Ahlfors (1973). {\it Conformal invariants: topics in geometric function theory}. McGraw-Hill Book Co., New York.
	
\bibitem{Albert-Kozdron} T.\ Alberts and M.\ Kozdron (2008).  Intersection probabilities for a chordal SLE path and a semicircle,
{\it Electron.\ Comm.\ Probab.} {\bf 13}, 448-460.

\bibitem{radial-Green} T.\ Alberts, M.\ Kozdron, and G.\ Lawler (2012). The Green function for the radial Schramm-Loewner evolution, {\it J.\ Phys.\ A} {\bf 45}, 494015.

\bibitem{Bf} V.\ Beffara (2008).  The dimension of SLE curves, {\it Annals of Probab.}
{\bf 36}, 1421-1452.

\bibitem{rate} C.\ Bene{\v s}, F.\ Johansson Viklund and M.\ Kozdron (2013). On the Rate of Convergence of Loop-Erased Random Walk to SLE$_2$. {\it Commun.\ Math.\ Phys} {\bf 318(2)}, 307-354.

\bibitem{FW} R.\ Friedrich and W.\ Werner (2003). Conformal restriction, highest-weight representations and SLE, {\it Commun.\ Math.\ Phys.} {\bf 243}, 105-122.

%\bibitem{FL} L.\ Field and G.\ Lawler (2015). Escape probability and transience for SLE, {\it Electron. J. Prob.} {\bf 20}, Article no. 10.
	
\bibitem{JJK} N.\ Jokela, M.\ J\"arvinen, and K.\ Kyt\"ol\"a (2015). SLE boundary visits, Arxiv preprint http://arxiv.org/abs/1311.2297.

\bibitem{Law1} G.\ Lawler (2005). {\em Conformally Invariant Processes
in the Plane}, Amer. Math. Soc.
	
\bibitem{Law4} G.\ Lawler (2009). Schramm-Loewner evolution, in {\em statistical mechanics}, S.Sheffield and T. Spencer, ed., IAS/Park City Mathematical Series, AMS (2009), 231-295.

\bibitem{Law2} G.\ Lawler (2013). Continuity of radial and two-sided radial SLE$_\kappa$ at the terminal point, {\it Contemporary Mathematics}
{\bf 590}, 101-124.
	
\bibitem{Law3} G.\ Lawler (2015). Minkowski content of the intersection of a Schramm-Loewner evolution (SLE) curve with the real
line,{\it J. Math. Soc. Japan.} {\bf 67}, 1631-1669.
	
\bibitem{LR1} G.\ Lawler and M.\ Rezaei (2015). Minkowski content and natural parameterization for the Schramm-Loewner
evolution,{\it Annals of Prob.} {\bf 43}, 1082-1120.
	
\bibitem{LR2} G.\ Lawler and M.\ Rezaei (2015). Up-to-constants bounds on the two-point Green's function for SLE curves,
{\it Electr. Comm. Prob.} {\bf 20}, Article no. 45.
	
\bibitem{LSW} G.\ Lawler, O.\ Schramm and W.\ Werner (2004). Conformal invariance
of planar loop-erased random walks and uniform spanning trees, {\it Annals of Prob.} {\bf 32(1B)}, 939-995.

\bibitem{LW} G.\ Lawler and B.\ Werness (2013). Multi-point Green's function for SLE and an estimate of Beffara, {\it Annals of Prob.} {\bf 41}, 1513-1555.
	
\bibitem{LZ} G.\ Lawler W.\ Zhou (2013). SLE curves and natural parametrization, {\it Annals of Prob.} {\bf 41}, 1556-1584.
	
\bibitem{RS} S.\ Rohde and O.\ Schramm (2005). Basic properties of
SLE, {\it Annals of Math.} {\bf 161}, 879-920.

\bibitem{RZ} M.\ Rezaei and D.\ Zhan (2017). Higher moments of the natural parameterization for SLE curves, {\it Ann. IHP.} {\bf 53}(1), 182-199.

\bibitem{Sch} O.\ Schramm (2000).
Scaling limits of loop-erased random walks
and uniform spanning trees, {\it Israel J. Math}. {\bf 118}, 221-288.

\bibitem{SW} O.\ Schramm and D.\ B.\ Wilson (2005). SLE coordinate changes, {\it New York Journal of Mathematics} {\bf 11}, 659-669.

\bibitem{LERW} D.\ Zhan (2008). The Scaling Limits of Planar LERW in Finitely Connected Domains. {\it Ann.\ Probab.} {\bf 36}, 467-529.	

	
\begin{comment}	

	\bibitem{Ahl} L.\ V.\ Ahlfors. {\it Conformal invariants: topics in geometric function theory}. McGraw-Hill Book Co., New York, 1973.
	
	
	
	\bibitem{dim-SLE} V.\ Beffara. The dimension of the SLE curves. {\it Ann. Probab.}, {\bf 36}(4):1421-1452, 2008.
	
	
	\bibitem{FL} L. Field and G. Lawler (2015). Escape probability and transience for SLE, Electron. J. Prob. {\bf 20}, Article no. 10.
	
	
	\bibitem{Law1} G. Lawler (2005). {\em Conformally Invariant Processes
		in the Plane}, Amer. Math. Soc.
	
	\bibitem{Law2} G. Lawler (2013). Continuity of radial and two-sided radial SLE$_\kappa$ at the terminal point, Contemporary Mathematics
	{\bf 590} , 101-124.
	
	\bibitem{Law3} G. Lawler (2015). Minkowski content of the intersection of a Schramm-Loewner evolution (SLE) curve with the real
	line, J. Math. Soc. Japan {\bf 67}, 1631-1669.
	
	
	\bibitem{LR1} G. Lawler and M. Rezaei (2015). Minkowski content and natural parameterization for the Schramm-Loewner
	evolution, Annals of Prob. {\bf 43}, 1082-1120.
	
	\bibitem{LR2} G. Lawler and M. Rezaei (2015). Up-to-constants bounds on the two-point Green's function for SLE curves,
	Electr. Comm. Prob. {\bf 20}, Article no. 45.
	
	\bibitem{LW} G. Lawler and B. Werness (2013). Multi-point Green's function for SLE and an estimate of Beffara, Annals of Prob. {\bf 41} .
	
	\bibitem{LZ} G. Lawler W. Zhou (2013). SLE curves and natural parametrization, Annals of Prob. {\bf 41}, 1556-1584.
	
	\bibitem{RZ} M. Rezaei and D. Zhan (2015). Higher moments of the natural parameterization for SLE curves, to appear Ann. IHP.
	
	
	
	\bibitem{LERW} Dapeng Zhan. The Scaling Limits of Planar LERW in Finitely Connected Domains. {\it Ann.\ Probab.}, 36(2):467-529, 2008.
	
	
	









Then because $|z_1| \geq 1$ we get $d_1 \leq \frac{1}{16}$. Here we consider 2 cases. First assume that $z_2 \not \in (L_{z,\frac{1}{8}} \cup \bar{C}_{4d_1})$ where $C_r=\{ w \in \Half| |w-z|=r$ and by $\bar{D}$ we mean the interior of $D$. Now consider $l=C_{2d_1} \cap L$. We should stop at the time that we hit $l$ and also we have good angle by above. Once we stop there we should map the whole things back. We know that we will stay in $L$ until $T_z$ which gurantee we stay away from $z_2$. Also know after mapping $z$ and $z_1$ are separated. So we should be able to go to $z$ then $z_1$ and then $z_2$. If $z_1$ is in $\bar{C}_{4d_1}$ then we should stop at the time that we hit $C_{8d_1}$ with a good angle and then this separates all the points as $z_2$ can not be much closer to $z$ compare to $z_1$ since $|z_2| \geq |z_1|$.


\end{comment}



\end{thebibliography}
\end{document}